%% file: panjer_calibration_new.tex
\documentclass[conference]{IEEEtran}
\IEEEoverridecommandlockouts
\usepackage{color}

\usepackage{enumerate}
\usepackage[shortlabels]{enumitem}

\usepackage[pdftex]{graphicx}
\usepackage{caption}
\usepackage{subcaption}

\usepackage{tikz}
\usetikzlibrary{spy, calc, arrows, shapes, positioning, intersections, decorations.markings, decorations.pathmorphing, matrix}
\usepackage{tikz-cd}

\usepackage{siunitx}

\usepackage{pgfplots}
\pgfplotsset{compat=1.3}

\usepackage{multirow}

\usepackage{dsfont}

\usepackage{cite}

\usepackage{multicol}

\usepackage{bm}

\usepackage{acronym}
\acrodef{phd}[PHD]{Probability Hypothesis Density}
\acrodef{cphd}[CPHD]{Cardinalized PHD}
\acrodef{gm}[GM]{Gaussian Mixture}
\acrodef{pgfl}[PGFL]{Probability Generating Functional}
\acrodef{iid}[i.i.d.]{independent and identically distributed}
\acrodef{wrt}[w.r.t.]{with respect to}
\acrodef{fisst}[FISST]{Finite Set Statistics}
\acrodef{epsrc}[EPSRC]{Engineering and Physical Sciences Research Council}
\acrodef{esric}[ESRIC]{Edinburgh Super-Resolution Imaging Consortium}
\acrodef{smc}[SMC]{Sequential Monte Carlo}
\acrodef{slam}[SLAM]{Simultaneous Localisation And Mapping}
\acrodef{rfs}[RFS]{Random Finite Set}
\acrodef{mc}[MC]{Monte Carlo}
\acrodef{ncv}[n.c.v.]{near constant velocity}

\usepackage[cmex10]{amsmath}
\makeatletter
\newcommand{\pushright}[1]{\ifmeasuring@#1\else\omit\hfill$\displaystyle#1$\fi\ignorespaces}
\newcommand{\pushleft}[1]{\ifmeasuring@#1\else\omit$\displaystyle#1$\hfill\fi\ignorespaces}

\makeatother

\usepackage[plain]{algorithm}
\usepackage{algpseudocode}

\usepackage{mathtools}

\usepackage{amssymb}
\usepackage{amsthm}
\newtheorem{thm}{Theorem}[section]

\newtheorem{prop}[thm]{Proposition}

\usepackage{multirow, multicol}

\def\b{\mathrm{b}}
\def\Bcal{\mathcal{B}}
\def\c{\mathrm{c}}

\def\d{\mathrm{d}}

\def\Fcal{\mathcal{F}}
\def\Gcal{\mathcal{G}}

\def\Nbb{\mathbb{N}}
\def\p{\mathrm{p}}

\def\Scal{\mathcal{S}}
\def\s{\mathrm{s}}
\def\var{\mathrm{var}}

\def\Xcal{\mathcal{X}}
\def\Xfk{\mathfrak{X}}
\def\Zcal{\mathcal{Z}}

\title{
Single-cluster PHD filter methods for joint multi-object filtering and parameter estimation
}

\author{\IEEEauthorblockN{Isabel Schlangen\IEEEauthorrefmark{1},
Daniel E. Clark,\IEEEauthorrefmark{1} and
Emmanuel Delande\IEEEauthorrefmark{1}}
\thanks{$^*$ School of Electrical and Physical Sciences,
Heriot-Watt University, Edinburgh EH14 4AS, UK. Email: \{is117, E.D.Delande, D.E.Clark\}@hw.ac.uk. Isabel Schlangen is supported by the Edinburgh Super-Resolution Imaging Consortium (MR/K01563X/1). This work was supported by the EPSRC Platform Grant (EP/J015180/1), and the MoD UDRC Phase 2 (EP/K014227/1).}}

\begin{document}

\maketitle

\begin{abstract}
Many multi-object estimation problems require additional estimation of model or sensor parameters that are either common to all objects or related to unknown characterisation of one or more sensors. Important examples of these include registration of multiple sensors, estimating clutter profiles, and robot localisation. Often these parameters are estimated separately to the multi-object estimation process, which can lead to systematic errors or overconfidence in the estimates. These parameters can be estimated jointly with the multi-object process based only on the sensor data using a single-cluster point process model. This paper presents novel results for joint parameter estimation and multi-object filtering based on a single-cluster second-order \ac{phd} and \ac{cphd} filter. Experiments provide a comparison between the discussed approaches using different likelihood functions.
\end{abstract}


\IEEEpeerreviewmaketitle

\acresetall
\section{Introduction}
\label{sec:intro}

Methods for detecting and estimating multiple targets from multiple sensors are fundamentally important for sensor fusion applications. Methods based on Finite Set Statistics~\cite{Mahler2007statistical} are a popular choice for developing solutions to multi-sensor fusion applications, including the sub-optimal solutions known as the \ac{phd} filter~\cite{Mahler2003Multitarget} that propagates the first-order moment of a point process, and the \ac{cphd} filter~\cite{Mahler2007PHD} that additionally propagates second-order information through its cardinality distribution. In particular, their Gaussian mixture~\cite{Vo2006Gaussian,Vo2007Analytic,Houssineau2010PHD} and sequential Monte Carlo implementations~\cite{Vo2005Sequential, Ristic2012Adaptive} have been applied for a wide range of applications like tracking maritime craft in live multi-sensor trials~\cite{Barr2013multi} or even tracking dolphin whistles~\cite{Gruden2016Automated}. A recent variation of the \ac{phd} filter, namely the second-order \ac{phd} filter~\cite{Schlangen2017second}, works with generalised model assumptions that allow to propagate the variance~\cite{Delande2014Regional} along with the first-order moment. 

There are many multi-object filtering problems that require additional estimation of parameters which are common to all objects or related to the sensor profile. Examples of these include registration of multiple sensors, estimating clutter profiles, and robot localisation with respect to its surroundings. Typically these parameters are estimated separately to the multi-object estimation process, which can lead to systematic errors or overconfidence in the estimates. This paper addresses specifically this kind of problems by estimating the parameters jointly with the multi-object process based only on the sensor data. Using the \ac{phd}, second-order \ac{phd} and \ac{cphd} filters as a basis for the multi-object estimation, the approach presented in this paper also estimates a parameter common to the multiple targets or related to the sensor configuration. The approach has been further studied for target tracking applications~\cite{Brekke2014novel} and more recently, it has been used for jointly triangulating multiple objects  and calibrating cameras~\cite{Houssineau2016unified} and estimating sensor drift in microscopes~\cite{Schlangen2016Marker} and telescopes~\cite{Hagen2016Joint}, and distributed multi-sensor localisation~\cite{Uney2015Cooperative}.

The approach adopted in this paper was initially developed for group and extended object tracking using hierarchical cluster point processes~\cite{Swain2010Extended, Swain2010First, Swain2012PHD, Clark2013Faa}, restricting the number of groups or extended objects to one~\cite{Swain2011Single}. This enables the modelling of multiple objects that are all conditioned on a single parameter, which could represent, for example, a vehicle position or sensor calibration and registration parameters. The method, known as the single-cluster \ac{phd} filter, has been applied to the problems of sensor calibration~\cite{Ristic2012Particle, Ristic2013Calibration, Ristic2012Calibration} and \ac{slam}~\cite{Lee2013SLAM}. The single-cluster PHD filter approach is distinct from the earlier random finite set method for SLAM developed by Mullane, Vo, Adams et al.~\cite{Mullane2011Random, Adams2014SLAM, Adams2013Circumventing} in that it considers the process to be estimated as a unified point process by making a single approximation and determining the vehicle location based on this approximation; therefore, there is no need to explicitly determine a set of landmarks to update the vehicle position. 

The single-cluster \ac{phd} filter, just like the \ac{phd} filter itself, relies on a Poisson assumption on the multi-target process. While this resulted in algorithms that are able to estimate the number of targets in the presence of false alarms, it has a restricted variance in the estimated target number. This paper develops more accurate and versatile estimators for single-cluster \ac{phd} filtering by considering different processes for modelling the target and clutter processes, in particular the \ac{iid} cluster process used in the \ac{cphd} filter, and binomial, Poisson, and negative binomial processes within the unified Panjer point process used in the recent second-order \ac{phd} filter.

The paper is structured as follows: In Sec.~\ref{sec:slam}, point process theory is quickly reviewed and the single-cluster point process is described with its two-level hierarchy. Section~\ref{sec:threefilters} describes three different single-cluster multi-object filters based on Poisson, Panjer, and \ac{iid} cluster point processes, with proofs given in the appendix. Section~\ref{sec:experiments} presents a detailed statistical analysis of the different methods in simulations, as well as a comparison with the most accurate method based on the  initial \ac{rfs} \ac{slam} formulation, eg.~\cite{Leung2014Evaluating,Leung2016Multifeature}. The paper concludes in section~\ref{sec:conclusion}, and the appendices present the pseudo-code and proofs for the likelihood functions.

\section{Joint multi-object filtering and parameter estimation}
\label{sec:slam}
This section aims to describe the joint multi-object state and parameter estimation in terms of a single-cluster point process framework as described in \cite{Swain2013Group} which carries a two-layer hierarchy. 
Note that in general, there might be interactions between the sensor and its environment, however this case is not considered in this article, i.e.~the dynamical models of the sensor and the targets are assumed to be independent.

\subsection{A short introduction to point processes and functionals}
Let $(\Omega, \Fcal, \mathbb{P})$ be a probability space with sample space $\Omega$, $\sigma$-algebra $\Fcal$, and probability measure $\mathbb{P}$. In the following, all random variables are defined on $(\Omega,\Fcal,\mathbb{P})$. 

A point process $\Phi$ on the state space $\Xcal$ is a random variable on the space $\Xfk = \bigcup_{n\geq 0}\Xcal^n$ of finite sequences of points in $\Xcal$. A realisation of $\Phi$ is a sequence ${\varphi = (x_1,\dots,x_n)\in \Xcal^n}$ which describes a population of $n$ targets with states $x_i \in \Xcal$. One way to describe a point process is by its probability distribution $P_\Phi$ on the measurable space $(\Xfk,\Bcal(\Xfk))$,  $\Bcal(\Xfk)$ being the Borel $\sigma$-algebra of $\Xfk$ \cite{Stoyan1997Stochastic}; the (symmetrical) projection measure $P^\Phi_n$ on the $n$-fold product space $\Xcal^n$ describes realisations with exactly $n$ elements for any natural number $n \geq 0$.  The projection measures facilitate the notation of the \ac{pgfl} of a point process $\Phi$ for any test function $h$:
\begin{equation}
\Gcal_\Phi(h) = \sum_{n\geq 0} \int \left[ \prod_{i=1}^n h(x_i) \right]P^{(n)}_\Phi(\d x_{1:n}). \label{eq:pgfl}
\end{equation}

\subsection{Single-cluster point processes}

Let $\Psi$ be a point process on the sensor state space $\Scal$ with probability distribution $P_\Psi$ describing the evolution of a single sensor and let $\Phi$ denote a point process on the target state space $\Xcal$ describing the evolution of the multi-target configuration. Their \ac{pgfl}s are given by
\begin{align}
G_\Psi(h) &= \int h(s) P_\Psi(s) \d s, \\
\Gcal_\Phi(g) &= \sum_{n\geq 0} \int \left[ \prod_{i=1}^n g(x_i) \right]P^{(n)}_\Phi(\d x_{1:n}).
\end{align}
In order to jointly describe the target process $\Phi$ and the sensor process $\Psi$, we can formulate the joint \ac{pgfl}
\begin{equation}
\begin{split}
&\Gcal_{\Phi,\Psi} (g,h) = G_\Psi(h G_\Phi(g|\cdot))\\
&= \int h(s) \left[\sum_{n\geq 0} \int \left[ \prod_{i=1}^n g(x_i|s) \right]P^{(n)}_\Phi(\d x_{1:n}|s) \right] P_\Psi(s) \d s.
\end{split}
\end{equation}
Therefore, it is of interest to propagate the probability distribution $P_\Psi$ of the sensor process over time, as well as the distribution $P_\Phi$ conditioned on the sensor state $s$. This structure induces a hierarchy on the processes, such that the sensor process $\Psi$ is also referred to as \emph{parent process} and the multi-object process $\Phi$ as \emph{daughter process}.
 
\subsubsection{The parent process}
The parent process estimates the time-varying sensor configuration, e.g.~the sensor position relative to its surroundings or other sensors, the clutter rate etc. The notation $\hat{\cdot}$ will refer to the parent process in this article.
The sensor configuration is assumed to evolve through some Markov transition function $\hat{t}_{k|k-1}$, and the {multi-object likelihood} $\hat{\ell}_k(s|Z)$ describes the sensor configuration based on the estimate of the respective daughter process. The  following Bayes recursion is used to propagate the law $\hat{P}_k$ of the sensor process at time $k$:
\begin{align}
\hat{P}_{k|k-1}(s) = \int \hat{t}_{k|k-1}(s|s')\hat{P}_{k-1}(s') \d s', \label{eq:parentpred}\\
\hat{P}_k(s|Z_k) = \frac{\hat{\ell}_k(Z_k|s)\hat{P}_{k|k-1}(s)}{\int \hat{\ell}_k(Z_k|s')\hat{P}_{k|k-1}(s') \d s'}. \label{eq:parentupdate}
\end{align}
\subsubsection{The daughter process}
The daughter processs estimates the time-varying multi-object configuration $\varphi\in \Xcal^{n_k}$. It is assumed to evolve through some Markov transition function $t_{k|k-1}$ and the multi-measurement/multi-target likelihood $\ell_k$ describes the association likelihood of targets and measurements based on the sensor state $s$. The following Bayes recursion is used to propagate the law $P_k$ of the target process at time $k$: 
\begin{align}
	P_{k|k-1}(\varphi|s) &= \int t_{k|k-1}(\varphi|\varphi',s) P_{k-1}(\varphi'|s)\d\varphi', \label{eq:multibayesfilter1}
	\\
	P_{k}(\varphi|s,Z_k) &= \frac{\ell_k(Z_k|\varphi,s)P_{k|k-1}(\varphi|s)}{\int \ell_k(Z_k|\varphi',s) P_{k|k-1}(\varphi')\d\varphi'}. \label{eq:multibayesfilter2}
\end{align}

In this article, multi-object tracking methods are considered for the daughter process that do not propagate the whole probability distribution but only the low-order moments of the latter which carry the most information, i.e.~the mean and possibly the variance of the process. Three different filters are chosen and described below, namely the first- and second-order \ac{phd} filters \cite{Mahler2003Multitarget,Schlangen2017second} and the \ac{cphd} filter \cite{Mahler2007PHD}. 

\section{Three multi-object filters and their single-cluster multi-object likelihoods}
\label{sec:threefilters}

For all three filters, let $p_{\s, k}(x)$ and $p_{\d, k}(x)$ denote the state-dependent probabilities of survival and detection at a given time $k$. Furthermore, $t_{k|k-1}$ stands for the Markov transition from time $k-1$ to time $k$, and $\ell_k(z|x)$ denotes the single-target association likelihood of measurement $z$ with target $x$ at time~$k$. The birth and clutter intensities at time $k$ will be denoted by $\mu_{\b,k}$ and $\mu_{\c,k}(z)$, respectively.

\subsection{The \ac{phd} filter \cite{Mahler2003Multitarget}}
\label{subsec:phd}

The \ac{phd} filter is the oldest and most widely used of the three multi-object estimation algorithms discussed in this article. It was first introduced in \cite{Mahler2003Multitarget}, and it is based on the assumption that the number of predicted targets, as well as the clutter cardinality, is Poisson distributed. The superscript $\flat$ will be used in the following to refer to the \ac{phd} filter. 

\begin{prop}[\ac{phd} recursion \cite{Mahler2003Multitarget}]
\begin{enumerate}[(a)]
\item The predicted first-order moment measure is given by
\begin{equation}
\label{eq:phdpred}
\mu_{k|k-1}^\flat(B) = \mu_{\b, k}(B) + \mu_{\mathrm{s},k}(B)
\end{equation}
with survival intensity
\begin{equation}
\label{eq:survivalint}
 \mu_{\mathrm{s},k}(B) = \int_\Xcal p_{\s, k}(x)t_{k|k-1}(B|x) \mu_{k-1}^\flat( \d x).
\end{equation}
\item The updated first-order moment measure with Poisson distributed prediction and clutter model is derived as
\begin{equation}
\label{eq:phdupdate}
\begin{split}
\mu_k^\flat(B) 
=&~\mu_{k}^\phi(B)+\sum_{z\in Z_k}  \frac{\mu_k^z(B)}{\mu_{\c,k}(z)+\mu_k^z(\Xcal)}
\end{split}
\end{equation}
with missed detection term 
\begin{equation}
\label{eq:mdterm}
\mu_{k}^\phi(B)= \int_B (1-p_{\d, k}(x)) \mu_{k|k-1}^\flat(\d x)
\end{equation}
and association term
\begin{equation}
\label{eq:assocterm}
\mu_k^z(B) =  \int_B p_{\d, k}(x) {\ell_k(z|x)} \mu_{k|k-1}^\flat(\d x)
\end{equation}
for any measurement $z \in Z_k$.
\end{enumerate}
\end{prop}
In context of \ac{slam} and parameter estimation, the \ac{phd} can be easily utilised for multi-object estimation using the following multi-object likelihood function \cite{Swain2013Group}:
\begin{thm}[\ac{phd} multi-object likelihood \cite{Swain2013Group}] 
The likelihood function of the \ac{phd} filter for a given sensor state $s$ is found to be
\begin{equation}
\label{eq:likelihood_phd}
\begin{split}
&\ell_k^\flat(s|Z) \\
&=\cfrac{ \prod_{z\in Z}\left[\mu_{\c,k}(z|s) +\int_\Xcal p_{\d,k}(x|s)\ell_k(z|x,s)\mu_{k|k-1}^\flat(\d x|s)\right]}{\exp\left[\int_\Zcal \mu_{\c,k}(z|s)\d z + \int_\Xcal p_{\d,k}(x|s)\mu_{k|k-1}^\flat(\d x|s) \right]}.
\end{split}
\end{equation}
\end{thm}

\subsection{The second-order \ac{phd} filter \cite{Schlangen2017second}}
\label{subsec:panjer}
The second-order \ac{phd} filter was introduced in \cite{Schlangen2017second} as an extension of the \ac{phd} filter, including second-order information by propagating the variance in the target number. It uses the assumption that the predicted target process and the clutter process are Panjer distributed as described in \cite{klugman2012loss,Schlangen2017second}, which generalises the Poisson distribution. It is therefore less restrictive, and it does not require the computationally expensive propagation of the whole cardinality distribution like in the \ac{cphd} filter. Instead, the Panjer distribution is completely characterised by two parameters which stand in direct correspondence with its mean and variance. This makes it possible to propagate both the mean and variance of the target cardinality in the second-order \ac{phd} recursion, leading to a very similar structure like the \ac{cphd} filter below. The notation $\natural$ will refer to the second-order \ac{phd} filter.

Recall the \emph{Pochhammer symbol} or \emph{rising factorial} $(\zeta)_n$ for any $\zeta\in \mathbb{R}$ and $n\in\mathbb{N}$:
\begin{equation}
\label{eq:pochhammer}
(\zeta)_n := \zeta(\zeta+1)\cdots (\zeta+n-1),\quad(\zeta)_0 :=1.
\end{equation}
Moreover, let $\alpha_{k|k-1},\beta_{k|k-1}$ and $\alpha_{\c,k},\beta_{\c,k}$ be the parameters of the predicted target and clutter processes at time $k$, respectively. Define the terms
\begin{equation}
\label{eq:ypsilon}
\begin{split}
Y_{u}(Z) := \sum_{j=0}^{|Z|} \frac{(\alpha_{k|k-1})_{j+u}}{(\beta_{k|k-1})^{j+u}} \frac{(\alpha_{\c, k})_{|Z|-j}}{(\beta_{\c, k} + 1)^{|Z|-j}}F_\d^{-j-u}e_{j}(Z)
\end{split}
\end{equation}
for any $Z \subseteq Z_k$, where $F_\d$ is the scalar given by
\begin{equation}
F_\d := \int \left[1 + \frac{p_{\d, k}({x})}{\beta_{k|k-1}}\right]\mu_{k|k-1}^\natural(\d{x}),
\end{equation}
and $e_j$ (cf.~Eq.~\eqref{eq:esf})
\begin{equation}
e_j(Z) := \sum_{\substack{{Z'}\subseteq Z\\|{Z'}|=j}}\prod_{z\in {Z'}}\frac{\mu_k^z(\Xcal)}{s_{\c, k}(z)},
\label{eq:esf2}
\end{equation}
where $s_{\c,k}$ denotes the spatial clutter distribution at time $k$ and the association term is defined as in \eqref{eq:assocterm} but using $\mu^\natural_{k|k-1}$ instead of $\mu^\flat_{k|k-1}$.
Furthermore, define the expression $l_d^\natural$ for $d = 1,2$ via
\begin{align}
l^\natural_d(\phi) :=  \cfrac{Y_d(Z_k)}{Y_0(Z_k)} \hspace{0.3cm}\text{and}\hspace{0.3cm}l^\natural_d(z) := \cfrac{Y_d(Z_k \backslash \{z\})}{Y_0(Z_k)}.
\end{align}
In a similar manner, define
\begin{equation}
l_2^{\natural}(z,z') := \left\{ \begin{array}{ll} 
 \cfrac{Y_2(Z_k \backslash \{z,z'\})}{Y_0(Z_k)} & \mathrm{if}\hspace{0.3cm}z\neq z',\\
 0 & \mathrm{otherwise}.
\end{array}\right.
\end{equation}
The prediction of the variance involves the second-order factorial moment  $\nu_{k}^{(2)}$ which, in general, cannot be retrieved from the predicted information $\mu_{k|k-1}^\natural,\var_{k|k-1}^\natural$ only. The assumption that $p_{\s,k}(x) = p_{\s,k}$ shall be uniform for all $x$ in the state space $\Xcal$, however, leads to the following closed-form recursion.

\begin{prop}[Second-order \ac{phd} recursion \cite{Schlangen2017second}]
\textcolor{white}{.}\\\vspace{-0.5cm}
\begin{enumerate}[(a)]
\item Assume that $p_{\s,k}(x) = p_{\s,k}$ is constant for all $x\in \Xcal$ at time $k$. In the manner of \eqref{eq:phdpred} and \eqref{eq:cphdpred}, the predicted first-order moment measure of the Panjer filter is given by
\begin{equation}
\label{eq:panjerpred}
\mu_{k|k-1}^\natural(B) = \mu_{\b, k}(B) + \mu_{\mathrm{s},k}(B),
\end{equation}
using $\mu_{k-1}^\natural$ instead of $\mu_{k-1}^\flat$ in \eqref{eq:survivalint}.
The predicted variance in the whole state space $\Xcal$ is given by
\begin{equation}
 \label{eq:varpred}
\var_{k|k-1}^\natural(\Xcal) = \var_{\b, k}(\Xcal) + \var_{\s, k}(\Xcal),
\end{equation}
where $\var_{\b, k}$ denotes the variance of the birth process and $\var_{\s, k}$ is the variance of the predicted process concerning the surviving targets which is found to be
\begin{equation}
\label{eq:var_s_assumption}
 \var_{\s, k}(\Xcal) = p_{\s,k}^2\var_{k-1}(\Xcal) + p_{\s,k}[1- p_{\s, k}]\mu_{k-1}(\Xcal).
\end{equation}
\item Find $\alpha_{k|k-1}$ and $\beta_{k|k-1}$ using
\begin{align}
\alpha_{k|k-1} &= \cfrac{\mu_{k|k-1}^\natural(\Xcal)^2}{\var_{k|k-1}^\natural(\Xcal)-\mu_{k|k-1}^\natural(\Xcal)},\label{eq:alpha}\\
\beta_{k|k-1} &= \cfrac{\mu_{k|k-1}^\natural(\Xcal)}{\var_{k|k-1}^\natural(\Xcal)-\mu_{k|k-1}^\natural(\Xcal)}.\label{eq:beta}
\end{align}
Then, the updated first-order moment measure becomes
\begin{equation}
\label{eq:panjerupdate}
\mu_k^\natural(B) =\hat{\mu}_{k}^\phi(B)\hat{l}_1(\phi)+\sum_{z\in Z_k}~\hat{\mu}_{k}^z(B) \hat{l}_1(z).
\end{equation}
The updated variance is obtained with
\begin{equation}
\label{eq:varupdate}
\begin{split}
&\var_k^\natural(B)=\mu_k^\natural(B)
+\hat{\mu}_k^\phi(B)^2\left[\hat{l}_2(\phi)-\hat{l}_1(\phi)^2\right] \\
&+2\hat{\mu}_k^\phi(B)\sum_{z\in Z_k}~\hat{\mu}_k^z(B)\left[\hat{l}_2(z)-\hat{l}_1(\phi)\hat{l}_1(z)\right]\\
&+\sum_{z,z'\in Z_k}\hat{\mu}_k^z(B)\hat{\mu}_k^{z'}(B)\left[ \hat{l}_2^{\neq}(z,z')-\hat{l}_1(z)\hat{l}_1(z')\right].
\end{split}
\end{equation}
\end{enumerate}
\end{prop}
\begin{thm}[Second-order \ac{phd} likelihood]
Write $\alpha = \alpha_{k|k-1}$ and $\beta = \beta_{k|k-1}$ for the sake of brevity, and let 
\begin{align}
\tilde{F}_{\d,s} &= 1-\frac{1}{\beta}\int_\Xcal p_{\d,k}(x|s) s_{k|k-1}(\d x, s),\\
F_\c &= 1+\frac{1}{\beta_\c}
\end{align}
for a given sensor state $s$. The multi-object likelihood function of the Panjer \ac{phd} filter for $s$ is found to be
\begin{equation}
\label{eq:likelihood_panjer}
\begin{split}
&\ell_k^\natural(s|Z) =\sum_{j=0}^{|Z|} \frac{(\alpha)_{j}}{(\beta)^{j}} \frac{(\alpha_{\c, k})_{|Z|-j}}{(\beta_{\c, k} + 1)^{|Z|-j}}\\
&\cdot\tilde{F}_{\d,s}^{-\alpha-j}F_\c^{-\alpha_\c-|Z|-j} \sum_{\substack{{Z'}\subseteq Z\\|{Z'}|=j}}\prod_{z\in {Z'}}{\hat{\mu}_k^z(\Xcal)} \prod_{z'\in (Z')^c}{\mu_{\c,k} (z|s)}.
\end{split}
\end{equation}
\end{thm}
\begin{proof}
See appendix.
\end{proof}

\subsection{The \ac{cphd} filter \cite{Mahler2007PHD}}
\label{subsec:cphd}
The \ac{cphd} filter was introduced by Mahler in \cite{Mahler2007PHD} after the need of a filter which propagates higher-order information was expressed in \cite{Erdinc2005Probability}. Instead of taking a particular assumption on the nature of the cardinality distribution, this filter estimates the latter alongside with the intensity of the point process. The notation $\sharp$ is used below to refer to the \ac{cphd} filter.

Let $\rho_{k}$ denote the cardinality distribution of the target population at time $k$, and let $\rho_\b$ and $\rho_\c$ denote the birth and clutter cardinality distributions, respectively. In comparison to \eqref{eq:phdupdate}, the \ac{cphd} filter update has additional factors $\l_d$ which depend on the cardinality distribution. First of all, recall the inner product
\begin{align}
\langle f , g \rangle &= \int f(x) g(x) \d (x) \quad \text{(continuous case)},\label{eq:innerproduct_c}\\
\langle f , g \rangle &= \sum_{n\geq 0} f(n) g(n) \hspace{0.81cm} \text{(discrete case)}. \label{eq:innerproduct_d}
\end{align}
Following the notation in \cite{Vo2007Analytic}, define the terms $\Upsilon^d[\mu,Z]$ via
\begin{equation}
\label{eq:Ycphd}
\begin{split}
\Upsilon^d[\mu,Z](n) &= \hspace{-0.2cm}\sum_{j=0}^{\min(|Z|,n-u)}\hspace{-0.2cm} \frac{n!(|Z|-j)!}{(n-(j+d))!}\\
& \hspace{0.5cm}\cdot\rho_\c(|Z|-j)\frac{\mu_{k}^\phi(\Xcal)^{n-(j+d)}}{\mu_{k|k-1}^\sharp(\Xcal)^n}e_j(Z),
\end{split}
\end{equation}
where 
\begin{equation}
\label{eq:esf}
e_j(Z) := \sum_{\substack{{Z'}\subseteq Z\\|{Z'}|=j}}\prod_{z\in {Z'}}\frac{\mu_k^z(\Xcal)}{\mu_{\c,k} (z)}.
\end{equation}
This leads to the terms
\begin{align}
l_1(\phi) &= \frac{\langle \Upsilon^1[\mu,Z], \rho_{k|k-1} \rangle}{\langle \Upsilon^0[\mu,Z], \rho_{k|k-1} \rangle},\label{eq:ltermphi}\\
l_1(z) &= \frac{\langle \Upsilon^1[\mu,Z\setminus\{z\}], \rho_{k|k-1} \rangle}{\langle \Upsilon^0[\mu,Z], \rho_{k|k-1} \rangle}. \label{eq:ltermz}
\end{align}
As mentioned before, the \ac{cphd} filter recursion involves both the intensity measure and the cardinality distribution. The notation below is inspired by \cite{Vo2007Analytic} and \cite{Delande2014Regional}.

\begin{prop}[\ac{cphd} recursion  \cite{Mahler2007PHD}]
\begin{enumerate}[(a)]
\item Similarly to \eqref{eq:phdpred}, the predicted first-order moment measure is given by
\begin{equation}
\label{eq:cphdpred}
\mu_{k|k-1}^\sharp(B) = \mu_{\b, k}(B) + \mu_{\mathrm{s},k}(B),
\end{equation}
using $\mu_{k-1}^\sharp$ instead of $\mu_{k-1}^\flat$ in \eqref{eq:survivalint}.
The predicted target cardinality distribution is found to be
\begin{equation}
\label{eq:cardpred}
\rho_{k|k-1}(n) = \sum_{j=0}^n \rho_\b(n-j) S[\mu_{k-1}^\sharp,\rho_{k-1}](j)
\end{equation}
for any $n \in \mathbb{N}$ with
\begin{equation}
S[\mu,\rho](j) =  \sum_{l=j}^\infty {l \choose j}\frac{\langle p_{\s,k},\mu \rangle^j \langle(1- p_{\s,k}),\mu \rangle^{l-j}}{\langle 1,\mu \rangle^l} \rho(l).
\end{equation}
\item The updated first-order moment measure is given by
\begin{equation}
\label{eq:cphdupdate}
\mu_k^\sharp(B) =\mu_{k}^\phi(B)l_1(\phi)+\sum_{z\in Z_k}~\frac{\mu_{k}^z(B)}{\mu_{\c,k}(z)} l_1(z)
\end{equation}
with missed detection term \eqref{eq:mdterm} and association term \eqref{eq:assocterm} using $\mu^\sharp_{k|k-1}$ instead of $\mu^\flat_{k|k-1}$. The updated target cardinality distribution is given by
\begin{equation}
\rho_k(n) = \frac{ \Upsilon^0[\mu_{k|k-1}^\sharp,Z](n) \rho_{k|k-1}(n)}{\langle \Upsilon^0[\mu_{k|k-1}^\sharp,Z], \rho_{k|k-1} \rangle}
\end{equation}
for any $n \in \mathbb{N}$.
\end{enumerate}
\end{prop}
\begin{thm}[\ac{cphd} multi-object likelihood] 
The multi-object likelihood function of the \ac{cphd} filter for a given sensor state $s$ is found to be
\begin{equation}
\label{eq:likelihood_cphd}
\ell_k^\sharp(s|Z) =\langle \tilde{\Upsilon}^0[\mu_{k|k-1}^\sharp,Z],\rho_{k,k-1}\rangle
\end{equation}
with 
\begin{equation}
\begin{split}
&\tilde{\Upsilon}^0[\mu_{k|k-1}^\sharp,Z](n) = \hspace{-0.2cm}\sum_{j=0}^{\min(|Z|,n-u)}\hspace{-0.2cm} \frac{n!(|Z|-j)!}{(n-(j+d))!}\rho_\c(|Z|-j)\\
& \cdot{\mu_{k}^\phi(\Xcal|s)^{n-(j+d)}} \sum_{\substack{{Z'}\subseteq Z\\|{Z'}|=j}}\prod_{z\in {Z'}}{\mu_k^z(\Xcal)} \prod_{z'\in (Z')^c}{\mu_{\c,k} (z)},
\end{split}
\end{equation}
where $(Z')^c = Z\setminus Z'$.
\end{thm}

\begin{proof}
See appendix.
\end{proof}

\subsection{An alternative likelihood function \cite{Leung2016Multifeature}}
For the sake of comparison, an alternative likelihood function is consulted which is introduced in \cite{Leung2016Multifeature}. This multi-object likelihood takes all possible associations between the measurement set and the target population into account. First of all, write the short-hand notation
\begin{equation}
p_\c(Z) = \frac{\prod_{z\in Z} s_\c(z|s_k)}{\exp\left( \int s_\c(z'|s_k) dz' \right)}.
\end{equation}
Furthermore, define the association function $\theta:\Bcal(\Zcal) \rightarrow \Bcal(\Xcal)$ which maps a selection of measurements $Z=\{z_1,\dots,z_{m_k}\}$ to a selection of targets $X=\{x_1,\dots,x_{n_k}\}$ via
\begin{equation}
\theta(z_j) = \begin{cases}
x_i & \text{if $z_j$ is associated with $x_i$,}\\
0 & \text{otherwise,}
\end{cases}
\end{equation}
where the $x_i$ are extracted from the predicted intensity $\mu_{k|k-1}^\bullet$ with $\bullet \in \{\flat,\sharp,\natural\}$.
Then,
\begin{equation}
\label{eq:likelihood_martin}
\begin{split}
\hat{\ell}_k(s|Z)&= p_\c(Z)\prod_{i=1}^{n_k}(1-p_\d(x_i|s))\\
&\quad\cdot\sum_{\theta}\prod_{\substack{j=1\\ \theta(j)\neq 0}}^{m_k}\frac{p_\d(\theta(z_j)|s) \ell(z_j|\theta(z_j),s)}{(1-p_\d(\theta(z_j|s)) p_\c(z_j)}.
\end{split}
\end{equation}
Note that in contrast to \eqref{eq:likelihood_phd}, the associations in the last term of \eqref{eq:likelihood_martin} are \emph{not} marginalised over all possible states, but this algorithm depends on the extraction of specific object locations $x_i$. In other words, this approach does not use the full information available through the predicted intensity.
\section{Simulations}
\label{sec:experiments}

All experiments presented in this section are performed using a Gaussian Mixture implementation of the algorithms described earlier, see \cite{Vo2006Gaussian,Vo2007Analytic,Schlangen2017second} for more detail. The calibration is conducted for all filters with their respective full multi-object likelihoods \eqref{eq:likelihood_phd}, \eqref{eq:likelihood_panjer} and \eqref{eq:likelihood_cphd}, henceforth globally denoted by L1, as well as with the likelihood \eqref{eq:likelihood_martin}, labeled by L2 in the following. The parent process is implemented in all cases with a \ac{smc} filter approach in the manner of \cite{Lee2012SLAM,Schlangen2016Marker} and others, using 300 \ac{mc} particles for each run and performing basic roulette resampling if the effective sample size falls below $150$. All results presented in the following are averaged over 50 \ac{mc} runs, and the sensor estimate as well as the estimated number of targets is computed as the weighted mean over all particles.

As \cite{Leung2016Multifeature} suggests, the likelihood \eqref{eq:likelihood_martin} needs additional assumptions to make it computationally feasible. Firstly, only associations are taken into account that lead to a single object - single association likelihood above a threshold $\tau_0 = 10^{-7}$ in the first time step and $\tau = 10^{-3}$ otherwise. Furthermore, a connected component analysis is performed to find groups of object-measurement clusters that are worth being associated, and these groups are restricted contain at most 3 measurements and 3 targets. Note that the three multi-object likelihood functions \eqref{eq:likelihood_phd}, \eqref{eq:likelihood_panjer} and \eqref{eq:likelihood_cphd} do not require any restrictions.

\subsection{Experiment 1}
For the first experiment, a global ground truth is simulated over 100 time steps (of unit \SI{1}{\second}) for both the multi-target configuration and the sensor trajectory, and the 50 \ac{mc} runs are performed on different measurement sets extracted from this ground truth. Both the target and the sensor state space are assumed four-dimensional, accounting for position and velocity in a two-dimensional environment measured in \si{\metre}. The targets follow a Poisson birth model with mean 4, and the objects move according to a \ac{ncv} model with acceleration noise of \SI{0.3}{\metre\per\square\second} and with an initial velocity of 0 with Gaussian noise of \SI{0.1}{\metre\per\second} in both x and y. Each target survives with a probability of $\p_\s = 0.95$. The sensor follows an \ac{ncv} model with acceleration noise \SI{0.2}{\metre\per\square\second} and initial velocity \SI{0}{\metre\per\second} in both x and y. The simulated sensor trajectory is depicted in Fig.~\ref{fig:drift_ex1}. 

From this scenario, measurements are generated with a uniform detection probability of $p_\d = 0.99$ over the whole state space. The observation space is assumed two-dimensional, accounting for the two dimensions of the environment that contains the objects. Measurements are superimposed with a measurement noise of \SI{0.1}{\metre} in both dimensions and with the created sensor drift, and false alarms are generated uniformly over the state space according to a Poisson noise model with mean $10$. 

In this scenario, the filter parameters are set to the same parameters that were used to generate the simulation. Fig.~\ref{fig:rmse_ex1} shows the root mean square error in the estimation of the sensor trajectory over all $100$ time steps for all filters. It can be seen that the three filters do not differ greatly among each other since the generated measurements fit the filter parameters. 
The full multi-object likelihoods associated to each filter, however, bring a much better estimation of the sensor state in comparison to the likelihood suggested by \cite{Leung2016Multifeature} which is consistently diverging after time step 10 for all three filters. In terms of the estimated number of targets (Fig.~\ref{fig:card_error_ex1_l1} and \ref{fig:card_error_ex1_l2}), all filters seem to monitor the ground truth consistently, apart from the second-order \ac{phd} filter with the alternative likelihood that slightly underestimates the number of targets in comparison to the other filters.  

Tab.~\ref{tab:runtimes_ex1} shows the averaged runtimes for the prediction, update and likelihood functions of all filters, averaged over 50 \ac{mc} runs. The \ac{cphd} filter update is up to 16 times slower than the updates of the \ac{phd} filters of first and second order. Moreover, the alternative likelihood function L2 also performs considerably slower than the full multi-object likelihood function.

\begin{table} 
\begin{center}
  \begin{tabular}{| l | c  c  c |}
    \hline
   filter & prediction & update & likelihood \\ \hline
   \ac{phd} L1 & 0.0022  &  0.2805  &  0.0182 \\ 
   \ac{phd} L2 & 0.0022  &  0.7205  &  0.5007 \\
   SO-\ac{phd} L1 & 0.0024  &  0.5914  &  0.0191 \\
   SO-\ac{phd} L2 & 0.0019  &  0.8681  &  0.4285 \\
   \ac{cphd} L1 & 0.2256  &  4.5857  &  0.0218 \\
   \ac{cphd} L2 & 0.2141  &  4.8107  &  0.4630 \\
    \hline
  \end{tabular}
\end{center}
\caption{Averaged runtimes in seconds, Experiment 1.\label{tab:runtimes_ex1}}
\end{table}

\begin{table}
\begin{center}
  \begin{tabular}{| r | l | c  c  c |}
    \hline
   & filter & prediction & update & likelihood \\ \hline
   \parbox[t]{2mm}{\multirow{6}{*}{\rotatebox[origin=c]{90}{14 deaths}}}&\ac{phd} L1 & 0.0014  &  0.0812  &  0.0033 \\ 
   &\ac{phd} L2 & 0.0014  &  0.0897  &  0.0173 \\
   &SO-\ac{phd} L1 & 0.0014  &  0.2003  &  0.0036 \\
   &SO-\ac{phd} L2 & 0.0013  &  0.1957  &  0.0153 \\
   &\ac{cphd} L1 & 0.3281  &  1.6374  &  0.0067 \\
   &\ac{cphd} L2 & 0.3079  &  1.5674  &  0.0166 \\
    \hline
    \parbox[t]{2mm}{\multirow{6}{*}{\rotatebox[origin=c]{90}{15 births}}}&\ac{phd} L1 & 0.0016  &  0.2323  &  0.0109 \\ 
   &\ac{phd} L2 & 0.0016  &  0.2664  &  0.0484 \\
   &SO-\ac{phd} L1 & 0.0017  &  0.7661  &  0.0117 \\
   &SO-\ac{phd} L2 & 0.0017  &  0.7389  &  0.0431 \\
   &\ac{cphd} L1 & 0.3185  &  5.1335  &  0.0143 \\
   &\ac{cphd} L2 & 0.3066  &  4.8654  &  0.0427 \\
    \hline
  \end{tabular}
\end{center}
\caption{Averaged runtimes in seconds, Experiment 2.\label{tab:runtimes_ex2}}
\end{table}

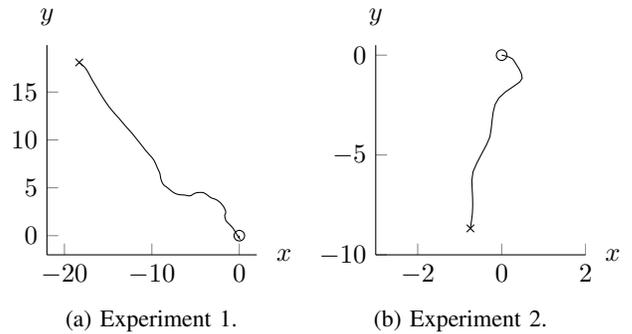
\begin{figure}[t] 
\centering
\begin{subfigure}[t]{0.45\linewidth}
\centering
\input{sensordrift.tikz}
\caption{Experiment 1.}
\label{fig:drift_ex1}
\end{subfigure}
\begin{subfigure}[t]{0.45\linewidth}
\centering
\input{truedrifts.tikz}
\caption{Experiment 2.}
\label{fig:ex2_truedrifts}
\end{subfigure}
\caption{Ground truth for the sensor state for Experiments 1 and 2, both originating at $(0,0)$.} 
\end{figure}

\begin{figure}[t] 
\centering
\begin{subfigure}[t]{0.9\linewidth}
\centering
\input{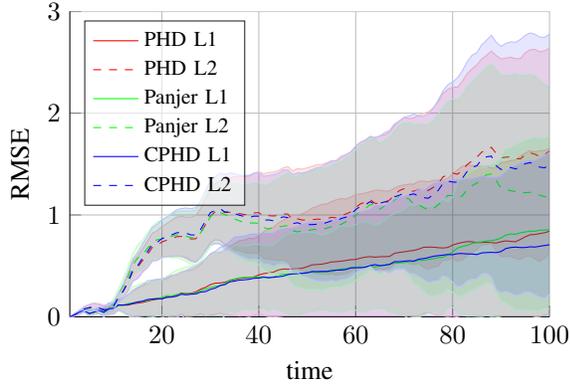}
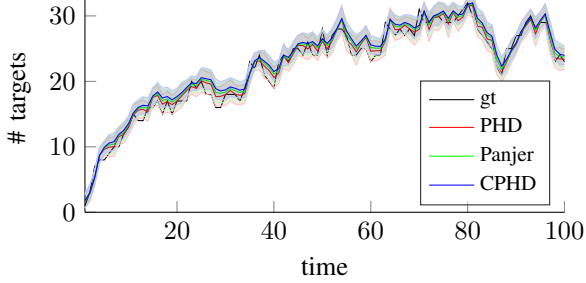
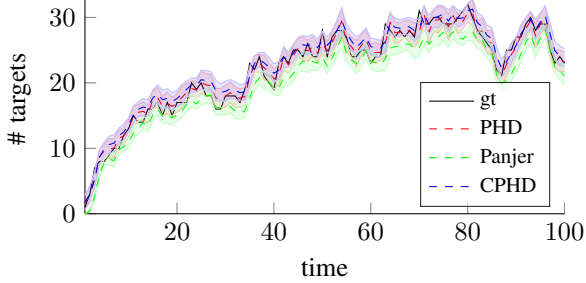
\caption{Error in the estimated sensor state.}
\label{fig:rmse_ex1}
\end{subfigure}\\
\begin{subfigure}[t]{0.9\linewidth}
\centering
\input{card_mean_L1.tikz}
\caption{Estimated number of targets, proposed likelihood (L1).}
\label{fig:card_error_ex1_l1}
\end{subfigure}\\
\begin{subfigure}[t]{0.9\linewidth}
\centering
\input{card_mean_L2.tikz}
\caption{Estimated number of targets, alternative likelihood (L2).}
\label{fig:card_error_ex1_l2}
\end{subfigure}
\caption{Results for Experiment 1 (no model mismatches).}
\label{fig:ex1}
\end{figure}

\subsection{Experiment 2}
The second experiment aims at analysing the effect of model mismatches on the robustness of the filters. While the filters follow the same target birth/death model as in Experiment 1, the ground truth does not: 15 targets are created at the initial step, and stay alive until time step $t = 15$; then, either $14$ objects are artificially removed (Experiment 2.1) or $15$ objects are added to the existing population (Experiment 2.2), such that the number of targets changes to $1$ or $30$, respectively, all of which stay alive until the end of the scenario. Little is known for the filters about the birth/death model, thus the second-order \ac{phd} and the \ac{cphd} filters are fed with a negative binomial birth with mean $2$ and variance $20$, accounting for a large uncertainty. The \ac{phd} filter, on the other hand, can only describe the number of newborn targets through its mean value, which is set to $2$ as well. The probability of survival in the three filters is set to $p_\s=0.99$. There are no model mismatches for the remaining parameters, whose values are set as in Experiment 1, except for a slightly smaller acceleration noise (\SI{0.1}{\metre\per\square\second}).

Fig.~\ref{fig:ex21} shows the estimation results for Experiment 2.1, with unexpected target death at time $t=15$. As for the first experiment, the proposed likelihood L1 leads to a significantly more accurate estimation of the sensor state. Regardless of the chosen filter or likelihood, the estimation error increases sharply after time $t=15$; it might be explained by the sudden target death which drastically decreases the amount of information for the estimation of the sensor state.  Fig.~\ref{fig:card_ex2_death_l1} and \ref{fig:card_ex2_death_l2} suggest that the \ac{cphd} filter is less reactive to unexpected target disappearances; this might be explained by its lower flexibility to out-of-model target deaths since it maintains a full cardinality distribution on the number of targets. On the other hand, it remains unclear why the second-order \ac{phd} filter slightly underestimates the number of targets if combined with the alternative likelihood L2, but shows accurate results with the proposed likelihood L1.

The results for Experiment 2.2 are displayed in Fig.~\ref{fig:ex22}. Again, the proposed method using likelihood L1 shows better performances in the estimation of the sensor state. Unlike Experiment 2.1, the unexpected increase in the number of targets provides more information and seems to facilitate the estimation of the sensor state, resulting in an improvement for all the filters and likelihoods immediately after time step $t = 15$. Again, the second-order \ac{phd} filter underestimates the number of targets when combined with the alternative likelihood L2. 

Tab.~\ref{tab:runtimes_ex2} confirms the findings on the runtime analysis of Experiment 1. The decrease in target number in Experiment 2.1 leads to a generally faster performance for all three filters, whereas the increase in target number in Experiment 2.2 results in longer runtimes. In the update, the \ac{cphd} filter again runs up to 20 times slower than the \ac{phd} filter, whereas the second-order \ac{phd} filter needs only twice as much time to run as the \ac{phd} filter. Furthermore, the proposed likelihood L1 is consistently faster than the alternative likelihood L2.

\begin{figure}[t] 
\centering
\begin{subfigure}[t]{0.9\linewidth}
\centering
\input{rmse_mean1.tikz}
\caption{Error in the estimated sensor state.}
\label{fig:rmse_ex2_death}
\end{subfigure}\\
\begin{subfigure}[t]{0.9\linewidth}
\centering
\input{card_mean1_likeli1.tikz}
\caption{Estimated number of targets, proposed likelihood (L1).}
\label{fig:card_ex2_death_l1}
\end{subfigure}\\
\begin{subfigure}[t]{0.9\linewidth}
\centering
\input{card_mean1_likeli2.tikz}
\caption{Estimated number of targets, alternative likelihood (L2).}
\label{fig:card_ex2_death_l2}
\end{subfigure}
\caption{Experiment 2.1 (out-of-model target deaths).}
\label{fig:ex21}
\end{figure}
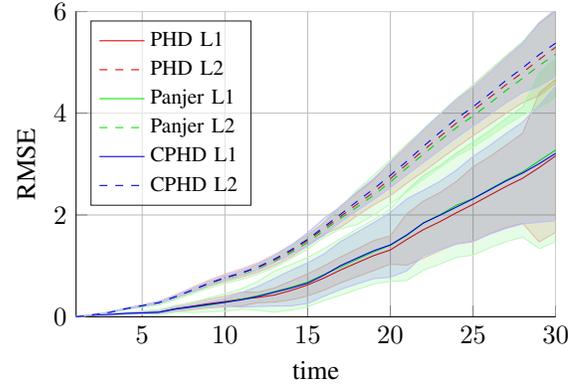
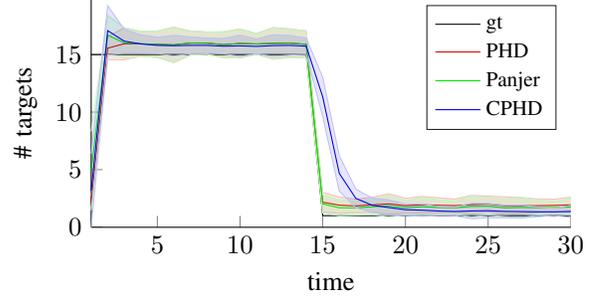
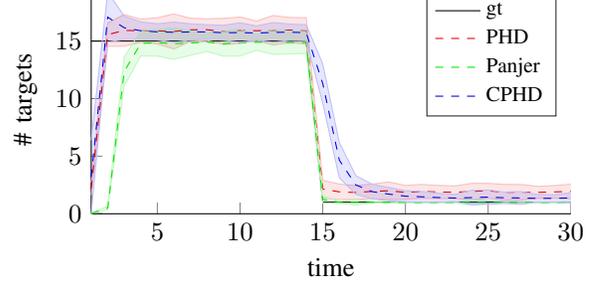

\begin{figure}[t] 
\centering
\begin{subfigure}[t]{0.9\linewidth}
\centering
\input{rmse_mean2.tikz}
\caption{Error in the estimated sensor state.}
\label{fig:rmse_ex2_birth}
\end{subfigure}\\
\begin{subfigure}[t]{0.9\linewidth}
\centering
\input{card_mean2_likeli1.tikz}
\caption{Estimated number of targets, proposed likelihood (L1).}
\label{fig:rmse_ex2_birth_l1}
\end{subfigure}\\
\begin{subfigure}[t]{0.9\linewidth}
\centering
\input{card_mean2_likeli2.tikz}
\caption{Estimated number of targets, alternative likelihood (L2).}
\label{fig:rmse_ex2_birth_l2}
\end{subfigure}
\caption{Experiment 2.2 (out-of-model target births).}
\label{fig:ex22}
\end{figure}
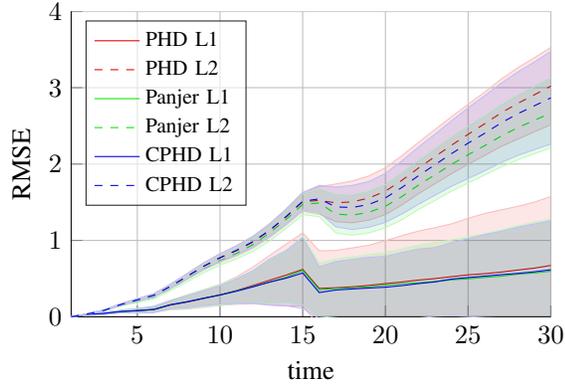
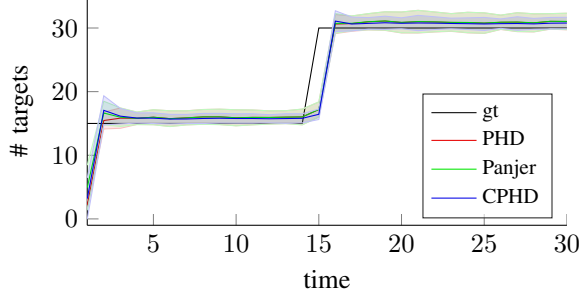
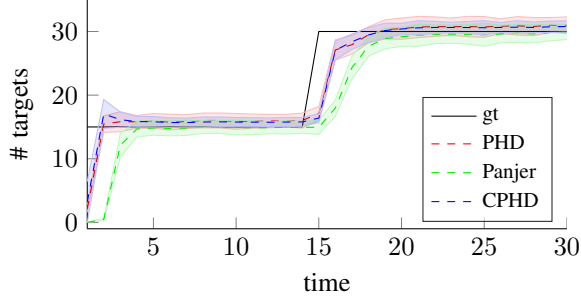

\section{Conclusion}
\label{sec:conclusion}

Following the idea of the single-cluster \ac{phd} filter, the second-order \ac{phd} filter and the \ac{cphd} filter have been embedded in a single-cluster environment to solve estimation problems where sensor-specific parameters have to be estimated alongside the multi-target configuration, e.g.~robot localisation, sensor calibration, or the estimation of the sensor profile. Specific multi-object likelihood functions for both filters have been derived in the same manner as it was previously done for the \ac{phd} filter, and simulations have shown that these likelihood functions yield better results in the sensor state estimation than existing multi-object likelihood approaches. All three filters perform well on the experimental setups studied in this article, even though the single-cluster \ac{cphd} filter is significantly slower than the other two filters.


\section*{Appendix A: Proofs for the multi-object likelihoods}

The appendix provides proofs for the likelihood functions \eqref{eq:likelihood_phd}, \eqref{eq:likelihood_cphd} and \eqref{eq:likelihood_panjer}. Therefore, let us define some basic concepts of \ac{pgfl}s and their differentiation.

\subsection{\ac{pgfl}s and functional differentiation}
\ac{pgfl}s are convenient tools for calculating moments and other statistical quantities using functional differentiation \cite{Stoyan1997Stochastic}. In the following, the so-called chain derivative will be used for this purpose as follows \cite{Bernhard2005Chain}. Let $(\varepsilon_n)_{n\in \mathbb{N}}$ be a sequence of real numbers converging to $0$ and $(\eta_n: \Xcal \rightarrow \mathbb{R}^+)_{n\in \mathbb{N}}$ be a series of functions converging pointwise to a function $\eta: \Xcal \rightarrow \mathbb{R}^+$.
For any given functional $G$ and test function $h: \Xcal \rightarrow \mathbb{R}^+$,  define the (chain) derivative of $G$ with respect to $h$ in the direction of $\eta$ with 
\begin{equation}
\label{eq:chaindiff}
\delta G(h;\eta) = \lim_{n\rightarrow\infty}\cfrac{G(h+\varepsilon_n \eta_n) - G(h)}{\varepsilon_n}.
\end{equation}
If the limit exists, it is identical for any two sequences $(\varepsilon_n)_{n\in \mathbb{N}}$ and $(\eta_n)_{n\in \mathbb{N}}$ as described above.

Since the multi-object likelihood $\ell_k(s|Z)$ is found by marginalising the observation process over the measurement set $Z$, the following differentiation will be of importance for the proofs.
\begin{equation}
P_\Phi^{(n)}(B_1\!\times \cdots \times\!B_n) = \delta^n\Gcal_\Phi(h;\mathds{1}_{B_1}, \ldots, \mathds{1}_{B_n})|_{h=0}; \label{eq:probaderivative}
\end{equation}
in particular,
\begin{equation}
P_\Phi^{(m)}\big((z_1, \dots, z_m)\big) = \delta^m\Gcal_\Phi(h;\delta_{z_1}, \ldots, \delta_{z_m})|_{h=0} \label{eq:probaderivative}
\end{equation}
with Dirac delta functions $\delta_{z_i}$; the notation $\delta G (f,\delta_x)$ will be understood as the Radon-Nikodym derivative of the measure $\mu' : B \mapsto \delta G (f,\mathds{1}_B)$ evaluated at point $x$, i.e. 
\begin{equation}
\delta G (f,\delta_x) := \dfrac{\mathrm{d} \mu' }{\mathrm{d} \lambda}(x),
\end{equation}
for any appropriate function $f$ on $\mathcal{X}$ and any point $x \in \mathcal{X}$ under the assumption that the theorem of Radon and Nikodym holds \cite{nikodym1930generalisation}.

Similarly to ordinary differentiation, the differential \eqref{eq:chaindiff} admits a product rule and a chain rule given by \cite{Bernhard2005Chain}
\begin{align}
\label{eq:productrule}
\delta(F\cdot G)(h;\eta) &= \delta F(h;\eta) G(h) + F(h) \delta G(h;\eta), \\
\label{eq:chainrule}
\delta(F\circ G)(h;\eta) &= \delta F(G(h);\delta G(h;\eta)),
\end{align}
respectively. These rules can be generalised to the $n$-fold product and chain rules \cite{Clark2013Faa, Clark2015few}
\begin{align}
\delta^n&(F\cdot G)(h;\eta_1,\dots,\eta_n) \nonumber \\ 
&= \! \sum_{\omega\subseteq \{1,\dots,n\}} \!\delta^{|\omega|} F\Big(h;(\eta_i)_{i \in \omega}\Big) \delta^{|\omega^{\c}|} G\Big(h;(\eta_j)_{j \in {\omega}^\c}\Big),\label{eq:generalproductrule}\\
\delta^n&(F\circ G)(h;\eta_1,\dots,\eta_n) \nonumber\\ 
&= \sum_{\pi \in \Pi_n} \delta^{|\pi|} F\left(G(h);\left(\delta^{|\omega|}G(h; (\eta_i)_{i \in \omega})\right)_{\omega \in \pi}\right).\label{eq:generalchainrule}
\end{align}

\subsection{Four examples of point processes}
Since the \ac{phd}, \ac{cphd} and second-order \ac{phd} filters utilise different special cases of point processes, four relevant examples will be listed below.
\subsubsection{Independent and identically distributed cluster process}
An \ac{iid} cluster process with cardinality distribution $\rho$ on $\Nbb$ and spatial distribution $s$ on $\Xcal$ describes a population whose size is described by $\rho$, and whose object states are \ac{iid} according to $s$. Its \ac{pgfl} is given by
\begin{equation}
\Gcal_\mathrm{i.i.d.}(h) = \sum_{n\geq 0} \rho(n) \left[\int h(x)s(\d x) \right]^n .\label{eq:iid}
\end{equation}
The \ac{cphd} filter is constructed using \ac{iid} cluster processes for the predicted target process as well as for the birth and clutter processes.

\subsubsection{Bernoulli process}
A Bernoulli point process with parameter $0 \leq p \leq 1$ and spatial distribution $s$ is an \ac{iid} cluster process with spatial distribution $s$ whose size is $1$ with probability $p$ and $0$ with probability $q = (1 - p)$. Its \ac{pgfl} is given by
\begin{equation}
\label{eq:bernoulli}
\Gcal_\mathrm{Bernoulli}(h)= q + p \int h(x)s(\d x).
\end{equation}
In the context of target tracking, Bernoulli processes are commonly used to describe binary events such as the detection or survival of individual targets.
\subsubsection{Poisson process}
A Poisson process with parameter $\lambda$ and spatial distribution $s$ is an \ac{iid} cluster process with spatial distribution $s$, whose size is Poisson distributed with rate $\lambda$. Its \ac{pgfl} is given by
\begin{equation}
\label{eq:poisson}
\Gcal_\mathrm{Poisson}(h) =  \exp \left(\int [h(x)-1]\mu(\d x)\right),
\end{equation}
where the intensity measure $\mu$ of the process  is such that $\mu(\d x) =  \lambda s(\d x)$. In the construction of the \ac{phd} filter, Poisson point processes are used for the predicted target and clutter processes.

\subsubsection{Panjer process}\label{subsubsec:panjerprocess}
A Panjer point process with parameters $\alpha$ and $\beta$ and spatial distribution $s$ is an \ac{iid} cluster process with spatial distribution $s$, whose size is Panjer distributed with parameters $\alpha$ and $\beta$ \cite{Fackler2009Panjer}. Its \ac{pgfl} is given by \cite{Schlangen2016PHD,Schlangen2017second}
\begin{equation}
\label{eq:panjer}
\Gcal_\mathrm{Panjer}(h) = \left( 1+\frac{1}{\beta} \int [1-h(x)]s(\d x)\right)^{-\alpha}.
\end{equation}

Each of the following proofs follows the same structure. First, the \ac{pgfl} of the observation process is found using the specific assumptions of the respective filter. In particular,
\begin{equation}
\label{eq:generalpgfl}
G_\text{obs}(g|s) = G_{k|k-1}(G_\d(g|\cdot,s)) G_\c(g|s) 
\end{equation}
where $s \in \Scal$. The functional $G_{k|k-1}$ denotes the \ac{pgfl} of the predicted process and $G_\c$ is the \ac{pgfl} of the clutter process; the detection process is always described by a Bernoulli point process with (possibly state-dependent) parameter $p_\d(\cdot)$ and single-object likelihood $l(\cdot|\cdot)$, thus its \ac{pgfl} takes the form \eqref{eq:bernoulli}.

\subsection{Proof of  \ac{phd} likelihood \eqref{eq:likelihood_phd}}
\begin{proof}[\unskip\nopunct]
The \ac{phd} likelihood was demonstrated in \cite{Swain2013Group} in the context of random finite sets, for the sake of completeness this section provides an alternative proof using point processes.
In case of the \ac{phd} filter, both $G_{k|k-1}$ and $G_\c$ are Poisson \ac{pgfl}s such that \eqref{eq:generalpgfl} takes the form
\begin{equation}
\label{eq:phd_pgfl}
\begin{split}
&G_\text{obs}^\flat(g|s) = \exp \Big( \int(g(z)-1)\mu_\c(z|s) \\
&+ \int\bigg(\bigg[1-p_\d(x|s)+  p_\d(x|s)\int g(z) \ell(z|x,s) \d z\bigg] -1\bigg) \\ &\pushright{\mu_{k|k-1}(x|s) \d x\Big)}.
\end{split}
\end{equation}
Differentiation by $g$ requires repeated applications of the chain rule \eqref{eq:chainrule} which pulls out one multiplicative term $\mu_\c(z)+\int p_\d(x|s)l(z|x,s)\mu_{k|k-1}(x|s) \d x$ for each measurement ${z \in Z_k}$. The final result is obtained by setting $g=0$. 
\end{proof}

\subsection{Proof of  \ac{cphd} likelihood \eqref{eq:likelihood_cphd}}
\begin{proof}[\unskip\nopunct]
The \ac{cphd} filter assumes the target and clutter processes to be \ac{iid} cluster processes. This results in the following variation of \eqref{eq:generalpgfl}:
\begin{equation}
\label{eq:cphd_pgfl}
\begin{split}
&G_\text{obs}^\sharp(g|s) = G^\sharp_\c(g|s) G_{k|k-1}^\sharp(G_\d(g|\cdot,s)) \\
&:=\left(\sum_{n\geq 0} \rho_\c(n) \left[\int g(z)s_\c( z|s)\d z \right]^n \right)\\
&\cdot\Bigg(\sum_{n\geq 0} \rho_{k|k-1}(n) \bigg[\int \bigg(1-p_\d(x|s) \\
&\pushright{\qquad+ p_\d(x|s)\int g(z)\ell(z|x,s)\d z \bigg)s_{k|k-1}^\sharp(x|s)\d x \bigg]^n\Bigg)}
\end{split}
\end{equation}
Since the product in \eqref{eq:cphd_pgfl} cannot be simplified like in \eqref{eq:phd_pgfl} above, it requires the use of the general product rule \eqref{eq:generalproductrule} leading to 
\begin{equation}
\label{eq:cphd_derivative}
\begin{split}
&\delta^m G_\text{obs}^\sharp(g;\delta_{z_1},\dots,\delta_{z_m}|s)|_{g=0} \\
&= \sum_{Z\subseteq Z_m}\Big(\delta^{|Z^c|}G^\sharp_\c(g;(\delta_{z})_{z\in Z^c}|s)|_{g=0} \\
&\qquad \qquad \qquad\cdot\delta^{|Z|}G_{k|k-1}^\sharp(G_\d(g;(\delta_{z})_{z\in Z}|\cdot,s))|_{g=0} \Big),\\
\end{split}
\end{equation}
where $Z^c = Z_m\setminus Z$. 
The first term in \eqref{eq:cphd_derivative} creates the product $\prod_{z\in Z_m\setminus Z} s_\c(z|s) \rho_\c(Z^c)$ for all $Z\subseteq Z_k$ since setting $g=0$ eliminates almost every term of the \ac{iid} cluster process. The second term is evaluated using the chain rule \eqref{eq:generalchainrule} which leads to the sum
\begin{equation}
\begin{split}
&\delta^{|Z|}G_{k|k-1}^\sharp(G_\d(g;(\delta_{z})_{z\in Z}|\cdot,s))|_{g=0}\\
&= \sum_{n\geq |Z|} \frac{n!\rho_{k|k-1}(n)}{(n-|Z|)!}\prod_{z\in Z}\int p_\d(x)\ell(z|x,s)s_{k|k-1}^\sharp(x|s) \d x  \\
&\quad\cdot\left(\int(1-p_\d(x))s_{k|k-1}^\sharp(x)\d x\right)^{n-|Z|}.
\end{split}
\end{equation}
The result is obtained by switching summations and rearranging the terms. 
\end{proof}

\subsection{Proof of  second-order \ac{phd} likelihood \eqref{eq:likelihood_panjer}}
\begin{proof}[\unskip\nopunct]
The Panjer assumption of the second-order \ac{phd} filter leads to the \ac{pgfl}
\begin{equation}
\label{eq:panjer_derivative}
\begin{split}
&G_\text{obs}^\natural(g|s) = G^\natural_\c(g|s) G_{k|k-1}^\natural(G_\d(g|\cdot,s)) \\
&= \bigg(1+ \frac{1}{\beta_{k|k-1}}\int\Big[1-\Big(1-p_\d(x|s) \\
&\qquad+ p_\d(x|s) \int \ell(z |x,s) g(z)\d z\Big)\Big]s_{k|k-1}^\natural(x|s)\d x\bigg)^{-\alpha_{k|k-1}}\\
&\quad \cdot \left(1+\frac{1}{\beta_\c} \int (1-g(z))s_\c(z|s)\d z \right)^{-\alpha_\c}.
\end{split}
\end{equation}
The general product rule \eqref{eq:generalproductrule} is needed for the product in \eqref{eq:panjer_derivative} in the same manner as \eqref{eq:cphd_derivative}, replacing $\sharp$ by $\natural$. Using the notations  $G_{k|k-1}^\natural(g|s) := F_\d(g|s)^{-\alpha_{k|k-1}}$ and $G_\c^\natural(g|s) := F_\c(g|s)^{-\alpha_\c}$, the corresponding derivatives are
\begin{equation}
\label{eq:panjer_assocterm}
\begin{split}
&\delta^{|Z|}G_{k|k-1}^\natural(g;(\delta_z)_{z\in Z}|s) = F_\d(g|s)^{-\alpha_{k|k-1}-|Z|} \\
&\cdot\frac{(\alpha_{k|k-1})_{|Z|}}{\beta_{k|k-1}^{|Z|}}\prod_{z\in Z}\int p_\d(x|s)\ell(z|x,s)s_{k|k-1}(x|s) \d x
\end{split}
\end{equation}
and
\begin{equation}
\label{eq:panjer_clutterterm}
\begin{split}
&\delta^{|Z^c|}G_\c^\natural(g;(\delta_z)_{z\in Z^c}|s)\\
&\quad = F_\d(g|s)^{-\alpha_\c-|Z^c|}\frac{(\alpha_\c)_{|Z|}}{\beta_\c^{|Z|}}\prod_{z\in Z^c}s_\c(z).
\end{split}
\end{equation}
Including \eqref{eq:panjer_assocterm} and \eqref{eq:panjer_clutterterm} into \eqref{eq:panjer_derivative}, switching the summations and rearranging the terms leads to the desired result.
\end{proof}
%

\section*{Appendix B: Pseudocode for the single-cluster multi-object filter}

Algorithm \ref{alg:1} shows the modular concept of the discussed approach in form of a particle filter. The sensor parameter at time $k$ is represented by a set of $N$ particles $y_k^i$ where $1 \leq i \leq N$, and each of the particles comes with a corresponding weight $w_k^i$ and and a set of multi-object estimation statistics $\Theta_k^i$. These statistics depend on the multi-object filter of choice: for the \ac{phd} filter, $\Theta_k^i$ is simply the intensity $\mu_k$ of the target population at time $k$, for the second-order \ac{phd} filter it is the intensity $\mu_k$ and variance $\var_k$ and for the \ac{cphd} filter it denotes the intensity $\mu_k$ and the full cardinality distribution $\rho_k$ of the target process, all dependent on the parameter $y_k^i$. Similarly, the functions {\tt Prediction} and {\tt Update}  are the respective prediction and update functions of the filter of choice, and the function {\tt MOL} can be the filter-specific multi-object likelihood of the respective filter, i.e.~\eqref{eq:likelihood_phd}, \eqref{eq:likelihood_panjer}, or \eqref{eq:likelihood_cphd}, or the method given in \cite{Leung2016Multifeature}. Throughout this article, the roulette method was used to perform the resampling, but other approaches can be used instead \cite{Doucet2001Sequential}. Furthermore, resampling is only performed if the effective sample size is low, based on the user-defined parameter $0 \leq r \leq 1$. In this article, $r=0.5$ was used.

\begin{algorithm}[H]
\begin{algorithmic}
	\State \emph{Input}
	\State Set of particles $\{ ({ y}^{i}_k,{ w}^{i}_k, \Theta^{i}_{k}) \}_{i=1}^{N}$
	\State Set of measurements ${Z}_{k+1}$
	\\
	\State \emph{Prediction}
	\For{$1 \leq i \leq N$}
	\State Sample ${ y}_{k+1}^{i} \sim \mathcal{N}({\bf y}^{i}_k, \Sigma_s)$
		\State $\Theta_{k+1|k}^{i} \leftarrow $ {\tt Prediction($\Theta_{k}^{i}$})
	\EndFor 
	\\
	\State \emph{Update}
	\For{$1 \leq i \leq N$}
			\State $\Theta_{k+1}^{i} \leftarrow $ {\tt Update($\Theta_{k+1|k}^{i},Z_{k+1}$})
			\State ${ w}^{i}_{k+1} \leftarrow $ {\tt MOL($\Theta_{k+1|k}^{i},Z_{k+1}$})
	\EndFor 
	\\ 
	\State \emph{Resampling} 
	\State $N_\text{eff} \leftarrow \cfrac{1}{\sum_{i=1}^N (w^i_{k+1})^2}$
	\\	
	\If{$N_\text{eff} \leq r \cdot N$}
	\State $\{{y}^{i}_{k+1}\}_{i=1}^N \leftarrow $ {\tt Resampling($\{({y}^{i}_{k+1},w^i_{k+1})\}_{i=1}^N$})
	\For{$1 \leq i \leq N$}
		\State $w^i_{k+1} \leftarrow \frac{1}{N}$
	\EndFor
	\EndIf
\end{algorithmic}
\caption{Algorithm for parameter estimation.\label{alg:1}}
\end{algorithm}

\bibliographystyle{IEEEtran}
\bibliography{refs}

\end{document}

%% file: sensordrift.tikz
%
\begin{tikzpicture}[]

\begin{axis}[%
width=0.7\linewidth,
height=0.7\linewidth,
at={(0.771875in,0.515625in)},
scale only axis,
axis x line*=bottom,
axis y line*=left,
xlabel=$x$,
ylabel=$y$,
every axis x label/.style={
    at={(ticklabel* cs:1.05)},
    anchor=west,
},
every axis y label/.style={
    at={(ticklabel* cs:1.05)},
    anchor=south,
},
xmin=-22,
xmax=2,
ymin=-2,
ymax=19.9
]
\addplot [color=black,solid,forget plot]
  table[row sep=crcr]{%
0	0\\
-0.00163935929479094	-0.0126820841930855\\
0.00337561626034735	-0.0530922974982752\\
-0.022925981304892	-0.123414197665864\\
-0.055883989863481	-0.195450104339777\\
-0.0349695213721801	-0.236161823612396\\
-0.0341523173669201	-0.237219529492606\\
-0.055960235492194	-0.219343622513441\\
-0.0743314561218991	-0.175803635617684\\
-0.126230953679885	-0.0934627818758643\\
-0.225835304123644	0.0500160245717312\\
-0.3674750512643	0.257482004315846\\
-0.517756995789861	0.485023443057277\\
-0.662845928842743	0.702166776512811\\
-0.813510690355783	0.882732887040467\\
-0.975395371727279	1.04332536854894\\
-1.13937823868087	1.19603360261135\\
-1.30990909996644	1.33942446141506\\
-1.45703702626551	1.48284481641269\\
-1.5542797996748	1.62704874953\\
-1.60571241249043	1.79516311470512\\
-1.6443201807525	1.96375715693431\\
-1.67214275286091	2.09440676932212\\
-1.64186732113017	2.19938436846284\\
-1.58317715627217	2.29649966426617\\
-1.54580391595925	2.41676818855523\\
-1.56209244547658	2.55754090662851\\
-1.6348077692039	2.73754167218695\\
-1.73942650680391	2.94426064031927\\
-1.87774390097901	3.13655631778531\\
-2.05755992262275	3.34222525264839\\
-2.28243542970054	3.54247959821001\\
-2.54199624661439	3.71396979986602\\
-2.828520431318	3.8360879278825\\
-3.1063382736817	3.91223566389946\\
-3.3059555216612	3.99982386778982\\
-3.47390356653601	4.11472497029165\\
-3.67474885011548	4.25694595824537\\
-3.89400111616377	4.39748043494008\\
-4.11190305182022	4.49378787936594\\
-4.32846267268796	4.51880678841037\\
-4.54704135886417	4.51191079811237\\
-4.75209585809782	4.50665667760081\\
-4.96466700684593	4.46759727651345\\
-5.17930520810435	4.39004803612692\\
-5.35983606009848	4.28419553307819\\
-5.51248758337565	4.19557971545705\\
-5.69012266480596	4.15666683041782\\
-5.89224748045869	4.17251331023687\\
-6.09842009010417	4.21925281375354\\
-6.34499904750203	4.25036652745333\\
-6.60447728509992	4.2601691209265\\
-6.85596579032959	4.27897963824588\\
-7.08573752795386	4.32708282139882\\
-7.31427305666162	4.40508986407819\\
-7.54024519663351	4.51529769418218\\
-7.73309639501572	4.64458187688537\\
-7.91747529984032	4.77649072276812\\
-8.11377069591199	4.9312879461896\\
-8.33437605448606	5.09083459683322\\
-8.53279907294603	5.21606878304445\\
-8.69263811545972	5.34845937464376\\
-8.82244713969131	5.54767351142291\\
-8.91886170630418	5.77702194116637\\
-8.98863903712909	5.9973713313547\\
-9.03084506179869	6.23419546823697\\
-9.07936630184869	6.49923500924556\\
-9.20975000411329	6.78395332621454\\
-9.36375876752583	7.06871271495621\\
-9.48533061216162	7.3480125653926\\
-9.62270516899562	7.62513222684362\\
-9.77544242703462	7.89839332813276\\
-9.97653784665954	8.14931665970031\\
-10.2365144713429	8.4103157746653\\
-10.5154517720757	8.69554317758901\\
-10.8066262166603	9.01263597526619\\
-11.1055618391248	9.36395415099408\\
-11.421171696386	9.73366912866895\\
-11.7744693736114	10.12958278532\\
-12.1507448481199	10.5648036678107\\
-12.5775515465903	11.017429931405\\
-13.0244219541652	11.456622159815\\
-13.4291200283507	11.9037066020331\\
-13.8052302235085	12.3050716243299\\
-14.1738249323067	12.676099738098\\
-14.5557503106966	13.0756028839224\\
-14.929472595126	13.4998758314092\\
-15.2717618780459	13.9463656698043\\
-15.5882037408834	14.3888980879405\\
-15.8787903732085	14.8068473206064\\
-16.1363522440144	15.2062882484075\\
-16.3840364095656	15.5993176908981\\
-16.6396993894834	15.9800687104914\\
-16.871352097651	16.3407540312187\\
-17.0745429695176	16.6836293542994\\
-17.2727477217761	17.027945763815\\
-17.4887364973744	17.3645458045637\\
-17.7395274343439	17.655572490296\\
-18.0334590862929	17.8856773820063\\
-18.3059802425464	18.0927420253925\\
};
\addplot [color=black,only marks,mark=x,mark options={solid},forget plot]
  table[row sep=crcr]{%
-18.3059802425464	18.0927420253925\\
};
\addplot [color=black,only marks,mark=o,mark options={solid},forget plot]
  table[row sep=crcr]{%
0 0\\
};
\end{axis}
\end{tikzpicture}%

%% file: truedrifts.tikz
%
\begin{tikzpicture}[spy using outlines={circle,yellow,magnification=4,size=1cm, connect spies}]

\begin{axis}[%
width=0.7\linewidth,
height=0.7\linewidth,
at={(-0.75in,-1.43in)},
scale only axis,
xmin=-3,
xmax=2,
ymin=-10,
ymax=0.5,
axis x line*=bottom,
axis y line*=left,
xlabel=$x$,
ylabel=$y$,
every axis x label/.style={
    at={(ticklabel* cs:1.05)},
    anchor=west,
},
every axis y label/.style={
    at={(ticklabel* cs:1.05)},
    anchor=south,
},
legend style={at={(0.97,0.97)},anchor=north east,legend cell align=left,align=left,draw=white!15!black,fill opacity=0.6,draw opacity=1,text opacity=1,font=\footnotesize}
]

\addplot [color=black,only marks,mark=o,mark options={solid},forget plot]
  table[row sep=crcr]{%
0 0\\
};
\addplot [color=black,solid,forget plot]
  table[row sep=crcr]{%
0	0\\
0.0228494805366445	-0.00855484415563054\\
0.0847610642674951	-0.0386893066943683\\
0.183176700733606	-0.101368575191659\\
0.266760220369878	-0.201993880693519\\
0.305052392597346	-0.322231414690507\\
0.355034523760315	-0.492162787532579\\
0.422819474175033	-0.716095348807933\\
0.47312566483537	-0.948197942049715\\
0.489010962190335	-1.16162751769615\\
0.428688934081147	-1.35069538540791\\
0.282583350965289	-1.56763087870515\\
0.095850386024269	-1.83343586573326\\
-0.0631980584906699	-2.13866334010012\\
-0.156434690064538	-2.47170550004999\\
-0.201720110550816	-2.85600157044321\\
-0.228568706114242	-3.28368053085249\\
-0.24990811132194	-3.69601963886631\\
-0.27987550103791	-4.09473630200992\\
-0.36273770010769	-4.4954812273556\\
-0.473787646295186	-4.93239030867382\\
-0.587769134132453	-5.3924590437208\\
-0.68119118454224	-5.83342549205282\\
-0.710021866646933	-6.27368587194122\\
-0.702252010518006	-6.70439315123428\\
-0.692311495812455	-7.12225972186957\\
-0.689089999462438	-7.5468136859123\\
-0.697596738412883	-7.94068227574295\\
-0.719964182198425	-8.30814807040801\\
-0.744767427447672	-8.67653100392253\\
};
\addplot [color=black,only marks,mark=x,mark options={solid},forget plot]
  table[row sep=crcr]{%
-0.744767427447672	-8.67653100392253\\
};
\end{axis}
\end{tikzpicture}%

%% file: card_mean_L1.tikz
%
\begin{tikzpicture}

\begin{axis}[%
width=0.8\linewidth,
height=1.2in,
at={(2.6in,1.114896in)},
scale only axis,
xmin=1,
xmax=100,
xlabel={time},
ymin=0,
ymax=35,
ylabel={\# targets},
axis x line*=bottom,
axis y line*=left,
legend style={at={(0.97,0.03)},anchor=south east,legend cell align=left,align=left,draw=white!15!black,fill opacity=0.6,draw opacity=1,text opacity=1,font=\footnotesize}
]
\addplot [color=black,solid]
  table[row sep=crcr]{%
1	1\\
2	3\\
3	7\\
4	8\\
5	8\\
6	9\\
7	10\\
8	10\\
9	12\\
10	13\\
11	15\\
12	14\\
13	14\\
14	16\\
15	16\\
16	17\\
17	15\\
18	17\\
19	15\\
20	17\\
21	17\\
22	17\\
23	20\\
24	19\\
25	20\\
26	18\\
27	18\\
28	16\\
29	16\\
30	18\\
31	18\\
32	18\\
33	17\\
34	19\\
35	23\\
36	22\\
37	24\\
38	21\\
39	20\\
40	19\\
41	23\\
42	24\\
43	23\\
44	25\\
45	25\\
46	24\\
47	26\\
48	24\\
49	24\\
50	28\\
51	24\\
52	27\\
53	28\\
54	28\\
55	27\\
56	24\\
57	24\\
58	25\\
59	25\\
60	23\\
61	24\\
62	24\\
63	29\\
64	28\\
65	27\\
66	28\\
67	27\\
68	28\\
69	27\\
70	31\\
71	29\\
72	29\\
73	30\\
74	30\\
75	29\\
76	29\\
77	29\\
78	30\\
79	31\\
80	32\\
81	30\\
82	28\\
83	27\\
84	26\\
85	25\\
86	22\\
87	21\\
88	25\\
89	25\\
90	27\\
91	27\\
92	29\\
93	30\\
94	28\\
95	29\\
96	29\\
97	25\\
98	23\\
99	24\\
100	23\\
};
\addlegendentry{gt};

\addplot[area legend,solid,fill=red,opacity=1.000000e-01,draw=none,forget plot]
table[row sep=crcr] {%
x	y\\
1	0.635009990479423\\
2	1.70565302699451\\
3	4.08569878345285\\
4	7.43560137304226\\
5	8.62689320698273\\
6	8.42437457157522\\
7	8.55350026966301\\
8	10.374766154276\\
9	10.5662978850247\\
10	11.7423020103865\\
11	13.2720558770078\\
12	14.4080008758153\\
13	14.3067470312878\\
14	13.9094354873841\\
15	16.123505425155\\
16	16.2900630410849\\
17	15.031085171979\\
18	15.2893572062318\\
19	15.0272243533013\\
20	15.5165766972204\\
21	16.0902917641979\\
22	17.3043560070453\\
23	18.026189989994\\
24	17.6627154806487\\
25	18.3422046805768\\
26	18.003937219734\\
27	17.8860161951761\\
28	16.1753829968681\\
29	16.2442572710237\\
30	16.3096814881459\\
31	17.0088051620277\\
32	16.5314840513954\\
33	16.4336327823654\\
34	16.6870651910787\\
35	19.5081575292033\\
36	21.5921892829452\\
37	21.8263599557516\\
38	20.4225675429353\\
39	20.2073700337856\\
40	19.030675211954\\
41	19.7118108416911\\
42	21.9773987962909\\
43	21.5582233785566\\
44	22.577916914715\\
45	23.9594164619159\\
46	23.6788963488915\\
47	23.3144019852892\\
48	23.9828604475735\\
49	23.2305493468652\\
50	24.5398925973099\\
51	23.5703537595947\\
52	24.3919484223384\\
53	26.1589217751846\\
54	27.307377153289\\
55	24.654238180558\\
56	24.0701196752798\\
57	22.4895733549631\\
58	24.1808298692747\\
59	25.0817798942984\\
60	23.0343346604149\\
61	23.1443394226301\\
62	23.4830410108988\\
63	23.9736985620941\\
64	27.8445196292125\\
65	26.2530408390177\\
66	26.5385040451909\\
67	27.2776197500296\\
68	26.3370826522752\\
69	25.5833855895472\\
70	26.5087859588385\\
71	29.1827739957389\\
72	26.1675097118789\\
73	27.9523040574679\\
74	27.7579075829419\\
75	29.0130842261697\\
76	26.7971452453916\\
77	29.1214680026963\\
78	27.67652486079\\
79	28.5999170634998\\
80	30.1908168529791\\
81	29.6847627295244\\
82	26.8781690287163\\
83	26.314555147315\\
84	24.7285351422026\\
85	24.2963820753033\\
86	21.0440782365324\\
87	19.9748152997657\\
88	21.1862440767188\\
89	22.8268563560325\\
90	24.0933515977005\\
91	25.9663575057167\\
92	26.8336910362215\\
93	28.0239841062892\\
94	25.4470752190318\\
95	27.5298103429792\\
96	28.3350995463021\\
97	24.3758115725784\\
98	22.1442240970841\\
99	21.6628432280281\\
100	21.8730456239151\\
100	24.874906963068\\
99	25.1923040283625\\
98	25.0019808764397\\
97	27.9806929780481\\
96	31.8379353005401\\
95	30.6482331648544\\
94	29.090394793793\\
93	30.9604566621047\\
92	30.0548062962703\\
91	29.034082461155\\
90	27.1272585895156\\
89	26.2762138969986\\
88	24.6788614281686\\
87	22.6548810755611\\
86	23.8761381621713\\
85	28.2670933662081\\
84	28.229841299791\\
83	30.2293275660061\\
82	29.9823062135727\\
81	33.4902284214389\\
80	32.9734686370195\\
79	31.0206573557951\\
78	30.8567602973667\\
77	31.9394527039838\\
76	30.6999835234227\\
75	31.7449451868793\\
74	31.940326966148\\
73	31.478076334366\\
72	29.6121439538087\\
71	32.1315377256626\\
70	29.4247946721115\\
69	29.4308274861585\\
68	29.6917674517528\\
67	29.9105668770051\\
66	29.3456415479095\\
65	29.5635192852535\\
64	31.2416941071625\\
63	28.297031828736\\
62	25.9478070714662\\
61	26.1958259278436\\
60	26.1487084948589\\
59	28.3861776971329\\
58	27.328450706262\\
57	26.5351528816923\\
56	27.2133095360244\\
55	27.8082934324943\\
54	31.5991030705832\\
53	29.5565175551356\\
52	27.9123281698951\\
51	26.3441739439522\\
50	27.7505848872683\\
49	26.0410423517622\\
48	27.2955871131625\\
47	26.5920080724688\\
46	27.2643100127245\\
45	26.9935271994958\\
44	25.8952249562788\\
43	24.1743356840805\\
42	25.8331628804522\\
41	22.7868540947567\\
40	22.1671320877562\\
39	23.1304570563494\\
38	24.2984830315232\\
37	24.4322247506711\\
36	24.3898655744544\\
35	22.6778228211243\\
34	19.5761489406934\\
33	19.2855342129858\\
32	19.5589837789388\\
31	20.2222381459788\\
30	19.6585596154052\\
29	19.1954577204306\\
28	19.1656311002376\\
27	21.2833326906181\\
26	21.5426034677792\\
25	21.8851968653653\\
24	20.2024877511342\\
23	20.3824556853237\\
22	19.9161565716241\\
21	19.4951481351672\\
20	18.4946069159712\\
19	17.8543513444623\\
18	18.5448094455534\\
17	17.984483082096\\
16	19.3974335145762\\
15	18.95154876744\\
14	16.6773559066323\\
13	17.1708407221817\\
12	16.8020024763043\\
11	16.5364451974889\\
10	14.4296204656259\\
9	13.200834045317\\
8	12.7693475257776\\
7	11.4917777965065\\
6	11.4699696592413\\
5	10.6385173156007\\
4	9.93975832012594\\
3	6.35994151988949\\
2	4.12153257855888\\
1	2.66744314359553\\
}--cycle;

\addplot [color=pink,solid,forget plot]
  table[row sep=crcr]{%
1	0.635009990479423\\
2	1.70565302699451\\
3	4.08569878345285\\
4	7.43560137304226\\
5	8.62689320698273\\
6	8.42437457157522\\
7	8.55350026966301\\
8	10.374766154276\\
9	10.5662978850247\\
10	11.7423020103865\\
11	13.2720558770078\\
12	14.4080008758153\\
13	14.3067470312878\\
14	13.9094354873841\\
15	16.123505425155\\
16	16.2900630410849\\
17	15.031085171979\\
18	15.2893572062318\\
19	15.0272243533013\\
20	15.5165766972204\\
21	16.0902917641979\\
22	17.3043560070453\\
23	18.026189989994\\
24	17.6627154806487\\
25	18.3422046805768\\
26	18.003937219734\\
27	17.8860161951761\\
28	16.1753829968681\\
29	16.2442572710237\\
30	16.3096814881459\\
31	17.0088051620277\\
32	16.5314840513954\\
33	16.4336327823654\\
34	16.6870651910787\\
35	19.5081575292033\\
36	21.5921892829452\\
37	21.8263599557516\\
38	20.4225675429353\\
39	20.2073700337856\\
40	19.030675211954\\
41	19.7118108416911\\
42	21.9773987962909\\
43	21.5582233785566\\
44	22.577916914715\\
45	23.9594164619159\\
46	23.6788963488915\\
47	23.3144019852892\\
48	23.9828604475735\\
49	23.2305493468652\\
50	24.5398925973099\\
51	23.5703537595947\\
52	24.3919484223384\\
53	26.1589217751846\\
54	27.307377153289\\
55	24.654238180558\\
56	24.0701196752798\\
57	22.4895733549631\\
58	24.1808298692747\\
59	25.0817798942984\\
60	23.0343346604149\\
61	23.1443394226301\\
62	23.4830410108988\\
63	23.9736985620941\\
64	27.8445196292125\\
65	26.2530408390177\\
66	26.5385040451909\\
67	27.2776197500296\\
68	26.3370826522752\\
69	25.5833855895472\\
70	26.5087859588385\\
71	29.1827739957389\\
72	26.1675097118789\\
73	27.9523040574679\\
74	27.7579075829419\\
75	29.0130842261697\\
76	26.7971452453916\\
77	29.1214680026963\\
78	27.67652486079\\
79	28.5999170634998\\
80	30.1908168529791\\
81	29.6847627295244\\
82	26.8781690287163\\
83	26.314555147315\\
84	24.7285351422026\\
85	24.2963820753033\\
86	21.0440782365324\\
87	19.9748152997657\\
88	21.1862440767188\\
89	22.8268563560325\\
90	24.0933515977005\\
91	25.9663575057167\\
92	26.8336910362215\\
93	28.0239841062892\\
94	25.4470752190318\\
95	27.5298103429792\\
96	28.3350995463021\\
97	24.3758115725784\\
98	22.1442240970841\\
99	21.6628432280281\\
100	21.8730456239151\\
};
\addplot [color=pink,solid,forget plot]
  table[row sep=crcr]{%
1	2.66744314359553\\
2	4.12153257855888\\
3	6.35994151988949\\
4	9.93975832012594\\
5	10.6385173156007\\
6	11.4699696592413\\
7	11.4917777965065\\
8	12.7693475257776\\
9	13.200834045317\\
10	14.4296204656259\\
11	16.5364451974889\\
12	16.8020024763043\\
13	17.1708407221817\\
14	16.6773559066323\\
15	18.95154876744\\
16	19.3974335145762\\
17	17.984483082096\\
18	18.5448094455534\\
19	17.8543513444623\\
20	18.4946069159712\\
21	19.4951481351672\\
22	19.9161565716241\\
23	20.3824556853237\\
24	20.2024877511342\\
25	21.8851968653653\\
26	21.5426034677792\\
27	21.2833326906181\\
28	19.1656311002376\\
29	19.1954577204306\\
30	19.6585596154052\\
31	20.2222381459788\\
32	19.5589837789388\\
33	19.2855342129858\\
34	19.5761489406934\\
35	22.6778228211243\\
36	24.3898655744544\\
37	24.4322247506711\\
38	24.2984830315232\\
39	23.1304570563494\\
40	22.1671320877562\\
41	22.7868540947567\\
42	25.8331628804522\\
43	24.1743356840805\\
44	25.8952249562788\\
45	26.9935271994958\\
46	27.2643100127245\\
47	26.5920080724688\\
48	27.2955871131625\\
49	26.0410423517622\\
50	27.7505848872683\\
51	26.3441739439522\\
52	27.9123281698951\\
53	29.5565175551356\\
54	31.5991030705832\\
55	27.8082934324943\\
56	27.2133095360244\\
57	26.5351528816923\\
58	27.328450706262\\
59	28.3861776971329\\
60	26.1487084948589\\
61	26.1958259278436\\
62	25.9478070714662\\
63	28.297031828736\\
64	31.2416941071625\\
65	29.5635192852535\\
66	29.3456415479095\\
67	29.9105668770051\\
68	29.6917674517528\\
69	29.4308274861585\\
70	29.4247946721115\\
71	32.1315377256626\\
72	29.6121439538087\\
73	31.478076334366\\
74	31.940326966148\\
75	31.7449451868793\\
76	30.6999835234227\\
77	31.9394527039838\\
78	30.8567602973667\\
79	31.0206573557951\\
80	32.9734686370195\\
81	33.4902284214389\\
82	29.9823062135727\\
83	30.2293275660061\\
84	28.229841299791\\
85	28.2670933662081\\
86	23.8761381621713\\
87	22.6548810755611\\
88	24.6788614281686\\
89	26.2762138969986\\
90	27.1272585895156\\
91	29.034082461155\\
92	30.0548062962703\\
93	30.9604566621047\\
94	29.090394793793\\
95	30.6482331648544\\
96	31.8379353005401\\
97	27.9806929780481\\
98	25.0019808764397\\
99	25.1923040283625\\
100	24.874906963068\\
};
\addplot [color=red,solid]
  table[row sep=crcr]{%
1	1.65122656703747\\
2	2.91359280277669\\
3	5.22282015167117\\
4	8.6876798465841\\
5	9.63270526129169\\
6	9.94717211540828\\
7	10.0226390330848\\
8	11.5720568400268\\
9	11.8835659651708\\
10	13.0859612380062\\
11	14.9042505372484\\
12	15.6050016760598\\
13	15.7387938767348\\
14	15.2933956970082\\
15	17.5375270962975\\
16	17.8437482778306\\
17	16.5077841270375\\
18	16.9170833258926\\
19	16.4407878488818\\
20	17.0055918065958\\
21	17.7927199496825\\
22	18.6102562893347\\
23	19.2043228376589\\
24	18.9326016158914\\
25	20.1137007729711\\
26	19.7732703437566\\
27	19.5846744428971\\
28	17.6705070485529\\
29	17.7198574957271\\
30	17.9841205517755\\
31	18.6155216540032\\
32	18.0452339151671\\
33	17.8595834976756\\
34	18.131607065886\\
35	21.0929901751638\\
36	22.9910274286998\\
37	23.1292923532114\\
38	22.3605252872292\\
39	21.6689135450675\\
40	20.5989036498551\\
41	21.2493324682239\\
42	23.9052808383715\\
43	22.8662795313185\\
44	24.2365709354969\\
45	25.4764718307059\\
46	25.471603180808\\
47	24.953205028879\\
48	25.639223780368\\
49	24.6357958493137\\
50	26.1452387422891\\
51	24.9572638517735\\
52	26.1521382961168\\
53	27.8577196651601\\
54	29.4532401119361\\
55	26.2312658065262\\
56	25.6417146056521\\
57	24.5123631183277\\
58	25.7546402877684\\
59	26.7339787957156\\
60	24.5915215776369\\
61	24.6700826752369\\
62	24.7154240411825\\
63	26.1353651954151\\
64	29.5431068681875\\
65	27.9082800621356\\
66	27.9420727965502\\
67	28.5940933135174\\
68	28.014425052014\\
69	27.5071065378528\\
70	27.966790315475\\
71	30.6571558607007\\
72	27.8898268328438\\
73	29.7151901959169\\
74	29.8491172745449\\
75	30.3790147065245\\
76	28.7485643844071\\
77	30.53046035334\\
78	29.2666425790784\\
79	29.8102872096474\\
80	31.5821427449993\\
81	31.5874955754817\\
82	28.4302376211445\\
83	28.2719413566606\\
84	26.4791882209968\\
85	26.2817377207557\\
86	22.4601081993518\\
87	21.3148481876634\\
88	22.9325527524437\\
89	24.5515351265155\\
90	25.610305093608\\
91	27.5002199834359\\
92	28.4442486662459\\
93	29.492220384197\\
94	27.2687350064124\\
95	29.0890217539168\\
96	30.0865174234211\\
97	26.1782522753133\\
98	23.5731024867619\\
99	23.4275736281953\\
100	23.3739762934915\\
};
\addlegendentry{PHD};

\addplot[area legend,solid,fill=green,opacity=1.000000e-01,draw=none,forget plot]
table[row sep=crcr] {%
x	y\\
1	0.635009990479422\\
2	1.73909013707274\\
3	4.16480502266475\\
4	7.50657190867381\\
5	8.81824400065663\\
6	8.87755808564691\\
7	9.0314566393272\\
8	10.6029266673398\\
9	10.9489531730239\\
10	12.1356513326494\\
11	13.6051704732157\\
12	14.6990296019627\\
13	14.7164013207416\\
14	14.5225377890898\\
15	16.3786134355988\\
16	16.6864055490346\\
17	15.5094442616055\\
18	15.7804014315964\\
19	15.4609583177108\\
20	16.0256888356455\\
21	16.4640444208073\\
22	17.663206552544\\
23	18.4157534069807\\
24	18.2138826420065\\
25	18.8117806807496\\
26	18.4542447406488\\
27	18.371936862093\\
28	16.7629566652803\\
29	16.7186062848149\\
30	16.7997803533347\\
31	17.4114521173029\\
32	17.1585110830727\\
33	16.952573871118\\
34	17.1629953263096\\
35	19.7152355005255\\
36	21.807215128319\\
37	22.3320408223433\\
38	20.9767186553726\\
39	20.6642034860172\\
40	19.5676391590337\\
41	20.247213190196\\
42	22.2545949245386\\
43	22.1276618482377\\
44	23.0125344193362\\
45	24.2626636671282\\
46	24.1823275195903\\
47	23.9373359254631\\
48	24.3477645945783\\
49	23.7281435492879\\
50	24.883221980502\\
51	24.140177085121\\
52	24.8494260422928\\
53	26.398249960016\\
54	27.6035466016573\\
55	25.6170936123786\\
56	24.6372314573438\\
57	23.1673386592886\\
58	24.5171578745267\\
59	25.3905882523186\\
60	23.7073992799507\\
61	23.6342108059723\\
62	23.9032963202682\\
63	24.4859053849339\\
64	27.8821323248266\\
65	26.9019166200184\\
66	27.0137996723095\\
67	27.6201440820021\\
68	26.9344736766\\
69	26.1355473726416\\
70	26.9682227802212\\
71	29.2528055771968\\
72	27.032430433777\\
73	28.3512274090451\\
74	28.2409052086414\\
75	29.3344396872526\\
76	27.4309265914226\\
77	29.2697356379047\\
78	28.3325334911537\\
79	28.9805892500235\\
80	30.4286309550403\\
81	30.1525251846256\\
82	27.6261902035778\\
83	26.8066849105436\\
84	25.3236301922905\\
85	24.8201183000927\\
86	21.900563731692\\
87	20.4526292936399\\
88	21.5910318468329\\
89	23.0887208771756\\
90	24.3676763754716\\
91	26.1454011855612\\
92	27.1999844161911\\
93	28.3117335400963\\
94	26.0985459008597\\
95	27.77970732011\\
96	28.6437975480383\\
97	25.0993111086312\\
98	22.8418679279605\\
99	22.1251190822892\\
100	22.2919816645008\\
100	25.1346660268185\\
99	25.4910231099306\\
98	25.7287250454141\\
97	28.7559945949572\\
96	31.7136823971434\\
95	30.5623626225064\\
94	29.5462297779784\\
93	30.8789424469217\\
92	29.9435731269578\\
91	28.8739451778291\\
90	27.1223542577445\\
89	26.1319971773611\\
88	24.6168842622703\\
87	23.2242760030984\\
86	24.7867667800223\\
85	28.4418016129477\\
84	28.7550793089171\\
83	30.4380902765479\\
82	30.7204452366549\\
81	33.4838447749403\\
80	32.7767367279197\\
79	31.1729205657809\\
78	31.2760438261893\\
77	31.9278465112952\\
76	31.1442964947289\\
75	31.8159606755127\\
74	31.9791567627221\\
73	31.2608093084812\\
72	30.1909022987603\\
71	31.8893531989766\\
70	29.6128846570182\\
69	29.6675943089523\\
68	29.9562044839439\\
67	30.0076590271347\\
66	29.6103873858751\\
65	29.9433928035506\\
64	30.710802164691\\
63	28.0848209078962\\
62	26.25204044456\\
61	26.4952025331115\\
60	26.6454311284587\\
59	28.3485134602432\\
58	27.3615508913009\\
57	26.87303972615\\
56	27.592122557665\\
55	28.6052887188419\\
54	31.2902827300893\\
53	29.4056506632785\\
52	27.885690717762\\
51	26.8081086358404\\
50	27.7206973706348\\
49	26.4939448707425\\
48	27.413846217439\\
47	26.9398740380292\\
46	27.342776269154\\
45	26.9585181359666\\
44	25.8504877418864\\
43	24.5786938936157\\
42	25.4628475950091\\
41	23.0150926274779\\
40	22.6284414790097\\
39	23.5945143938671\\
38	24.5942909004446\\
37	24.5508785221389\\
36	24.1708582170436\\
35	22.3264137800489\\
34	19.8832165525484\\
33	19.6262671030386\\
32	19.9629688257506\\
31	20.3911642324642\\
30	19.9483157271362\\
29	19.5639301564767\\
28	19.9014401224653\\
27	21.5453070854774\\
26	21.8237209941474\\
25	21.8598506302185\\
24	20.5140705265925\\
23	20.6373192884972\\
22	20.0432793821864\\
21	19.6583227896023\\
20	18.6926005997724\\
19	18.2294726660587\\
18	18.8489411341026\\
17	18.5556411130044\\
16	19.6002889903333\\
15	18.8559636678803\\
14	17.1247580013215\\
13	17.4598951948042\\
12	16.9997879387051\\
11	16.4691927323202\\
10	14.5267296487465\\
9	13.5047982088166\\
8	12.8716478680525\\
7	11.8172214527775\\
6	11.6906003760376\\
5	10.8209704602171\\
4	9.7544768990001\\
3	6.34311561246652\\
2	4.14122248749918\\
1	2.66744314359552\\
}--cycle;

\addplot [color=white!75!green,solid,forget plot]
  table[row sep=crcr]{%
1	0.635009990479422\\
2	1.73909013707274\\
3	4.16480502266475\\
4	7.50657190867381\\
5	8.81824400065663\\
6	8.87755808564691\\
7	9.0314566393272\\
8	10.6029266673398\\
9	10.9489531730239\\
10	12.1356513326494\\
11	13.6051704732157\\
12	14.6990296019627\\
13	14.7164013207416\\
14	14.5225377890898\\
15	16.3786134355988\\
16	16.6864055490346\\
17	15.5094442616055\\
18	15.7804014315964\\
19	15.4609583177108\\
20	16.0256888356455\\
21	16.4640444208073\\
22	17.663206552544\\
23	18.4157534069807\\
24	18.2138826420065\\
25	18.8117806807496\\
26	18.4542447406488\\
27	18.371936862093\\
28	16.7629566652803\\
29	16.7186062848149\\
30	16.7997803533347\\
31	17.4114521173029\\
32	17.1585110830727\\
33	16.952573871118\\
34	17.1629953263096\\
35	19.7152355005255\\
36	21.807215128319\\
37	22.3320408223433\\
38	20.9767186553726\\
39	20.6642034860172\\
40	19.5676391590337\\
41	20.247213190196\\
42	22.2545949245386\\
43	22.1276618482377\\
44	23.0125344193362\\
45	24.2626636671282\\
46	24.1823275195903\\
47	23.9373359254631\\
48	24.3477645945783\\
49	23.7281435492879\\
50	24.883221980502\\
51	24.140177085121\\
52	24.8494260422928\\
53	26.398249960016\\
54	27.6035466016573\\
55	25.6170936123786\\
56	24.6372314573438\\
57	23.1673386592886\\
58	24.5171578745267\\
59	25.3905882523186\\
60	23.7073992799507\\
61	23.6342108059723\\
62	23.9032963202682\\
63	24.4859053849339\\
64	27.8821323248266\\
65	26.9019166200184\\
66	27.0137996723095\\
67	27.6201440820021\\
68	26.9344736766\\
69	26.1355473726416\\
70	26.9682227802212\\
71	29.2528055771968\\
72	27.032430433777\\
73	28.3512274090451\\
74	28.2409052086414\\
75	29.3344396872526\\
76	27.4309265914226\\
77	29.2697356379047\\
78	28.3325334911537\\
79	28.9805892500235\\
80	30.4286309550403\\
81	30.1525251846256\\
82	27.6261902035778\\
83	26.8066849105436\\
84	25.3236301922905\\
85	24.8201183000927\\
86	21.900563731692\\
87	20.4526292936399\\
88	21.5910318468329\\
89	23.0887208771756\\
90	24.3676763754716\\
91	26.1454011855612\\
92	27.1999844161911\\
93	28.3117335400963\\
94	26.0985459008597\\
95	27.77970732011\\
96	28.6437975480383\\
97	25.0993111086312\\
98	22.8418679279605\\
99	22.1251190822892\\
100	22.2919816645008\\
};
\addplot [color=white!75!green,solid,forget plot]
  table[row sep=crcr]{%
1	2.66744314359552\\
2	4.14122248749918\\
3	6.34311561246652\\
4	9.7544768990001\\
5	10.8209704602171\\
6	11.6906003760376\\
7	11.8172214527775\\
8	12.8716478680525\\
9	13.5047982088166\\
10	14.5267296487465\\
11	16.4691927323202\\
12	16.9997879387051\\
13	17.4598951948042\\
14	17.1247580013215\\
15	18.8559636678803\\
16	19.6002889903333\\
17	18.5556411130044\\
18	18.8489411341026\\
19	18.2294726660587\\
20	18.6926005997724\\
21	19.6583227896023\\
22	20.0432793821864\\
23	20.6373192884972\\
24	20.5140705265925\\
25	21.8598506302185\\
26	21.8237209941474\\
27	21.5453070854774\\
28	19.9014401224653\\
29	19.5639301564767\\
30	19.9483157271362\\
31	20.3911642324642\\
32	19.9629688257506\\
33	19.6262671030386\\
34	19.8832165525484\\
35	22.3264137800489\\
36	24.1708582170436\\
37	24.5508785221389\\
38	24.5942909004446\\
39	23.5945143938671\\
40	22.6284414790097\\
41	23.0150926274779\\
42	25.4628475950091\\
43	24.5786938936157\\
44	25.8504877418864\\
45	26.9585181359666\\
46	27.342776269154\\
47	26.9398740380292\\
48	27.413846217439\\
49	26.4939448707425\\
50	27.7206973706348\\
51	26.8081086358404\\
52	27.885690717762\\
53	29.4056506632785\\
54	31.2902827300893\\
55	28.6052887188419\\
56	27.592122557665\\
57	26.87303972615\\
58	27.3615508913009\\
59	28.3485134602432\\
60	26.6454311284587\\
61	26.4952025331115\\
62	26.25204044456\\
63	28.0848209078962\\
64	30.710802164691\\
65	29.9433928035506\\
66	29.6103873858751\\
67	30.0076590271347\\
68	29.9562044839439\\
69	29.6675943089523\\
70	29.6128846570182\\
71	31.8893531989766\\
72	30.1909022987603\\
73	31.2608093084812\\
74	31.9791567627221\\
75	31.8159606755127\\
76	31.1442964947289\\
77	31.9278465112952\\
78	31.2760438261893\\
79	31.1729205657809\\
80	32.7767367279197\\
81	33.4838447749403\\
82	30.7204452366549\\
83	30.4380902765479\\
84	28.7550793089171\\
85	28.4418016129477\\
86	24.7867667800223\\
87	23.2242760030984\\
88	24.6168842622703\\
89	26.1319971773611\\
90	27.1223542577445\\
91	28.8739451778291\\
92	29.9435731269578\\
93	30.8789424469217\\
94	29.5462297779784\\
95	30.5623626225064\\
96	31.7136823971434\\
97	28.7559945949572\\
98	25.7287250454141\\
99	25.4910231099306\\
100	25.1346660268185\\
};
\addplot [color=green,solid]
  table[row sep=crcr]{%
1	1.65122656703747\\
2	2.94015631228596\\
3	5.25396031756563\\
4	8.63052440383695\\
5	9.81960723043687\\
6	10.2840792308422\\
7	10.4243390460524\\
8	11.7372872676962\\
9	12.2268756909203\\
10	13.331190490698\\
11	15.037181602768\\
12	15.8494087703339\\
13	16.0881482577729\\
14	15.8236478952056\\
15	17.6172885517395\\
16	18.1433472696839\\
17	17.032542687305\\
18	17.3146712828495\\
19	16.8452154918847\\
20	17.359144717709\\
21	18.0611836052048\\
22	18.8532429673652\\
23	19.5265363477389\\
24	19.3639765842995\\
25	20.3358156554841\\
26	20.1389828673981\\
27	19.9586219737852\\
28	18.3321983938728\\
29	18.1412682206458\\
30	18.3740480402354\\
31	18.9013081748835\\
32	18.5607399544116\\
33	18.2894204870783\\
34	18.523105939429\\
35	21.0208246402872\\
36	22.9890366726813\\
37	23.4414596722411\\
38	22.7855047779086\\
39	22.1293589399421\\
40	21.0980403190217\\
41	21.631152908837\\
42	23.8587212597738\\
43	23.3531778709267\\
44	24.4315110806113\\
45	25.6105909015474\\
46	25.7625518943722\\
47	25.4386049817462\\
48	25.8808054060086\\
49	25.1110442100152\\
50	26.3019596755684\\
51	25.4741428604807\\
52	26.3675583800274\\
53	27.9019503116473\\
54	29.4469146658733\\
55	27.1111911656103\\
56	26.1146770075044\\
57	25.0201891927193\\
58	25.9393543829138\\
59	26.8695508562809\\
60	25.1764152042047\\
61	25.0647066695419\\
62	25.0776683824141\\
63	26.285363146415\\
64	29.2964672447588\\
65	28.4226547117845\\
66	28.3120935290923\\
67	28.8139015545684\\
68	28.4453390802719\\
69	27.901570840797\\
70	28.2905537186197\\
71	30.5710793880867\\
72	28.6116663662687\\
73	29.8060183587631\\
74	30.1100309856818\\
75	30.5752001813826\\
76	29.2876115430758\\
77	30.5987910745999\\
78	29.8042886586715\\
79	30.0767549079022\\
80	31.60268384148\\
81	31.818184979783\\
82	29.1733177201163\\
83	28.6223875935458\\
84	27.0393547506038\\
85	26.6309599565202\\
86	23.3436652558572\\
87	21.8384526483692\\
88	23.1039580545516\\
89	24.6103590272684\\
90	25.7450153166081\\
91	27.5096731816952\\
92	28.5717787715745\\
93	29.595337993509\\
94	27.8223878394191\\
95	29.1710349713082\\
96	30.1787399725908\\
97	26.9276528517942\\
98	24.2852964866873\\
99	23.8080710961099\\
100	23.7133238456597\\
};
\addlegendentry{Panjer};

\addplot[area legend,solid,fill=blue,opacity=1.000000e-01,draw=none,forget plot]
table[row sep=crcr] {%
x	y\\
1	0.637765407216571\\
2	1.74729442796352\\
3	4.20606957689721\\
4	7.51170893572576\\
5	8.94408941844826\\
6	9.27839208667847\\
7	9.51154618851857\\
8	10.7723732540086\\
9	11.2762956390018\\
10	12.4108537389769\\
11	13.802741989056\\
12	14.9085716759943\\
13	15.0406164666811\\
14	15.0023627227979\\
15	16.5230055426455\\
16	16.9513907476588\\
17	15.9734669980193\\
18	16.2209263160402\\
19	15.851282837678\\
20	16.4142116429272\\
21	16.7397413879563\\
22	17.9160444350637\\
23	18.6812750289209\\
24	18.5898258205037\\
25	19.0552647277518\\
26	18.790420244255\\
27	18.7639027553167\\
28	17.3508077942345\\
29	17.180724344497\\
30	17.1900706758713\\
31	17.7100669695757\\
32	17.6282880126376\\
33	17.4002871336046\\
34	17.5248533453461\\
35	19.8614046268757\\
36	21.9326775792168\\
37	22.5955745911792\\
38	21.4074283188472\\
39	21.0214942246761\\
40	20.0431012193766\\
41	20.6204397329785\\
42	22.4244098923807\\
43	22.4963568597677\\
44	23.2694191729001\\
45	24.4419563914404\\
46	24.4947137786235\\
47	24.3293847516352\\
48	24.5855542445021\\
49	24.0792645812904\\
50	25.112604594374\\
51	24.5347242548015\\
52	25.1280988687296\\
53	26.5673863685156\\
54	27.7745692639493\\
55	26.2939619259438\\
56	25.0707719854052\\
57	23.7284175922147\\
58	24.7477277156912\\
59	25.5643340517051\\
60	24.2156496282998\\
61	23.9808632259933\\
62	24.1742410983378\\
63	24.7662080105491\\
64	27.9647353612692\\
65	27.3343133071876\\
66	27.3170167918774\\
67	27.8085206936683\\
68	27.2838775569305\\
69	26.5071018532624\\
70	27.2698299275252\\
71	29.3721490503335\\
72	27.5969913845541\\
73	28.5711337013707\\
74	28.5026706322823\\
75	29.5155920337463\\
76	27.918991060812\\
77	29.4246333428379\\
78	28.7068735371921\\
79	29.2331408177956\\
80	30.559620901651\\
81	30.4113143222669\\
82	28.1538729436245\\
83	27.1630424722158\\
84	25.7921737363665\\
85	25.2120798915719\\
86	22.5951904001462\\
87	20.9253634901123\\
88	21.8850135190077\\
89	23.2458020699987\\
90	24.5267051354751\\
91	26.2507069273301\\
92	27.3604659166043\\
93	28.4316129343528\\
94	26.5645378528735\\
95	28.0329018195697\\
96	28.7922959329063\\
97	25.7303002994928\\
98	23.4645293112346\\
99	22.4984570698576\\
100	22.5751006784584\\
100	25.3496511412019\\
99	25.7963563332271\\
98	26.3099607313803\\
97	29.2378976804989\\
96	31.8321678258068\\
95	30.6388403315829\\
94	29.8223910842493\\
93	30.9986034963281\\
92	30.0432492565513\\
91	28.963304051307\\
90	27.2315390373411\\
89	26.2562597793441\\
88	24.8016747453438\\
87	23.7127160115189\\
86	25.5084915353107\\
85	28.6064133521473\\
84	29.0960067083088\\
83	30.6497979001753\\
82	31.1633088442776\\
81	33.5765029188546\\
80	32.9341817468632\\
79	31.3667374115492\\
78	31.5711287407939\\
77	32.1108943867398\\
76	31.3536526543244\\
75	31.9807990134905\\
74	32.142021733499\\
73	31.4466784757668\\
72	30.5117409486279\\
71	32.0269211206484\\
70	29.8307509995372\\
69	29.8485194883365\\
68	30.1549992350807\\
67	30.1911737009587\\
66	29.8415419267657\\
65	30.1469892249765\\
64	30.8914665293294\\
63	28.2324342903004\\
62	26.5101455214785\\
61	26.7606590130734\\
60	26.9363709387639\\
59	28.4832465752829\\
58	27.5876945642438\\
57	27.0717546482607\\
56	27.9289857515764\\
55	29.0601492071641\\
54	31.3854403183979\\
53	29.546700562271\\
52	28.0724440680538\\
51	27.0827456303811\\
50	27.9149061434745\\
49	26.769407377999\\
48	27.5767868592125\\
47	27.1718390396631\\
46	27.4235884866189\\
45	27.0830460704878\\
44	26.033424779863\\
43	24.8201804391003\\
42	25.6148873282042\\
41	23.2903944107963\\
40	22.9431176516779\\
39	23.9201227535948\\
38	24.7580691994087\\
37	24.7099506961104\\
36	24.2967296810529\\
35	22.5149245949488\\
34	20.1652173073015\\
33	19.8988562182206\\
32	20.2490175703738\\
31	20.5807252125888\\
30	20.209380867445\\
29	19.8852893548938\\
28	20.3979968340455\\
27	21.7357335831719\\
26	22.0144818334013\\
25	22.0544106394365\\
24	20.7618659694124\\
23	20.8471120364943\\
22	20.2110979058949\\
21	19.8496655493627\\
20	18.9245458248563\\
19	18.5313058089593\\
18	19.1248033122157\\
17	18.9510678685218\\
16	19.7989735879974\\
15	19.0246886866667\\
14	17.4759234669588\\
13	17.6527944386637\\
12	17.1523137340407\\
11	16.5926892490734\\
10	14.712424210761\\
9	13.7356326802694\\
8	13.0655564570239\\
7	12.0863114282189\\
6	11.8596086582055\\
5	10.9450293964869\\
4	9.83518845192502\\
3	6.38758703786591\\
2	4.16226710046466\\
1	2.94235529784153\\
}--cycle;

\addplot [color=white!75!blue,solid,forget plot]
  table[row sep=crcr]{%
1	0.637765407216571\\
2	1.74729442796352\\
3	4.20606957689721\\
4	7.51170893572576\\
5	8.94408941844826\\
6	9.27839208667847\\
7	9.51154618851857\\
8	10.7723732540086\\
9	11.2762956390018\\
10	12.4108537389769\\
11	13.802741989056\\
12	14.9085716759943\\
13	15.0406164666811\\
14	15.0023627227979\\
15	16.5230055426455\\
16	16.9513907476588\\
17	15.9734669980193\\
18	16.2209263160402\\
19	15.851282837678\\
20	16.4142116429272\\
21	16.7397413879563\\
22	17.9160444350637\\
23	18.6812750289209\\
24	18.5898258205037\\
25	19.0552647277518\\
26	18.790420244255\\
27	18.7639027553167\\
28	17.3508077942345\\
29	17.180724344497\\
30	17.1900706758713\\
31	17.7100669695757\\
32	17.6282880126376\\
33	17.4002871336046\\
34	17.5248533453461\\
35	19.8614046268757\\
36	21.9326775792168\\
37	22.5955745911792\\
38	21.4074283188472\\
39	21.0214942246761\\
40	20.0431012193766\\
41	20.6204397329785\\
42	22.4244098923807\\
43	22.4963568597677\\
44	23.2694191729001\\
45	24.4419563914404\\
46	24.4947137786235\\
47	24.3293847516352\\
48	24.5855542445021\\
49	24.0792645812904\\
50	25.112604594374\\
51	24.5347242548015\\
52	25.1280988687296\\
53	26.5673863685156\\
54	27.7745692639493\\
55	26.2939619259438\\
56	25.0707719854052\\
57	23.7284175922147\\
58	24.7477277156912\\
59	25.5643340517051\\
60	24.2156496282998\\
61	23.9808632259933\\
62	24.1742410983378\\
63	24.7662080105491\\
64	27.9647353612692\\
65	27.3343133071876\\
66	27.3170167918774\\
67	27.8085206936683\\
68	27.2838775569305\\
69	26.5071018532624\\
70	27.2698299275252\\
71	29.3721490503335\\
72	27.5969913845541\\
73	28.5711337013707\\
74	28.5026706322823\\
75	29.5155920337463\\
76	27.918991060812\\
77	29.4246333428379\\
78	28.7068735371921\\
79	29.2331408177956\\
80	30.559620901651\\
81	30.4113143222669\\
82	28.1538729436245\\
83	27.1630424722158\\
84	25.7921737363665\\
85	25.2120798915719\\
86	22.5951904001462\\
87	20.9253634901123\\
88	21.8850135190077\\
89	23.2458020699987\\
90	24.5267051354751\\
91	26.2507069273301\\
92	27.3604659166043\\
93	28.4316129343528\\
94	26.5645378528735\\
95	28.0329018195697\\
96	28.7922959329063\\
97	25.7303002994928\\
98	23.4645293112346\\
99	22.4984570698576\\
100	22.5751006784584\\
};
\addplot [color=white!75!blue,solid,forget plot]
  table[row sep=crcr]{%
1	2.94235529784153\\
2	4.16226710046466\\
3	6.38758703786591\\
4	9.83518845192502\\
5	10.9450293964869\\
6	11.8596086582055\\
7	12.0863114282189\\
8	13.0655564570239\\
9	13.7356326802694\\
10	14.712424210761\\
11	16.5926892490734\\
12	17.1523137340407\\
13	17.6527944386637\\
14	17.4759234669588\\
15	19.0246886866667\\
16	19.7989735879974\\
17	18.9510678685218\\
18	19.1248033122157\\
19	18.5313058089593\\
20	18.9245458248563\\
21	19.8496655493627\\
22	20.2110979058949\\
23	20.8471120364943\\
24	20.7618659694124\\
25	22.0544106394365\\
26	22.0144818334013\\
27	21.7357335831719\\
28	20.3979968340455\\
29	19.8852893548938\\
30	20.209380867445\\
31	20.5807252125888\\
32	20.2490175703738\\
33	19.8988562182206\\
34	20.1652173073015\\
35	22.5149245949488\\
36	24.2967296810529\\
37	24.7099506961104\\
38	24.7580691994087\\
39	23.9201227535948\\
40	22.9431176516779\\
41	23.2903944107963\\
42	25.6148873282042\\
43	24.8201804391003\\
44	26.033424779863\\
45	27.0830460704878\\
46	27.4235884866189\\
47	27.1718390396631\\
48	27.5767868592125\\
49	26.769407377999\\
50	27.9149061434745\\
51	27.0827456303811\\
52	28.0724440680538\\
53	29.546700562271\\
54	31.3854403183979\\
55	29.0601492071641\\
56	27.9289857515764\\
57	27.0717546482607\\
58	27.5876945642438\\
59	28.4832465752829\\
60	26.9363709387639\\
61	26.7606590130734\\
62	26.5101455214785\\
63	28.2324342903004\\
64	30.8914665293294\\
65	30.1469892249765\\
66	29.8415419267657\\
67	30.1911737009587\\
68	30.1549992350807\\
69	29.8485194883365\\
70	29.8307509995372\\
71	32.0269211206484\\
72	30.5117409486279\\
73	31.4466784757668\\
74	32.142021733499\\
75	31.9807990134905\\
76	31.3536526543244\\
77	32.1108943867398\\
78	31.5711287407939\\
79	31.3667374115492\\
80	32.9341817468632\\
81	33.5765029188546\\
82	31.1633088442776\\
83	30.6497979001753\\
84	29.0960067083088\\
85	28.6064133521473\\
86	25.5084915353107\\
87	23.7127160115189\\
88	24.8016747453438\\
89	26.2562597793441\\
90	27.2315390373411\\
91	28.963304051307\\
92	30.0432492565513\\
93	30.9986034963281\\
94	29.8223910842493\\
95	30.6388403315829\\
96	31.8321678258068\\
97	29.2378976804989\\
98	26.3099607313803\\
99	25.7963563332271\\
100	25.3496511412019\\
};
\addplot [color=blue,solid]
  table[row sep=crcr]{%
1	1.79006035252905\\
2	2.95478076421409\\
3	5.29682830738156\\
4	8.67344869382539\\
5	9.94455940746757\\
6	10.569000372442\\
7	10.7989288083687\\
8	11.9189648555163\\
9	12.5059641596356\\
10	13.5616389748689\\
11	15.1977156190647\\
12	16.0304427050175\\
13	16.3467054526724\\
14	16.2391430948784\\
15	17.7738471146561\\
16	18.3751821678281\\
17	17.4622674332705\\
18	17.6728648141279\\
19	17.1912943233186\\
20	17.6693787338917\\
21	18.2947034686595\\
22	19.0635711704793\\
23	19.7641935327076\\
24	19.675845894958\\
25	20.5548376835941\\
26	20.4024510388282\\
27	20.2498181692443\\
28	18.87440231414\\
29	18.5330068496954\\
30	18.6997257716581\\
31	19.1453960910823\\
32	18.9386527915057\\
33	18.6495716759126\\
34	18.8450353263238\\
35	21.1881646109122\\
36	23.1147036301348\\
37	23.6527626436448\\
38	23.0827487591279\\
39	22.4708084891354\\
40	21.4931094355272\\
41	21.9554170718874\\
42	24.0196486102924\\
43	23.658268649434\\
44	24.6514219763815\\
45	25.7625012309641\\
46	25.9591511326212\\
47	25.7506118956492\\
48	26.0811705518573\\
49	25.4243359796447\\
50	26.5137553689242\\
51	25.8087349425913\\
52	26.6002714683917\\
53	28.0570434653933\\
54	29.5800047911736\\
55	27.6770555665539\\
56	26.4998788684908\\
57	25.4000861202377\\
58	26.1677111399675\\
59	27.023790313494\\
60	25.5760102835318\\
61	25.3707611195333\\
62	25.3421933099081\\
63	26.4993211504247\\
64	29.4281009452993\\
65	28.7406512660821\\
66	28.5792793593215\\
67	28.9998471973135\\
68	28.7194383960056\\
69	28.1778106707994\\
70	28.5502904635312\\
71	30.699535085491\\
72	29.054366166591\\
73	30.0089060885688\\
74	30.3223461828907\\
75	30.7481955236184\\
76	29.6363218575682\\
77	30.7677638647889\\
78	30.139001138993\\
79	30.2999391146724\\
80	31.7469013242571\\
81	31.9939086205608\\
82	29.6585908939511\\
83	28.9064201861956\\
84	27.4440902223376\\
85	26.9092466218596\\
86	24.0518409677284\\
87	22.3190397508156\\
88	23.3433441321757\\
89	24.7510309246714\\
90	25.8791220864081\\
91	27.6070054893185\\
92	28.7018575865778\\
93	29.7151082153405\\
94	28.1934644685614\\
95	29.3358710755763\\
96	30.3122318793566\\
97	27.4840989899958\\
98	24.8872450213075\\
99	24.1474067015423\\
100	23.9623759098302\\
};
\addlegendentry{CPHD};

\end{axis}
\end{tikzpicture}%

%% file: card_mean_L2.tikz
%
\begin{tikzpicture}

\begin{axis}[%
width=0.8\linewidth,
height=1.2in,
at={(2.6in,1.114896in)},
scale only axis,
xmin=1,
xmax=100,
xlabel={time},
ymin=0,
ymax=35,
ylabel={\# targets},
axis x line*=bottom,
axis y line*=left,
legend style={at={(0.97,0.03)},anchor=south east,legend cell align=left,align=left,draw=white!15!black,fill opacity=0.6,draw opacity=1,text opacity=1,font=\footnotesize}
]
\addplot [color=black,solid]
  table[row sep=crcr]{%
1	1\\
2	3\\
3	7\\
4	8\\
5	8\\
6	9\\
7	10\\
8	10\\
9	12\\
10	13\\
11	15\\
12	14\\
13	14\\
14	16\\
15	16\\
16	17\\
17	15\\
18	17\\
19	15\\
20	17\\
21	17\\
22	17\\
23	20\\
24	19\\
25	20\\
26	18\\
27	18\\
28	16\\
29	16\\
30	18\\
31	18\\
32	18\\
33	17\\
34	19\\
35	23\\
36	22\\
37	24\\
38	21\\
39	20\\
40	19\\
41	23\\
42	24\\
43	23\\
44	25\\
45	25\\
46	24\\
47	26\\
48	24\\
49	24\\
50	28\\
51	24\\
52	27\\
53	28\\
54	28\\
55	27\\
56	24\\
57	24\\
58	25\\
59	25\\
60	23\\
61	24\\
62	24\\
63	29\\
64	28\\
65	27\\
66	28\\
67	27\\
68	28\\
69	27\\
70	31\\
71	29\\
72	29\\
73	30\\
74	30\\
75	29\\
76	29\\
77	29\\
78	30\\
79	31\\
80	32\\
81	30\\
82	28\\
83	27\\
84	26\\
85	25\\
86	22\\
87	21\\
88	25\\
89	25\\
90	27\\
91	27\\
92	29\\
93	30\\
94	28\\
95	29\\
96	29\\
97	25\\
98	23\\
99	24\\
100	23\\
};
\addlegendentry{gt};

\addplot[area legend,solid,fill=red,opacity=1.000000e-01,draw=none,forget plot]
table[row sep=crcr] {%
x	y\\
1	0.635009990479417\\
2	1.73052712577469\\
3	4.11161016810285\\
4	7.45882649902604\\
5	8.61820993426178\\
6	8.42137186756072\\
7	8.56076490195658\\
8	10.3629971653115\\
9	10.5693132055395\\
10	11.7526963798464\\
11	13.287008316435\\
12	14.2843249316764\\
13	14.2091968919201\\
14	13.8677552011717\\
15	16.0917800767698\\
16	16.2409807478421\\
17	14.9970056986662\\
18	15.2369479192065\\
19	15.0084094222556\\
20	15.5359319738811\\
21	16.072457069125\\
22	17.2605275870349\\
23	18.0366558523852\\
24	17.633668805894\\
25	18.314733978739\\
26	17.9815731558373\\
27	17.880442774237\\
28	16.1598678910785\\
29	16.2272768416583\\
30	16.2983406676058\\
31	16.9303343485871\\
32	16.3487957160341\\
33	16.2440461202058\\
34	16.5267123258402\\
35	19.2807287975486\\
36	21.5262907168587\\
37	21.7879435600699\\
38	20.4302781966101\\
39	20.1783610680411\\
40	19.0321095185422\\
41	19.6866394986729\\
42	21.9184866238618\\
43	21.4971312467684\\
44	22.5263576818284\\
45	23.8474116868992\\
46	23.5778756417933\\
47	23.2576226284224\\
48	24.0227059852768\\
49	23.2156857244695\\
50	24.4942749400409\\
51	23.5419515763289\\
52	24.3607427050641\\
53	26.0979099842974\\
54	27.270449246284\\
55	24.6471531467478\\
56	24.0618168916004\\
57	22.4560525386708\\
58	24.1094819454632\\
59	25.0561017156798\\
60	23.0343242063353\\
61	23.1305936246858\\
62	23.4155558295809\\
63	23.884749964456\\
64	27.5542335685913\\
65	26.0010248942053\\
66	26.3030056794097\\
67	26.9881863999784\\
68	26.158588206323\\
69	25.4729393776226\\
70	26.358011739792\\
71	28.9396342348353\\
72	26.0988558295383\\
73	27.8354492809325\\
74	27.6418166713014\\
75	28.7562007707769\\
76	26.6104069012923\\
77	28.8719458567319\\
78	27.4639894146791\\
79	28.2063146967407\\
80	29.3523969771416\\
81	28.7151187941006\\
82	26.3791125339884\\
83	25.9875179712827\\
84	24.5913195986548\\
85	24.1509852022613\\
86	20.9688973509211\\
87	19.8880089195989\\
88	20.983977071344\\
89	22.5716059391429\\
90	23.9706825801958\\
91	25.8290540848757\\
92	26.7907296324848\\
93	28.0281307580901\\
94	25.3835492237315\\
95	27.4669721738468\\
96	28.3230911455484\\
97	24.4251642978483\\
98	22.2002027107435\\
99	21.7026643798844\\
100	21.7274433534207\\
100	24.7141760506011\\
99	25.0306292254147\\
98	24.9170047367409\\
97	27.8417941777586\\
96	31.7100010357841\\
95	30.588452680332\\
94	29.0175745698344\\
93	30.8330861810182\\
92	29.9047543889897\\
91	28.847616398964\\
90	26.8382451565747\\
89	26.0438889596924\\
88	24.5366052829052\\
87	22.4739545790189\\
86	23.7575116496047\\
85	28.0197570324886\\
84	27.9483418631114\\
83	29.8269424541295\\
82	29.5000171614809\\
81	32.6191428458952\\
80	32.0554838321351\\
79	30.5574604630136\\
78	30.4427045087731\\
77	31.6165350231551\\
76	30.4949142389987\\
75	31.5094074037459\\
74	31.6589062420753\\
73	31.2736427745507\\
72	29.4228572827684\\
71	31.8185120417582\\
70	29.2718258996342\\
69	29.1718247626848\\
68	29.5528385400853\\
67	29.5920690369759\\
66	29.1288229230899\\
65	29.3207428040339\\
64	30.8786370493587\\
63	28.1913475992234\\
62	25.8779410716943\\
61	26.1232404895939\\
60	26.0308455212754\\
59	28.290480222625\\
58	27.2722736380118\\
57	26.490815640093\\
56	27.0277527161723\\
55	27.7080503885731\\
54	31.5128869029345\\
53	29.4772948228143\\
52	27.7982425281037\\
51	26.2564174776936\\
50	27.6877763008898\\
49	25.9621948771477\\
48	27.0442660288796\\
47	26.4911477295235\\
46	27.0073413391418\\
45	26.8385814432184\\
44	25.7317568457749\\
43	24.036768643003\\
42	25.6388098587894\\
41	22.7538050433638\\
40	22.103993487681\\
39	23.0660966382881\\
38	24.1815289391086\\
37	24.3281593319368\\
36	24.2322370395354\\
35	22.434001925402\\
34	19.3107489125691\\
33	19.0577125978791\\
32	19.4156092555268\\
31	20.0970052558862\\
30	19.6509333160009\\
29	19.1475087895636\\
28	19.1371694618248\\
27	21.2595212932511\\
26	21.4131686359873\\
25	21.7245041238315\\
24	20.1126452850429\\
23	20.2828820171778\\
22	19.8577623924178\\
21	19.4376338089311\\
20	18.3478599996739\\
19	17.7855025260655\\
18	18.5250416577389\\
17	17.91037617138\\
16	19.2985762618173\\
15	18.7814695081866\\
14	16.5929788548152\\
13	17.1048494048962\\
12	16.6982971327614\\
11	16.3998806193916\\
10	14.3633612014916\\
9	13.1718176297727\\
8	12.7632273597592\\
7	11.4579594183884\\
6	11.45659280636\\
5	10.6155538161732\\
4	9.90205267205593\\
3	6.31722599809335\\
2	4.12027668195063\\
1	2.66744314359553\\
}--cycle;

\addplot [color=pink,solid,forget plot]
  table[row sep=crcr]{%
1	0.635009990479417\\
2	1.73052712577469\\
3	4.11161016810285\\
4	7.45882649902604\\
5	8.61820993426178\\
6	8.42137186756072\\
7	8.56076490195658\\
8	10.3629971653115\\
9	10.5693132055395\\
10	11.7526963798464\\
11	13.287008316435\\
12	14.2843249316764\\
13	14.2091968919201\\
14	13.8677552011717\\
15	16.0917800767698\\
16	16.2409807478421\\
17	14.9970056986662\\
18	15.2369479192065\\
19	15.0084094222556\\
20	15.5359319738811\\
21	16.072457069125\\
22	17.2605275870349\\
23	18.0366558523852\\
24	17.633668805894\\
25	18.314733978739\\
26	17.9815731558373\\
27	17.880442774237\\
28	16.1598678910785\\
29	16.2272768416583\\
30	16.2983406676058\\
31	16.9303343485871\\
32	16.3487957160341\\
33	16.2440461202058\\
34	16.5267123258402\\
35	19.2807287975486\\
36	21.5262907168587\\
37	21.7879435600699\\
38	20.4302781966101\\
39	20.1783610680411\\
40	19.0321095185422\\
41	19.6866394986729\\
42	21.9184866238618\\
43	21.4971312467684\\
44	22.5263576818284\\
45	23.8474116868992\\
46	23.5778756417933\\
47	23.2576226284224\\
48	24.0227059852768\\
49	23.2156857244695\\
50	24.4942749400409\\
51	23.5419515763289\\
52	24.3607427050641\\
53	26.0979099842974\\
54	27.270449246284\\
55	24.6471531467478\\
56	24.0618168916004\\
57	22.4560525386708\\
58	24.1094819454632\\
59	25.0561017156798\\
60	23.0343242063353\\
61	23.1305936246858\\
62	23.4155558295809\\
63	23.884749964456\\
64	27.5542335685913\\
65	26.0010248942053\\
66	26.3030056794097\\
67	26.9881863999784\\
68	26.158588206323\\
69	25.4729393776226\\
70	26.358011739792\\
71	28.9396342348353\\
72	26.0988558295383\\
73	27.8354492809325\\
74	27.6418166713014\\
75	28.7562007707769\\
76	26.6104069012923\\
77	28.8719458567319\\
78	27.4639894146791\\
79	28.2063146967407\\
80	29.3523969771416\\
81	28.7151187941006\\
82	26.3791125339884\\
83	25.9875179712827\\
84	24.5913195986548\\
85	24.1509852022613\\
86	20.9688973509211\\
87	19.8880089195989\\
88	20.983977071344\\
89	22.5716059391429\\
90	23.9706825801958\\
91	25.8290540848757\\
92	26.7907296324848\\
93	28.0281307580901\\
94	25.3835492237315\\
95	27.4669721738468\\
96	28.3230911455484\\
97	24.4251642978483\\
98	22.2002027107435\\
99	21.7026643798844\\
100	21.7274433534207\\
};
\addplot [color=pink,solid,forget plot]
  table[row sep=crcr]{%
1	2.66744314359553\\
2	4.12027668195063\\
3	6.31722599809335\\
4	9.90205267205593\\
5	10.6155538161732\\
6	11.45659280636\\
7	11.4579594183884\\
8	12.7632273597592\\
9	13.1718176297727\\
10	14.3633612014916\\
11	16.3998806193916\\
12	16.6982971327614\\
13	17.1048494048962\\
14	16.5929788548152\\
15	18.7814695081866\\
16	19.2985762618173\\
17	17.91037617138\\
18	18.5250416577389\\
19	17.7855025260655\\
20	18.3478599996739\\
21	19.4376338089311\\
22	19.8577623924178\\
23	20.2828820171778\\
24	20.1126452850429\\
25	21.7245041238315\\
26	21.4131686359873\\
27	21.2595212932511\\
28	19.1371694618248\\
29	19.1475087895636\\
30	19.6509333160009\\
31	20.0970052558862\\
32	19.4156092555268\\
33	19.0577125978791\\
34	19.3107489125691\\
35	22.434001925402\\
36	24.2322370395354\\
37	24.3281593319368\\
38	24.1815289391086\\
39	23.0660966382881\\
40	22.103993487681\\
41	22.7538050433638\\
42	25.6388098587894\\
43	24.036768643003\\
44	25.7317568457749\\
45	26.8385814432184\\
46	27.0073413391418\\
47	26.4911477295235\\
48	27.0442660288796\\
49	25.9621948771477\\
50	27.6877763008898\\
51	26.2564174776936\\
52	27.7982425281037\\
53	29.4772948228143\\
54	31.5128869029345\\
55	27.7080503885731\\
56	27.0277527161723\\
57	26.490815640093\\
58	27.2722736380118\\
59	28.290480222625\\
60	26.0308455212754\\
61	26.1232404895939\\
62	25.8779410716943\\
63	28.1913475992234\\
64	30.8786370493587\\
65	29.3207428040339\\
66	29.1288229230899\\
67	29.5920690369759\\
68	29.5528385400853\\
69	29.1718247626848\\
70	29.2718258996342\\
71	31.8185120417582\\
72	29.4228572827684\\
73	31.2736427745507\\
74	31.6589062420753\\
75	31.5094074037459\\
76	30.4949142389987\\
77	31.6165350231551\\
78	30.4427045087731\\
79	30.5574604630136\\
80	32.0554838321351\\
81	32.6191428458952\\
82	29.5000171614809\\
83	29.8269424541295\\
84	27.9483418631114\\
85	28.0197570324886\\
86	23.7575116496047\\
87	22.4739545790189\\
88	24.5366052829052\\
89	26.0438889596924\\
90	26.8382451565747\\
91	28.847616398964\\
92	29.9047543889897\\
93	30.8330861810182\\
94	29.0175745698344\\
95	30.588452680332\\
96	31.7100010357841\\
97	27.8417941777586\\
98	24.9170047367409\\
99	25.0306292254147\\
100	24.7141760506011\\
};
\addplot [color=red,dashed]
  table[row sep=crcr]{%
1	1.65122656703748\\
2	2.92540190386266\\
3	5.2144180830981\\
4	8.68043958554099\\
5	9.61688187521747\\
6	9.93898233696037\\
7	10.0093621601725\\
8	11.5631122625353\\
9	11.8705654176561\\
10	13.058028790669\\
11	14.8434444679133\\
12	15.4913110322189\\
13	15.6570231484082\\
14	15.2303670279934\\
15	17.4366247924782\\
16	17.7697785048297\\
17	16.4536909350231\\
18	16.8809947884727\\
19	16.3969559741606\\
20	16.9418959867775\\
21	17.7550454390281\\
22	18.5591449897264\\
23	19.1597689347815\\
24	18.8731570454684\\
25	20.0196190512852\\
26	19.6973708959123\\
27	19.569982033744\\
28	17.6485186764517\\
29	17.687392815611\\
30	17.9746369918033\\
31	18.5136698022366\\
32	17.8822024857804\\
33	17.6508793590424\\
34	17.9187306192046\\
35	20.8573653614753\\
36	22.8792638781971\\
37	23.0580514460033\\
38	22.3059035678593\\
39	21.6222288531646\\
40	20.5680515031116\\
41	21.2202222710183\\
42	23.7786482413256\\
43	22.7669499448857\\
44	24.1290572638016\\
45	25.3429965650588\\
46	25.2926084904675\\
47	24.8743851789729\\
48	25.5334860070782\\
49	24.5889403008086\\
50	26.0910256204653\\
51	24.8991845270113\\
52	26.0794926165839\\
53	27.7876024035559\\
54	29.3916680746093\\
55	26.1776017676605\\
56	25.5447848038863\\
57	24.4734340893819\\
58	25.6908777917375\\
59	26.6732909691524\\
60	24.5325848638054\\
61	24.6269170571398\\
62	24.6467484506376\\
63	26.0380487818397\\
64	29.216435308975\\
65	27.6608838491196\\
66	27.7159143012498\\
67	28.2901277184772\\
68	27.8557133732041\\
69	27.3223820701537\\
70	27.8149188197131\\
71	30.3790731382967\\
72	27.7608565561534\\
73	29.5545460277416\\
74	29.6503614566883\\
75	30.1328040872614\\
76	28.5526605701455\\
77	30.2442404399435\\
78	28.9533469617261\\
79	29.3818875798771\\
80	30.7039404046384\\
81	30.6671308199979\\
82	27.9395648477347\\
83	27.9072302127061\\
84	26.2698307308831\\
85	26.085371117375\\
86	22.3632045002629\\
87	21.1809817493089\\
88	22.7602911771246\\
89	24.3077474494177\\
90	25.4044638683853\\
91	27.3383352419199\\
92	28.3477420107372\\
93	29.4306084695541\\
94	27.200561896783\\
95	29.0277124270894\\
96	30.0165460906662\\
97	26.1334792378035\\
98	23.5586037237422\\
99	23.3666468026496\\
100	23.2208097020109\\
};
\addlegendentry{PHD};

\addplot[area legend,solid,fill=green,opacity=1.000000e-01,draw=none,forget plot]
table[row sep=crcr] {%
x	y\\
1	0.0055175173070529\\
2	0.26152176925676\\
3	2.25687087216789\\
4	5.51760426500984\\
5	7.51404081958116\\
6	7.47692046980827\\
7	7.12631947367771\\
8	9.06619040858811\\
9	9.56882492418429\\
10	9.98327409084021\\
11	11.3728539576928\\
12	12.1730820517501\\
13	12.9825695171972\\
14	12.5029131316912\\
15	14.333154947399\\
16	15.1148776857732\\
17	13.8494328772462\\
18	13.7990173201214\\
19	13.5465202099728\\
20	14.1536509440073\\
21	14.6746460998963\\
22	15.6570395437117\\
23	16.3038735522349\\
24	16.0132389551492\\
25	16.6075525312684\\
26	16.8176964034062\\
27	16.9498441120062\\
28	15.2396027373674\\
29	15.0396749687399\\
30	15.1390459374481\\
31	15.3708344682253\\
32	14.5653273341859\\
33	14.4512028010202\\
34	14.7640201640722\\
35	16.7200587333248\\
36	18.8180337432531\\
37	19.7499793583143\\
38	19.0105510812751\\
39	18.8849163064343\\
40	18.1971765388271\\
41	18.0133872423343\\
42	19.3413602947482\\
43	19.571242588837\\
44	20.1834481244341\\
45	21.5771812665405\\
46	21.8054549477218\\
47	21.5963285722299\\
48	22.007915473032\\
49	21.8609326853876\\
50	22.7141250941476\\
51	21.482590664008\\
52	22.5650047178827\\
53	24.0437209573164\\
54	25.294426014482\\
55	23.1867146011783\\
56	22.4903071199075\\
57	21.3402325186838\\
58	22.3339851904798\\
59	23.6160480455557\\
60	22.0638426834225\\
61	21.7257040006921\\
62	21.7508978070667\\
63	21.934594513001\\
64	24.1032937944476\\
65	24.3584002868884\\
66	24.4956627451846\\
67	24.9316951208229\\
68	24.7779793720567\\
69	23.6847628458989\\
70	23.953727664678\\
71	25.8109232721715\\
72	24.5833923870887\\
73	25.65090385728\\
74	25.6854463383683\\
75	26.4969312488367\\
76	25.4305079878718\\
77	26.7942072332303\\
78	25.9058965768972\\
79	25.7219903993395\\
80	26.482683712103\\
81	26.3775571067492\\
82	25.1902418333083\\
83	24.9294728085349\\
84	23.7200017702379\\
85	22.9243267423707\\
86	19.8929660097152\\
87	18.8735407380927\\
88	19.1520080637797\\
89	20.2807624672929\\
90	21.6220066525235\\
91	23.1112549341947\\
92	25.0469467800859\\
93	26.1927410522775\\
94	23.8689528216\\
95	25.5102899448241\\
96	26.7999364743354\\
97	23.6729961166723\\
98	21.3639944416416\\
99	20.5258185312327\\
100	19.8267704040455\\
100	22.3248593498009\\
99	23.1213083453471\\
98	23.5461295917518\\
97	26.5374283952446\\
96	29.4426150804657\\
95	28.1837705135262\\
94	26.5340119485324\\
93	28.5546835697357\\
92	27.3954987807645\\
91	25.6187244544001\\
90	23.9416705680362\\
89	22.34065224508\\
88	21.7676880948349\\
87	20.7492278142606\\
86	22.4102035010577\\
85	25.9338890852471\\
84	26.2831494536897\\
83	28.2185336417876\\
82	27.6162417716186\\
81	29.3102600184105\\
80	28.2840441813731\\
79	27.6191738216981\\
78	28.144567276944\\
77	28.8562608207795\\
76	28.3265841966376\\
75	28.9866035343311\\
74	28.4785263669169\\
73	27.9328270631749\\
72	27.2497745074873\\
71	28.0780589056457\\
70	26.8074653212225\\
69	26.6289168934985\\
68	27.081852981966\\
67	26.8788817664337\\
66	26.6943998682669\\
65	26.7137736515784\\
64	26.5174769664917\\
63	25.379034055412\\
62	23.8570389746269\\
61	24.445792750988\\
60	24.5341483718983\\
59	26.0797387945651\\
58	25.0033137010682\\
57	24.5850647645121\\
56	24.4976863274871\\
55	25.698434250349\\
54	28.4362661550725\\
53	26.6390305702396\\
52	24.9647857067888\\
51	23.592540250169\\
50	25.0551361895839\\
49	24.1797496148285\\
48	24.5814392423415\\
47	24.3050687224577\\
46	24.8797549492378\\
45	23.7009116706763\\
44	22.5917695071235\\
43	21.6183751523008\\
42	22.08803136444\\
41	20.3817764330619\\
40	20.3764970532373\\
39	21.3609468755927\\
38	21.6972365973379\\
37	21.7845146518301\\
36	20.9104533539139\\
35	18.5759623299419\\
34	16.5957443758242\\
33	16.3293451121043\\
32	16.882416138059\\
31	17.338617992254\\
30	17.363368040659\\
29	17.3031211799637\\
28	17.3344680302455\\
27	19.3166114092176\\
26	19.1727661599226\\
25	18.6263768889724\\
24	17.5504940163392\\
23	18.2229764211614\\
22	17.8190481339258\\
21	16.6526977051011\\
20	16.142891213086\\
19	15.7677339157162\\
18	16.4484999802278\\
17	15.9151127321941\\
16	16.9177601308817\\
15	15.8602123904546\\
14	14.3257250567417\\
13	14.8556992225975\\
12	13.9823767106593\\
11	13.4345930972353\\
10	11.616174551586\\
9	10.9911576586621\\
8	10.4723847557837\\
7	9.12475284540792\\
6	9.14866236291484\\
5	8.37316509441741\\
4	6.48712674825046\\
3	3.0178880280766\\
2	0.550453330490119\\
1	0.0192244830001345\\
}--cycle;

\addplot [color=white!75!green,solid,forget plot]
  table[row sep=crcr]{%
1	0.0055175173070529\\
2	0.26152176925676\\
3	2.25687087216789\\
4	5.51760426500984\\
5	7.51404081958116\\
6	7.47692046980827\\
7	7.12631947367771\\
8	9.06619040858811\\
9	9.56882492418429\\
10	9.98327409084021\\
11	11.3728539576928\\
12	12.1730820517501\\
13	12.9825695171972\\
14	12.5029131316912\\
15	14.333154947399\\
16	15.1148776857732\\
17	13.8494328772462\\
18	13.7990173201214\\
19	13.5465202099728\\
20	14.1536509440073\\
21	14.6746460998963\\
22	15.6570395437117\\
23	16.3038735522349\\
24	16.0132389551492\\
25	16.6075525312684\\
26	16.8176964034062\\
27	16.9498441120062\\
28	15.2396027373674\\
29	15.0396749687399\\
30	15.1390459374481\\
31	15.3708344682253\\
32	14.5653273341859\\
33	14.4512028010202\\
34	14.7640201640722\\
35	16.7200587333248\\
36	18.8180337432531\\
37	19.7499793583143\\
38	19.0105510812751\\
39	18.8849163064343\\
40	18.1971765388271\\
41	18.0133872423343\\
42	19.3413602947482\\
43	19.571242588837\\
44	20.1834481244341\\
45	21.5771812665405\\
46	21.8054549477218\\
47	21.5963285722299\\
48	22.007915473032\\
49	21.8609326853876\\
50	22.7141250941476\\
51	21.482590664008\\
52	22.5650047178827\\
53	24.0437209573164\\
54	25.294426014482\\
55	23.1867146011783\\
56	22.4903071199075\\
57	21.3402325186838\\
58	22.3339851904798\\
59	23.6160480455557\\
60	22.0638426834225\\
61	21.7257040006921\\
62	21.7508978070667\\
63	21.934594513001\\
64	24.1032937944476\\
65	24.3584002868884\\
66	24.4956627451846\\
67	24.9316951208229\\
68	24.7779793720567\\
69	23.6847628458989\\
70	23.953727664678\\
71	25.8109232721715\\
72	24.5833923870887\\
73	25.65090385728\\
74	25.6854463383683\\
75	26.4969312488367\\
76	25.4305079878718\\
77	26.7942072332303\\
78	25.9058965768972\\
79	25.7219903993395\\
80	26.482683712103\\
81	26.3775571067492\\
82	25.1902418333083\\
83	24.9294728085349\\
84	23.7200017702379\\
85	22.9243267423707\\
86	19.8929660097152\\
87	18.8735407380927\\
88	19.1520080637797\\
89	20.2807624672929\\
90	21.6220066525235\\
91	23.1112549341947\\
92	25.0469467800859\\
93	26.1927410522775\\
94	23.8689528216\\
95	25.5102899448241\\
96	26.7999364743354\\
97	23.6729961166723\\
98	21.3639944416416\\
99	20.5258185312327\\
100	19.8267704040455\\
};
\addplot [color=white!75!green,solid,forget plot]
  table[row sep=crcr]{%
1	0.0192244830001345\\
2	0.550453330490119\\
3	3.0178880280766\\
4	6.48712674825046\\
5	8.37316509441741\\
6	9.14866236291484\\
7	9.12475284540792\\
8	10.4723847557837\\
9	10.9911576586621\\
10	11.616174551586\\
11	13.4345930972353\\
12	13.9823767106593\\
13	14.8556992225975\\
14	14.3257250567417\\
15	15.8602123904546\\
16	16.9177601308817\\
17	15.9151127321941\\
18	16.4484999802278\\
19	15.7677339157162\\
20	16.142891213086\\
21	16.6526977051011\\
22	17.8190481339258\\
23	18.2229764211614\\
24	17.5504940163392\\
25	18.6263768889724\\
26	19.1727661599226\\
27	19.3166114092176\\
28	17.3344680302455\\
29	17.3031211799637\\
30	17.363368040659\\
31	17.338617992254\\
32	16.882416138059\\
33	16.3293451121043\\
34	16.5957443758242\\
35	18.5759623299419\\
36	20.9104533539139\\
37	21.7845146518301\\
38	21.6972365973379\\
39	21.3609468755927\\
40	20.3764970532373\\
41	20.3817764330619\\
42	22.08803136444\\
43	21.6183751523008\\
44	22.5917695071235\\
45	23.7009116706763\\
46	24.8797549492378\\
47	24.3050687224577\\
48	24.5814392423415\\
49	24.1797496148285\\
50	25.0551361895839\\
51	23.592540250169\\
52	24.9647857067888\\
53	26.6390305702396\\
54	28.4362661550725\\
55	25.698434250349\\
56	24.4976863274871\\
57	24.5850647645121\\
58	25.0033137010682\\
59	26.0797387945651\\
60	24.5341483718983\\
61	24.445792750988\\
62	23.8570389746269\\
63	25.379034055412\\
64	26.5174769664917\\
65	26.7137736515784\\
66	26.6943998682669\\
67	26.8788817664337\\
68	27.081852981966\\
69	26.6289168934985\\
70	26.8074653212225\\
71	28.0780589056457\\
72	27.2497745074873\\
73	27.9328270631749\\
74	28.4785263669169\\
75	28.9866035343311\\
76	28.3265841966376\\
77	28.8562608207795\\
78	28.144567276944\\
79	27.6191738216981\\
80	28.2840441813731\\
81	29.3102600184105\\
82	27.6162417716186\\
83	28.2185336417876\\
84	26.2831494536897\\
85	25.9338890852471\\
86	22.4102035010577\\
87	20.7492278142606\\
88	21.7676880948349\\
89	22.34065224508\\
90	23.9416705680362\\
91	25.6187244544001\\
92	27.3954987807645\\
93	28.5546835697357\\
94	26.5340119485324\\
95	28.1837705135262\\
96	29.4426150804657\\
97	26.5374283952446\\
98	23.5461295917518\\
99	23.1213083453471\\
100	22.3248593498009\\
};
\addplot [color=green,dashed]
  table[row sep=crcr]{%
1	0.0123710001535937\\
2	0.40598754987344\\
3	2.63737945012225\\
4	6.00236550663015\\
5	7.94360295699928\\
6	8.31279141636156\\
7	8.12553615954281\\
8	9.76928758218591\\
9	10.2799912914232\\
10	10.7997243212131\\
11	12.4037235274641\\
12	13.0777293812047\\
13	13.9191343698974\\
14	13.4143190942164\\
15	15.0966836689268\\
16	16.0163189083275\\
17	14.8822728047201\\
18	15.1237586501746\\
19	14.6571270628445\\
20	15.1482710785467\\
21	15.6636719024987\\
22	16.7380438388187\\
23	17.2634249866982\\
24	16.7818664857442\\
25	17.6169647101204\\
26	17.9952312816644\\
27	18.1332277606119\\
28	16.2870353838065\\
29	16.1713980743518\\
30	16.2512069890536\\
31	16.3547262302396\\
32	15.7238717361224\\
33	15.3902739565623\\
34	15.6798822699482\\
35	17.6480105316334\\
36	19.8642435485835\\
37	20.7672470050722\\
38	20.3538938393065\\
39	20.1229315910135\\
40	19.2868367960322\\
41	19.1975818376981\\
42	20.7146958295941\\
43	20.5948088705689\\
44	21.3876088157788\\
45	22.6390464686084\\
46	23.3426049484798\\
47	22.9506986473438\\
48	23.2946773576867\\
49	23.020341150108\\
50	23.8846306418657\\
51	22.5375654570885\\
52	23.7648952123357\\
53	25.341375763778\\
54	26.8653460847772\\
55	24.4425744257637\\
56	23.4939967236973\\
57	22.9626486415979\\
58	23.668649445774\\
59	24.8478934200604\\
60	23.2989955276604\\
61	23.08574837584\\
62	22.8039683908468\\
63	23.6568142842065\\
64	25.3103853804697\\
65	25.5360869692334\\
66	25.5950313067257\\
67	25.9052884436283\\
68	25.9299161770113\\
69	25.1568398696987\\
70	25.3805964929503\\
71	26.9444910889086\\
72	25.916583447288\\
73	26.7918654602275\\
74	27.0819863526426\\
75	27.7417673915839\\
76	26.8785460922547\\
77	27.8252340270049\\
78	27.0252319269206\\
79	26.6705821105188\\
80	27.3833639467381\\
81	27.8439085625798\\
82	26.4032418024634\\
83	26.5740032251612\\
84	25.0015756119638\\
85	24.4291079138089\\
86	21.1515847553864\\
87	19.8113842761767\\
88	20.4598480793073\\
89	21.3107073561865\\
90	22.7818386102798\\
91	24.3649896942974\\
92	26.2212227804252\\
93	27.3737123110066\\
94	25.2014823850662\\
95	26.8470302291751\\
96	28.1212757774005\\
97	25.1052122559584\\
98	22.4550620166967\\
99	21.8235634382899\\
100	21.0758148769232\\
};
\addlegendentry{Panjer};

\addplot[area legend,solid,fill=blue,opacity=1.000000e-01,draw=none,forget plot]
table[row sep=crcr] {%
x	y\\
1	0.637765407216566\\
2	1.75821665734828\\
3	4.22448366883149\\
4	7.53834475757903\\
5	8.93437010383099\\
6	9.27723698550747\\
7	9.51484471303712\\
8	10.7631756247923\\
9	11.2798950796265\\
10	12.4219088260046\\
11	13.816118851894\\
12	14.8302050466096\\
13	14.9811740940973\\
14	14.9549204292727\\
15	16.5150445808219\\
16	16.9293213521845\\
17	15.9313874435707\\
18	16.178917016631\\
19	15.8117886086493\\
20	16.4089558533962\\
21	16.7169013486018\\
22	17.891686875753\\
23	18.6875894524288\\
24	18.5645428084025\\
25	19.0285899301553\\
26	18.7787851358048\\
27	18.7534015930443\\
28	17.3394055175201\\
29	17.1657035920476\\
30	17.1833675517804\\
31	17.6627738497803\\
32	17.4911505255754\\
33	17.2346498217895\\
34	17.3732051020901\\
35	19.6903112306307\\
36	21.8607386883626\\
37	22.5500896287786\\
38	21.4082570328725\\
39	20.9950795834684\\
40	20.0462359092475\\
41	20.5950655056293\\
42	22.3779893767237\\
43	22.4069910923872\\
44	23.2110512638861\\
45	24.3303104908593\\
46	24.4003913254158\\
47	24.2744087535418\\
48	24.6061759139027\\
49	24.0696357303121\\
50	25.0701552656363\\
51	24.512988971149\\
52	25.0983300622827\\
53	26.5055059487003\\
54	27.724996931309\\
55	26.243904298058\\
56	25.051342287049\\
57	23.6836338524075\\
58	24.7093006905013\\
59	25.5614276848235\\
60	24.2114850045419\\
61	23.9711773929707\\
62	24.1084887602594\\
63	24.7189880794218\\
64	27.7567712158059\\
65	27.1093425527696\\
66	27.0888378341886\\
67	27.5511043471614\\
68	27.0669115174562\\
69	26.3106279570434\\
70	27.0969856853487\\
71	29.1512039845696\\
72	27.5000545283356\\
73	28.4733433597024\\
74	28.4057068958373\\
75	29.3398812423512\\
76	27.763385844161\\
77	29.2428138723087\\
78	28.5183235622843\\
79	28.8795239651866\\
80	29.9525636975736\\
81	29.6320959854861\\
82	27.6470803106091\\
83	26.8840310125002\\
84	25.693730483722\\
85	25.0554223160995\\
86	22.5313917879154\\
87	20.8537003297476\\
88	21.7476204388717\\
89	23.0633574800864\\
90	24.4319264415415\\
91	26.1500196228324\\
92	27.3204232952247\\
93	28.4186718255671\\
94	26.5799735388622\\
95	27.946124363353\\
96	28.7687807898257\\
97	25.7031286312379\\
98	23.5195141077821\\
99	22.5832099822186\\
100	22.5381035564667\\
100	25.2191970367007\\
99	25.6650173456401\\
98	26.2635897723681\\
97	29.2087441576965\\
96	31.7467524429032\\
95	30.6456898665351\\
94	29.7433770400123\\
93	30.9223959091637\\
92	29.9230860298473\\
91	28.8219000237104\\
90	26.9994978829302\\
89	26.0599202589617\\
88	24.7052069231204\\
87	23.5851329441327\\
86	25.3866327753134\\
85	28.469891905856\\
84	28.8049868390147\\
83	30.266036479628\\
82	30.7130750764696\\
81	32.8792271823612\\
80	32.1537491076573\\
79	30.9600941472087\\
78	31.138111500091\\
77	31.7917555714895\\
76	31.1378394795213\\
75	31.7560288556049\\
74	31.9161868390768\\
73	31.2757771531191\\
72	30.354280136065\\
71	31.7667804467182\\
70	29.7008855042964\\
69	29.6514415368058\\
68	29.9891075934636\\
67	29.910281243425\\
66	29.6075389820168\\
65	29.8630132962358\\
64	30.5576930808279\\
63	28.1335280011737\\
62	26.4706599440423\\
61	26.6893260603812\\
60	26.8482878613633\\
59	28.3825326186558\\
58	27.5313792301665\\
57	27.0449958438479\\
56	27.7934788486073\\
55	28.986113345523\\
54	31.3192027980137\\
53	29.4630031738579\\
52	27.9734074686859\\
51	27.015838637313\\
50	27.859118164747\\
49	26.6962882571605\\
48	27.382880937843\\
47	27.0774721519173\\
46	27.2304118615701\\
45	26.9275724917542\\
44	25.8584574905477\\
43	24.697036178301\\
42	25.446927778924\\
41	23.2566275484537\\
40	22.8939790729886\\
39	23.8539083601986\\
38	24.6672521596584\\
37	24.6333017207692\\
36	24.1321257549712\\
35	22.2331950306409\\
34	19.904924968897\\
33	19.7475789654787\\
32	20.1494606850085\\
31	20.5013445119745\\
30	20.2169931503567\\
29	19.8639521593398\\
28	20.3895572464801\\
27	21.7185549215429\\
26	21.9005259364269\\
25	21.9465315412466\\
24	20.6875886485008\\
23	20.7628667486071\\
22	20.1425948867238\\
21	19.8005590332839\\
20	18.8108434262862\\
19	18.4945426984265\\
18	19.099455859028\\
17	18.9107330914249\\
16	19.7042755570942\\
15	18.8803488206403\\
14	17.4138263494375\\
13	17.5933363760415\\
12	17.0679149663376\\
11	16.5060135384094\\
10	14.6680271489735\\
9	13.7098404971451\\
8	13.063447555291\\
7	12.0660067384834\\
6	11.8539430860599\\
5	10.9315583588612\\
4	9.79429094451161\\
3	6.34859990757057\\
2	4.16507008128682\\
1	2.94235529784154\\
}--cycle;

\addplot [color=white!75!blue,solid,forget plot]
  table[row sep=crcr]{%
1	0.637765407216566\\
2	1.75821665734828\\
3	4.22448366883149\\
4	7.53834475757903\\
5	8.93437010383099\\
6	9.27723698550747\\
7	9.51484471303712\\
8	10.7631756247923\\
9	11.2798950796265\\
10	12.4219088260046\\
11	13.816118851894\\
12	14.8302050466096\\
13	14.9811740940973\\
14	14.9549204292727\\
15	16.5150445808219\\
16	16.9293213521845\\
17	15.9313874435707\\
18	16.178917016631\\
19	15.8117886086493\\
20	16.4089558533962\\
21	16.7169013486018\\
22	17.891686875753\\
23	18.6875894524288\\
24	18.5645428084025\\
25	19.0285899301553\\
26	18.7787851358048\\
27	18.7534015930443\\
28	17.3394055175201\\
29	17.1657035920476\\
30	17.1833675517804\\
31	17.6627738497803\\
32	17.4911505255754\\
33	17.2346498217895\\
34	17.3732051020901\\
35	19.6903112306307\\
36	21.8607386883626\\
37	22.5500896287786\\
38	21.4082570328725\\
39	20.9950795834684\\
40	20.0462359092475\\
41	20.5950655056293\\
42	22.3779893767237\\
43	22.4069910923872\\
44	23.2110512638861\\
45	24.3303104908593\\
46	24.4003913254158\\
47	24.2744087535418\\
48	24.6061759139027\\
49	24.0696357303121\\
50	25.0701552656363\\
51	24.512988971149\\
52	25.0983300622827\\
53	26.5055059487003\\
54	27.724996931309\\
55	26.243904298058\\
56	25.051342287049\\
57	23.6836338524075\\
58	24.7093006905013\\
59	25.5614276848235\\
60	24.2114850045419\\
61	23.9711773929707\\
62	24.1084887602594\\
63	24.7189880794218\\
64	27.7567712158059\\
65	27.1093425527696\\
66	27.0888378341886\\
67	27.5511043471614\\
68	27.0669115174562\\
69	26.3106279570434\\
70	27.0969856853487\\
71	29.1512039845696\\
72	27.5000545283356\\
73	28.4733433597024\\
74	28.4057068958373\\
75	29.3398812423512\\
76	27.763385844161\\
77	29.2428138723087\\
78	28.5183235622843\\
79	28.8795239651866\\
80	29.9525636975736\\
81	29.6320959854861\\
82	27.6470803106091\\
83	26.8840310125002\\
84	25.693730483722\\
85	25.0554223160995\\
86	22.5313917879154\\
87	20.8537003297476\\
88	21.7476204388717\\
89	23.0633574800864\\
90	24.4319264415415\\
91	26.1500196228324\\
92	27.3204232952247\\
93	28.4186718255671\\
94	26.5799735388622\\
95	27.946124363353\\
96	28.7687807898257\\
97	25.7031286312379\\
98	23.5195141077821\\
99	22.5832099822186\\
100	22.5381035564667\\
};
\addplot [color=white!75!blue,solid,forget plot]
  table[row sep=crcr]{%
1	2.94235529784154\\
2	4.16507008128682\\
3	6.34859990757057\\
4	9.79429094451161\\
5	10.9315583588612\\
6	11.8539430860599\\
7	12.0660067384834\\
8	13.063447555291\\
9	13.7098404971451\\
10	14.6680271489735\\
11	16.5060135384094\\
12	17.0679149663376\\
13	17.5933363760415\\
14	17.4138263494375\\
15	18.8803488206403\\
16	19.7042755570942\\
17	18.9107330914249\\
18	19.099455859028\\
19	18.4945426984265\\
20	18.8108434262862\\
21	19.8005590332839\\
22	20.1425948867238\\
23	20.7628667486071\\
24	20.6875886485008\\
25	21.9465315412466\\
26	21.9005259364269\\
27	21.7185549215429\\
28	20.3895572464801\\
29	19.8639521593398\\
30	20.2169931503567\\
31	20.5013445119745\\
32	20.1494606850085\\
33	19.7475789654787\\
34	19.904924968897\\
35	22.2331950306409\\
36	24.1321257549712\\
37	24.6333017207692\\
38	24.6672521596584\\
39	23.8539083601986\\
40	22.8939790729886\\
41	23.2566275484537\\
42	25.446927778924\\
43	24.697036178301\\
44	25.8584574905477\\
45	26.9275724917542\\
46	27.2304118615701\\
47	27.0774721519173\\
48	27.382880937843\\
49	26.6962882571605\\
50	27.859118164747\\
51	27.015838637313\\
52	27.9734074686859\\
53	29.4630031738579\\
54	31.3192027980137\\
55	28.986113345523\\
56	27.7934788486073\\
57	27.0449958438479\\
58	27.5313792301665\\
59	28.3825326186558\\
60	26.8482878613633\\
61	26.6893260603812\\
62	26.4706599440423\\
63	28.1335280011737\\
64	30.5576930808279\\
65	29.8630132962358\\
66	29.6075389820168\\
67	29.910281243425\\
68	29.9891075934636\\
69	29.6514415368058\\
70	29.7008855042964\\
71	31.7667804467182\\
72	30.354280136065\\
73	31.2757771531191\\
74	31.9161868390768\\
75	31.7560288556049\\
76	31.1378394795213\\
77	31.7917555714895\\
78	31.138111500091\\
79	30.9600941472087\\
80	32.1537491076573\\
81	32.8792271823612\\
82	30.7130750764696\\
83	30.266036479628\\
84	28.8049868390147\\
85	28.469891905856\\
86	25.3866327753134\\
87	23.5851329441327\\
88	24.7052069231204\\
89	26.0599202589617\\
90	26.9994978829302\\
91	28.8219000237104\\
92	29.9230860298473\\
93	30.9223959091637\\
94	29.7433770400123\\
95	30.6456898665351\\
96	31.7467524429032\\
97	29.2087441576965\\
98	26.2635897723681\\
99	25.6650173456401\\
100	25.2191970367007\\
};
\addplot [color=blue,dashed]
  table[row sep=crcr]{%
1	1.79006035252905\\
2	2.96164336931755\\
3	5.28654178820103\\
4	8.66631785104532\\
5	9.9329642313461\\
6	10.5655900357837\\
7	10.7904257257603\\
8	11.9133115900416\\
9	12.4948677883858\\
10	13.544967987489\\
11	15.1610661951517\\
12	15.9490600064736\\
13	16.2872552350694\\
14	16.1843733893551\\
15	17.6976967007311\\
16	18.3167984546394\\
17	17.4210602674978\\
18	17.6391864378295\\
19	17.1531656535379\\
20	17.6098996398412\\
21	18.2587301909428\\
22	19.0171408812384\\
23	19.7252281005179\\
24	19.6260657284517\\
25	20.4875607357009\\
26	20.3396555361158\\
27	20.2359782572936\\
28	18.8644813820001\\
29	18.5148278756937\\
30	18.7001803510685\\
31	19.0820591808774\\
32	18.820305605292\\
33	18.4911143936341\\
34	18.6390650354935\\
35	20.9617531306358\\
36	22.9964322216669\\
37	23.5916956747739\\
38	23.0377545962654\\
39	22.4244939718335\\
40	21.470107491118\\
41	21.9258465270415\\
42	23.9124585778239\\
43	23.5520136353441\\
44	24.5347543772169\\
45	25.6289414913067\\
46	25.815401593493\\
47	25.6759404527296\\
48	25.9945284258728\\
49	25.3829619937363\\
50	26.4646367151916\\
51	25.764413804231\\
52	26.5358687654843\\
53	27.9842545612791\\
54	29.5220998646613\\
55	27.6150088217905\\
56	26.4224105678281\\
57	25.3643148481277\\
58	26.1203399603339\\
59	26.9719801517397\\
60	25.5298864329526\\
61	25.3302517266759\\
62	25.2895743521508\\
63	26.4262580402977\\
64	29.1572321483169\\
65	28.4861779245027\\
66	28.3481884081027\\
67	28.7306927952932\\
68	28.5280095554599\\
69	27.9810347469246\\
70	28.3989355948226\\
71	30.4589922156439\\
72	28.9271673322003\\
73	29.8745602564107\\
74	30.160946867457\\
75	30.5479550489781\\
76	29.4506126618412\\
77	30.5172847218991\\
78	29.8282175311877\\
79	29.9198090561976\\
80	31.0531564026154\\
81	31.2556615839237\\
82	29.1800776935394\\
83	28.5750337460641\\
84	27.2493586613684\\
85	26.7626571109777\\
86	23.9590122816144\\
87	22.2194166369401\\
88	23.226413680996\\
89	24.561638869524\\
90	25.7157121622359\\
91	27.4859598232714\\
92	28.621754662536\\
93	29.6705338673654\\
94	28.1616752894372\\
95	29.295907114944\\
96	30.2577666163644\\
97	27.4559363944672\\
98	24.8915519400751\\
99	24.1241136639294\\
100	23.8786502965837\\
};
\addlegendentry{CPHD};

\end{axis}
\end{tikzpicture}%

%% file: rmse_mean1.tikz
%
\definecolor{mycolor1}{rgb}{0.75000,0.97500,0.75000}%
\definecolor{mycolor2}{rgb}{0.75000,0.75000,0.97500}%
\begin{tikzpicture}

\begin{axis}[%
width=0.8\linewidth,
height=1.6in,
at={(1.011111in,0.641667in)},
scale only axis,
xmin=1,
xmax=30,
xmajorgrids,
xlabel={time},
ymin=0,
ymax=6,
ymajorgrids,
ylabel={RMSE},
axis x line*=bottom,
axis y line*=left,
legend style={at={(0.03,0.97)},anchor=north west,legend cell align=left,align=left,draw=white!15!black,fill opacity=0.6,draw opacity=1,text opacity=1,font=\footnotesize}]

\addplot[area legend,solid,fill=black!10!red,opacity=1.000000e-01,draw=none,forget plot]
table[row sep=crcr] {%
x	y\\
1	0\\
2	0.0269927878682839\\
3	0.0328152583665364\\
4	0.0561400924668571\\
5	0.0577954573119578\\
6	0.0570583839134379\\
7	0.0991142372298583\\
8	0.136296602171465\\
9	0.171602002446323\\
10	0.194513938213472\\
11	0.212071066958937\\
12	0.27651833927047\\
13	0.288597704543515\\
14	0.333044101587839\\
15	0.408207557552842\\
16	0.525856412626329\\
17	0.666601382220383\\
18	0.789660954084423\\
19	0.908642671798023\\
20	1.03355655747079\\
21	1.01377266823244\\
22	1.1535612747491\\
23	1.29895612160882\\
24	1.4300901914984\\
25	1.47203879522445\\
26	1.63012651166647\\
27	1.74991819768988\\
28	1.83619180908825\\
29	1.47016117159584\\
30	1.660172348578\\
30	4.66460113218801\\
29	4.39387362719355\\
28	3.60597909370198\\
27	3.37150112045498\\
26	3.1420810530438\\
25	2.94426300460153\\
24	2.65829807627314\\
23	2.43838258684596\\
22	2.27060714477312\\
21	2.02446158713865\\
20	1.59007073389964\\
19	1.49694023983061\\
18	1.3348761516139\\
17	1.17533535827397\\
16	1.00053534416497\\
15	0.843422014463509\\
14	0.701972203710419\\
13	0.567049359038075\\
12	0.482758926949552\\
11	0.484057356756166\\
10	0.402888252463328\\
9	0.341278299328154\\
8	0.277288511773153\\
7	0.217856339231391\\
6	0.131476810034335\\
5	0.104607383649324\\
4	0.0788523522877956\\
3	0.0484389666817581\\
2	0.0336973750820664\\
1	0\\
}--cycle;

\addplot [color=pink!90!lightgray,solid,forget plot]
  table[row sep=crcr]{%
1	0\\
2	0.0269927878682839\\
3	0.0328152583665364\\
4	0.0561400924668571\\
5	0.0577954573119578\\
6	0.0570583839134379\\
7	0.0991142372298583\\
8	0.136296602171465\\
9	0.171602002446323\\
10	0.194513938213472\\
11	0.212071066958937\\
12	0.27651833927047\\
13	0.288597704543515\\
14	0.333044101587839\\
15	0.408207557552842\\
16	0.525856412626329\\
17	0.666601382220383\\
18	0.789660954084423\\
19	0.908642671798023\\
20	1.03355655747079\\
21	1.01377266823244\\
22	1.1535612747491\\
23	1.29895612160882\\
24	1.4300901914984\\
25	1.47203879522445\\
26	1.63012651166647\\
27	1.74991819768988\\
28	1.83619180908825\\
29	1.47016117159584\\
30	1.660172348578\\
};
\addplot [color=pink!90!lightgray,solid,forget plot]
  table[row sep=crcr]{%
1	0\\
2	0.0336973750820664\\
3	0.0484389666817581\\
4	0.0788523522877956\\
5	0.104607383649324\\
6	0.131476810034335\\
7	0.217856339231391\\
8	0.277288511773153\\
9	0.341278299328154\\
10	0.402888252463328\\
11	0.484057356756166\\
12	0.482758926949552\\
13	0.567049359038075\\
14	0.701972203710419\\
15	0.843422014463509\\
16	1.00053534416497\\
17	1.17533535827397\\
18	1.3348761516139\\
19	1.49694023983061\\
20	1.59007073389964\\
21	2.02446158713865\\
22	2.27060714477312\\
23	2.43838258684596\\
24	2.65829807627314\\
25	2.94426300460153\\
26	3.1420810530438\\
27	3.37150112045498\\
28	3.60597909370198\\
29	4.39387362719355\\
30	4.66460113218801\\
};
\addplot [color=black!10!red,solid]
  table[row sep=crcr]{%
1	0\\
2	0.0303450814751752\\
3	0.0406271125241473\\
4	0.0674962223773263\\
5	0.0812014204806411\\
6	0.0942675969738867\\
7	0.158485288230625\\
8	0.206792556972309\\
9	0.256440150887238\\
10	0.2987010953384\\
11	0.348064211857552\\
12	0.379638633110011\\
13	0.427823531790795\\
14	0.517508152649129\\
15	0.625814786008176\\
16	0.76319587839565\\
17	0.920968370247176\\
18	1.06226855284916\\
19	1.20279145581432\\
20	1.31181364568521\\
21	1.51911712768554\\
22	1.71208420976111\\
23	1.86866935422739\\
24	2.04419413388577\\
25	2.20815089991299\\
26	2.38610378235513\\
27	2.56070965907243\\
28	2.72108545139511\\
29	2.93201739939469\\
30	3.16238674038301\\
};
\addlegendentry{PHD L1};

\addplot[area legend,solid,fill=black!10!red,opacity=1.000000e-01,draw=none,forget plot]
table[row sep=crcr] {%
x	y\\
1	0\\
2	0.0363224255356497\\
3	0.080839807020723\\
4	0.153735504054854\\
5	0.210124308326776\\
6	0.263130693274928\\
7	0.367280632942394\\
8	0.478637994943819\\
9	0.594959061075612\\
10	0.6924017109132\\
11	0.776655338217855\\
12	0.879313606896965\\
13	1.01173938743279\\
14	1.17737748043093\\
15	1.39222705380407\\
16	1.62372746050799\\
17	1.76019137484591\\
18	1.98770783323672\\
19	2.20198369780705\\
20	2.379705154625\\
21	2.62462265186134\\
22	2.85134888636757\\
23	3.09228079186772\\
24	3.32759791220678\\
25	3.55633918282571\\
26	3.77438515698971\\
27	3.99830545724068\\
28	4.20590281625741\\
29	4.34081628069474\\
30	4.58614025807468\\
30	6.00922687270366\\
29	5.76466314968178\\
28	5.39290538978731\\
27	5.12080692189514\\
26	4.83688141802511\\
25	4.54740561169045\\
24	4.28544689972085\\
23	3.98863829080709\\
22	3.68611167828112\\
21	3.36719467458221\\
20	3.03978379684238\\
19	2.69133668138217\\
18	2.43917981833618\\
17	2.18510744381921\\
16	1.852073706966\\
15	1.60683858483024\\
14	1.40900687352996\\
13	1.22352072072315\\
12	1.06200951544081\\
11	0.941952495666581\\
10	0.838286834273928\\
9	0.716012054011612\\
8	0.573136931504617\\
7	0.431508742019906\\
6	0.304074800212095\\
5	0.236139592660911\\
4	0.170219834574338\\
3	0.0894609163787013\\
2	0.0410187154647501\\
1	0\\
}--cycle;

\addplot [color=pink!90!lightgray,solid,forget plot]
  table[row sep=crcr]{%
1	0\\
2	0.0363224255356497\\
3	0.080839807020723\\
4	0.153735504054854\\
5	0.210124308326776\\
6	0.263130693274928\\
7	0.367280632942394\\
8	0.478637994943819\\
9	0.594959061075612\\
10	0.6924017109132\\
11	0.776655338217855\\
12	0.879313606896965\\
13	1.01173938743279\\
14	1.17737748043093\\
15	1.39222705380407\\
16	1.62372746050799\\
17	1.76019137484591\\
18	1.98770783323672\\
19	2.20198369780705\\
20	2.379705154625\\
21	2.62462265186134\\
22	2.85134888636757\\
23	3.09228079186772\\
24	3.32759791220678\\
25	3.55633918282571\\
26	3.77438515698971\\
27	3.99830545724068\\
28	4.20590281625741\\
29	4.34081628069474\\
30	4.58614025807468\\
};
\addplot [color=pink!90!lightgray,solid,forget plot]
  table[row sep=crcr]{%
1	0\\
2	0.0410187154647501\\
3	0.0894609163787013\\
4	0.170219834574338\\
5	0.236139592660911\\
6	0.304074800212095\\
7	0.431508742019906\\
8	0.573136931504617\\
9	0.716012054011612\\
10	0.838286834273928\\
11	0.941952495666581\\
12	1.06200951544081\\
13	1.22352072072315\\
14	1.40900687352996\\
15	1.60683858483024\\
16	1.852073706966\\
17	2.18510744381921\\
18	2.43917981833618\\
19	2.69133668138217\\
20	3.03978379684238\\
21	3.36719467458221\\
22	3.68611167828112\\
23	3.98863829080709\\
24	4.28544689972085\\
25	4.54740561169045\\
26	4.83688141802511\\
27	5.12080692189514\\
28	5.39290538978731\\
29	5.76466314968178\\
30	6.00922687270366\\
};
\addplot [color=black!10!red,dashed]
  table[row sep=crcr]{%
1	0\\
2	0.0386705705001999\\
3	0.0851503616997122\\
4	0.161977669314596\\
5	0.223131950493843\\
6	0.283602746743511\\
7	0.39939468748115\\
8	0.525887463224218\\
9	0.655485557543612\\
10	0.765344272593564\\
11	0.859303916942218\\
12	0.970661561168887\\
13	1.11763005407797\\
14	1.29319217698045\\
15	1.49953281931716\\
16	1.73790058373699\\
17	1.97264940933256\\
18	2.21344382578645\\
19	2.44666018959461\\
20	2.70974447573369\\
21	2.99590866322177\\
22	3.26873028232434\\
23	3.5404595413374\\
24	3.80652240596382\\
25	4.05187239725808\\
26	4.30563328750741\\
27	4.55955618956791\\
28	4.79940410302236\\
29	5.05273971518826\\
30	5.29768356538917\\
};
\addlegendentry{PHD L2};

\addplot[area legend,solid,fill=black!10!green,opacity=1.000000e-01,draw=none,forget plot]
table[row sep=crcr] {%
x	y\\
1	0\\
2	0.0270165763043622\\
3	0.0353095196800562\\
4	0.0492957378229449\\
5	0.0509106820978006\\
6	0.0475269528166627\\
7	0.0897128399903637\\
8	0.0844454753233249\\
9	0.0945367928844237\\
10	0.110659794466763\\
11	0.146608263171466\\
12	0.13154009072023\\
13	0.0887188691678689\\
14	0.132015021600412\\
15	0.193476658687525\\
16	0.321347878337757\\
17	0.422796124378346\\
18	0.53686672061401\\
19	0.646371709668328\\
20	0.694760826862562\\
21	0.711429132415242\\
22	0.914277459068273\\
23	1.03448253677601\\
24	1.16773785090007\\
25	1.20935768739461\\
26	1.35597306993579\\
27	1.47267269240475\\
28	1.56661160460672\\
29	1.33962278651235\\
30	1.47429363137427\\
30	5.07984718094454\\
29	4.78508234705342\\
28	4.12502831601462\\
27	3.8944158425706\\
26	3.65329683191711\\
25	3.43113971978103\\
24	3.21628027151057\\
23	2.97155219926793\\
22	2.77418489271783\\
21	2.48150123107447\\
20	2.11552813841373\\
19	1.9090487875453\\
18	1.7246482603753\\
17	1.53819656685645\\
16	1.33832597612112\\
15	1.16607458700195\\
14	1.02804836461943\\
13	0.902427606476043\\
12	0.695859388276616\\
11	0.549136366858325\\
10	0.452886169118657\\
9	0.366566703065914\\
8	0.280936790566891\\
7	0.212976999407173\\
6	0.136010184103578\\
5	0.10797414516468\\
4	0.087584321768392\\
3	0.0540713366817317\\
2	0.0359111893651401\\
1	0\\
}--cycle;

\addplot [color=mycolor1,solid,forget plot]
  table[row sep=crcr]{%
1	0\\
2	0.0270165763043622\\
3	0.0353095196800562\\
4	0.0492957378229449\\
5	0.0509106820978006\\
6	0.0475269528166627\\
7	0.0897128399903637\\
8	0.0844454753233249\\
9	0.0945367928844237\\
10	0.110659794466763\\
11	0.146608263171466\\
12	0.13154009072023\\
13	0.0887188691678689\\
14	0.132015021600412\\
15	0.193476658687525\\
16	0.321347878337757\\
17	0.422796124378346\\
18	0.53686672061401\\
19	0.646371709668328\\
20	0.694760826862562\\
21	0.711429132415242\\
22	0.914277459068273\\
23	1.03448253677601\\
24	1.16773785090007\\
25	1.20935768739461\\
26	1.35597306993579\\
27	1.47267269240475\\
28	1.56661160460672\\
29	1.33962278651235\\
30	1.47429363137427\\
};
\addplot [color=mycolor1,solid,forget plot]
  table[row sep=crcr]{%
1	0\\
2	0.0359111893651401\\
3	0.0540713366817317\\
4	0.087584321768392\\
5	0.10797414516468\\
6	0.136010184103578\\
7	0.212976999407173\\
8	0.280936790566891\\
9	0.366566703065914\\
10	0.452886169118657\\
11	0.549136366858325\\
12	0.695859388276616\\
13	0.902427606476043\\
14	1.02804836461943\\
15	1.16607458700195\\
16	1.33832597612112\\
17	1.53819656685645\\
18	1.7246482603753\\
19	1.9090487875453\\
20	2.11552813841373\\
21	2.48150123107447\\
22	2.77418489271783\\
23	2.97155219926793\\
24	3.21628027151057\\
25	3.43113971978103\\
26	3.65329683191711\\
27	3.8944158425706\\
28	4.12502831601462\\
29	4.78508234705342\\
30	5.07984718094454\\
};
\addplot [color=black!10!green,solid]
  table[row sep=crcr]{%
1	0\\
2	0.0314638828347511\\
3	0.0446904281808939\\
4	0.0684400297956684\\
5	0.0794424136312401\\
6	0.0917685684601205\\
7	0.151344919698768\\
8	0.182691132945108\\
9	0.230551747975169\\
10	0.28177298179271\\
11	0.347872315014896\\
12	0.413699739498423\\
13	0.495573237821956\\
14	0.580031693109921\\
15	0.679775622844736\\
16	0.829836927229439\\
17	0.980496345617398\\
18	1.13075749049466\\
19	1.27771024860681\\
20	1.40514448263814\\
21	1.59646518174486\\
22	1.84423117589305\\
23	2.00301736802197\\
24	2.19200906120532\\
25	2.32024870358782\\
26	2.50463495092645\\
27	2.68354426748768\\
28	2.84581996031067\\
29	3.06235256678288\\
30	3.2770704061594\\
};
\addlegendentry{Panjer L1};

\addplot[area legend,solid,fill=black!10!green,opacity=1.000000e-01,draw=none,forget plot]
table[row sep=crcr] {%
x	y\\
1	0\\
2	0.0344275989917864\\
3	0.0758443207254289\\
4	0.138669505177318\\
5	0.185581841380982\\
6	0.230630910240125\\
7	0.331248368656371\\
8	0.445872702603054\\
9	0.556583980471254\\
10	0.651517016193762\\
11	0.733305164440132\\
12	0.836732772585747\\
13	0.96846852129749\\
14	1.1393256752731\\
15	1.32677318681515\\
16	1.56167022893787\\
17	1.72738657306421\\
18	1.81322056086338\\
19	1.9991847495165\\
20	2.21266285951072\\
21	2.45558523812653\\
22	2.70247304814684\\
23	2.93019534658265\\
24	3.15068961018676\\
25	3.33124288486285\\
26	3.54134453505131\\
27	3.76224899177523\\
28	3.96246613516647\\
29	4.08932375375674\\
30	4.28370876652493\\
30	6.03451716403887\\
29	5.78039055232365\\
28	5.40753574924188\\
27	5.13013807281245\\
26	4.84003253012076\\
25	4.55798485191039\\
24	4.28386169191881\\
23	3.98810492789738\\
22	3.68620595657587\\
21	3.39128795867491\\
20	3.08491876841972\\
19	2.78235209964093\\
18	2.50053171621759\\
17	2.14534468984036\\
16	1.83612226739971\\
15	1.59150029915334\\
14	1.38282882613953\\
13	1.20147936426846\\
12	1.03387562057119\\
11	0.908777904772465\\
10	0.802994413092956\\
9	0.682904001093771\\
8	0.540298112548846\\
7	0.407696716693538\\
6	0.285052012216836\\
5	0.222277052807768\\
4	0.160675972905031\\
3	0.0908395606360577\\
2	0.0425125932752297\\
1	0\\
}--cycle;

\addplot [color=mycolor1,solid,forget plot]
  table[row sep=crcr]{%
1	0\\
2	0.0344275989917864\\
3	0.0758443207254289\\
4	0.138669505177318\\
5	0.185581841380982\\
6	0.230630910240125\\
7	0.331248368656371\\
8	0.445872702603054\\
9	0.556583980471254\\
10	0.651517016193762\\
11	0.733305164440132\\
12	0.836732772585747\\
13	0.96846852129749\\
14	1.1393256752731\\
15	1.32677318681515\\
16	1.56167022893787\\
17	1.72738657306421\\
18	1.81322056086338\\
19	1.9991847495165\\
20	2.21266285951072\\
21	2.45558523812653\\
22	2.70247304814684\\
23	2.93019534658265\\
24	3.15068961018676\\
25	3.33124288486285\\
26	3.54134453505131\\
27	3.76224899177523\\
28	3.96246613516647\\
29	4.08932375375674\\
30	4.28370876652493\\
};
\addplot [color=mycolor1,solid,forget plot]
  table[row sep=crcr]{%
1	0\\
2	0.0425125932752297\\
3	0.0908395606360577\\
4	0.160675972905031\\
5	0.222277052807768\\
6	0.285052012216836\\
7	0.407696716693538\\
8	0.540298112548846\\
9	0.682904001093771\\
10	0.802994413092956\\
11	0.908777904772465\\
12	1.03387562057119\\
13	1.20147936426846\\
14	1.38282882613953\\
15	1.59150029915334\\
16	1.83612226739971\\
17	2.14534468984036\\
18	2.50053171621759\\
19	2.78235209964093\\
20	3.08491876841972\\
21	3.39128795867491\\
22	3.68620595657587\\
23	3.98810492789738\\
24	4.28386169191881\\
25	4.55798485191039\\
26	4.84003253012076\\
27	5.13013807281245\\
28	5.40753574924188\\
29	5.78039055232365\\
30	6.03451716403887\\
};
\addplot [color=black!10!green,dashed]
  table[row sep=crcr]{%
1	0\\
2	0.038470096133508\\
3	0.0833419406807433\\
4	0.149672739041175\\
5	0.203929447094375\\
6	0.25784146122848\\
7	0.369472542674954\\
8	0.49308540757595\\
9	0.619743990782513\\
10	0.727255714643359\\
11	0.821041534606299\\
12	0.935304196578468\\
13	1.08497394278298\\
14	1.26107725070631\\
15	1.45913674298425\\
16	1.69889624816879\\
17	1.93636563145229\\
18	2.15687613854049\\
19	2.39076842457872\\
20	2.64879081396522\\
21	2.92343659840072\\
22	3.19433950236135\\
23	3.45915013724002\\
24	3.71727565105279\\
25	3.94461386838662\\
26	4.19068853258603\\
27	4.44619353229384\\
28	4.68500094220417\\
29	4.93485715304019\\
30	5.1591129652819\\
};
\addlegendentry{Panjer L2};

\addplot[area legend,solid,fill=black!10!blue,opacity=1.000000e-01,draw=none,forget plot]
table[row sep=crcr] {%
x	y\\
1	0\\
2	0.0270377355556675\\
3	0.0351185643342078\\
4	0.0512934534170073\\
5	0.0573974474160986\\
6	0.0640834149479363\\
7	0.111774279524079\\
8	0.162436581188638\\
9	0.194434168367113\\
10	0.213648296295052\\
11	0.190242351497431\\
12	0.20280439039725\\
13	0.20998083743079\\
14	0.214130564290259\\
15	0.262083043931622\\
16	0.403953505757928\\
17	0.53763544639858\\
18	0.647337816885116\\
19	0.749238071683233\\
20	0.776146411347248\\
21	0.903974784345864\\
22	1.19740058603991\\
23	1.32860474940574\\
24	1.43043110925419\\
25	1.48463808329306\\
26	1.6209683384567\\
27	1.73651560260016\\
28	1.82330579303657\\
29	1.85657375230538\\
30	1.88753702204198\\
30	4.53515350345741\\
29	4.15794731954417\\
28	3.82298019226818\\
27	3.60229932416351\\
26	3.36234676324979\\
25	3.14966544486687\\
24	2.88121185081904\\
23	2.66038987795656\\
22	2.48499765292828\\
21	2.31567685119378\\
20	2.04697335140422\\
19	1.84827076784319\\
18	1.66292344750731\\
17	1.47780925721909\\
16	1.23640614788799\\
15	1.05862964414983\\
14	0.912596419603508\\
13	0.753945054088929\\
12	0.594690101037601\\
11	0.475267312569233\\
10	0.353717334481518\\
9	0.290563112956652\\
8	0.232661009553496\\
7	0.190388597092338\\
6	0.111239135100255\\
5	0.0961833909930269\\
4	0.0797925343339791\\
3	0.0553920129252887\\
2	0.0355746278207139\\
1	0\\
}--cycle;

\addplot [color=mycolor2,solid,forget plot]
  table[row sep=crcr]{%
1	0\\
2	0.0270377355556675\\
3	0.0351185643342078\\
4	0.0512934534170073\\
5	0.0573974474160986\\
6	0.0640834149479363\\
7	0.111774279524079\\
8	0.162436581188638\\
9	0.194434168367113\\
10	0.213648296295052\\
11	0.190242351497431\\
12	0.20280439039725\\
13	0.20998083743079\\
14	0.214130564290259\\
15	0.262083043931622\\
16	0.403953505757928\\
17	0.53763544639858\\
18	0.647337816885116\\
19	0.749238071683233\\
20	0.776146411347248\\
21	0.903974784345864\\
22	1.19740058603991\\
23	1.32860474940574\\
24	1.43043110925419\\
25	1.48463808329306\\
26	1.6209683384567\\
27	1.73651560260016\\
28	1.82330579303657\\
29	1.85657375230538\\
30	1.88753702204198\\
};
\addplot [color=mycolor2,solid,forget plot]
  table[row sep=crcr]{%
1	0\\
2	0.0355746278207139\\
3	0.0553920129252887\\
4	0.0797925343339791\\
5	0.0961833909930269\\
6	0.111239135100255\\
7	0.190388597092338\\
8	0.232661009553496\\
9	0.290563112956652\\
10	0.353717334481518\\
11	0.475267312569233\\
12	0.594690101037601\\
13	0.753945054088929\\
14	0.912596419603508\\
15	1.05862964414983\\
16	1.23640614788799\\
17	1.47780925721909\\
18	1.66292344750731\\
19	1.84827076784319\\
20	2.04697335140422\\
21	2.31567685119378\\
22	2.48499765292828\\
23	2.66038987795656\\
24	2.88121185081904\\
25	3.14966544486687\\
26	3.36234676324979\\
27	3.60229932416351\\
28	3.82298019226818\\
29	4.15794731954417\\
30	4.53515350345741\\
};
\addplot [color=black!10!blue,solid]
  table[row sep=crcr]{%
1	0\\
2	0.0313061816881907\\
3	0.0452552886297483\\
4	0.0655429938754932\\
5	0.0767904192045627\\
6	0.0876612750240956\\
7	0.151081438308208\\
8	0.197548795371067\\
9	0.242498640661882\\
10	0.283682815388285\\
11	0.332754832033332\\
12	0.398747245717426\\
13	0.48196294575986\\
14	0.563363491946884\\
15	0.660356344040727\\
16	0.82017982682296\\
17	1.00772235180884\\
18	1.15513063219621\\
19	1.29875441976321\\
20	1.41155988137573\\
21	1.60982581776982\\
22	1.8411991194841\\
23	1.99449731368115\\
24	2.15582148003662\\
25	2.31715176407997\\
26	2.49165755085325\\
27	2.66940746338184\\
28	2.82314299265237\\
29	3.00726053592478\\
30	3.2113452627497\\
};
\addlegendentry{CPHD L1};

\addplot[area legend,solid,fill=black!10!blue,opacity=1.000000e-01,draw=none,forget plot]
table[row sep=crcr] {%
x	y\\
1	0\\
2	0.035016683414905\\
3	0.0766222730437811\\
4	0.149107051045913\\
5	0.204227217114492\\
6	0.25759408314587\\
7	0.367873699739769\\
8	0.492546571391349\\
9	0.619272314688325\\
10	0.722697677080807\\
11	0.815586691502898\\
12	0.922755515720877\\
13	1.06711109434822\\
14	1.22445199533415\\
15	1.39407105176873\\
16	1.60353639199542\\
17	1.8655872848843\\
18	2.0852520251747\\
19	2.3078327464286\\
20	2.53313324450074\\
21	2.79138501513513\\
22	3.03045324157136\\
23	3.27321988802619\\
24	3.51875091832225\\
25	3.74602225948936\\
26	3.96988761302628\\
27	4.20320136691698\\
28	4.40771418527342\\
29	4.52337833116271\\
30	4.71912665120021\\
30	6.03588538075305\\
29	5.77424753130888\\
28	5.37125537957739\\
27	5.0855992556153\\
26	4.78552689219309\\
25	4.50060698923993\\
24	4.26573743733434\\
23	3.97038688234911\\
22	3.66797003914334\\
21	3.32341160320622\\
20	3.02331326649014\\
19	2.71727815675532\\
18	2.45532426497877\\
17	2.16447840444438\\
16	1.93402996158547\\
15	1.65000052774916\\
14	1.40503872978415\\
13	1.19925572727865\\
12	1.03619729956801\\
11	0.908754221131613\\
10	0.808092490870323\\
9	0.688389520347285\\
8	0.551306868075886\\
7	0.419208877180316\\
6	0.297709219948083\\
5	0.232437577663887\\
4	0.16794669417754\\
3	0.0908727438819959\\
2	0.0410224430969482\\
1	0\\
}--cycle;

\addplot [color=mycolor2,solid,forget plot]
  table[row sep=crcr]{%
1	0\\
2	0.035016683414905\\
3	0.0766222730437811\\
4	0.149107051045913\\
5	0.204227217114492\\
6	0.25759408314587\\
7	0.367873699739769\\
8	0.492546571391349\\
9	0.619272314688325\\
10	0.722697677080807\\
11	0.815586691502898\\
12	0.922755515720877\\
13	1.06711109434822\\
14	1.22445199533415\\
15	1.39407105176873\\
16	1.60353639199542\\
17	1.8655872848843\\
18	2.0852520251747\\
19	2.3078327464286\\
20	2.53313324450074\\
21	2.79138501513513\\
22	3.03045324157136\\
23	3.27321988802619\\
24	3.51875091832225\\
25	3.74602225948936\\
26	3.96988761302628\\
27	4.20320136691698\\
28	4.40771418527342\\
29	4.52337833116271\\
30	4.71912665120021\\
};
\addplot [color=mycolor2,solid,forget plot]
  table[row sep=crcr]{%
1	0\\
2	0.0410224430969482\\
3	0.0908727438819959\\
4	0.16794669417754\\
5	0.232437577663887\\
6	0.297709219948083\\
7	0.419208877180316\\
8	0.551306868075886\\
9	0.688389520347285\\
10	0.808092490870323\\
11	0.908754221131613\\
12	1.03619729956801\\
13	1.19925572727865\\
14	1.40503872978415\\
15	1.65000052774916\\
16	1.93402996158547\\
17	2.16447840444438\\
18	2.45532426497877\\
19	2.71727815675532\\
20	3.02331326649014\\
21	3.32341160320622\\
22	3.66797003914334\\
23	3.97038688234911\\
24	4.26573743733434\\
25	4.50060698923993\\
26	4.78552689219309\\
27	5.0855992556153\\
28	5.37125537957739\\
29	5.77424753130888\\
30	6.03588538075305\\
};
\addplot [color=black!10!blue,dashed]
  table[row sep=crcr]{%
1	0\\
2	0.0380195632559266\\
3	0.0837475084628885\\
4	0.158526872611726\\
5	0.218332397389189\\
6	0.277651651546976\\
7	0.393541288460042\\
8	0.521926719733618\\
9	0.653830917517805\\
10	0.765395083975565\\
11	0.862170456317256\\
12	0.979476407644444\\
13	1.13318341081343\\
14	1.31474536255915\\
15	1.52203578975895\\
16	1.76878317679045\\
17	2.01503284466434\\
18	2.27028814507674\\
19	2.51255545159196\\
20	2.77822325549544\\
21	3.05739830917067\\
22	3.34921164035735\\
23	3.62180338518765\\
24	3.8922441778283\\
25	4.12331462436465\\
26	4.37770725260969\\
27	4.64440031126614\\
28	4.8894847824254\\
29	5.1488129312358\\
30	5.37750601597663\\
};
\addlegendentry{CPHD L2};

\end{axis}
\end{tikzpicture}%

%% file: card_mean1_likeli1.tikz
%
\definecolor{mycolor1}{rgb}{0.75000,0.97500,0.75000}%
\definecolor{mycolor2}{rgb}{0.75000,0.75000,0.97500}%
\begin{tikzpicture}

\begin{axis}[%
width=0.8\linewidth,
height=1.2in,
at={(1.011111in,0.641667in)},
scale only axis,
xmin=1,
xmax=30,
xlabel={time},
ymin=-0.00779533260380738,
ymax=19.9,
ylabel={\# targets},
axis x line*=bottom,
axis y line*=left,
legend style={at={(0.97,0.97)},anchor=north east,legend cell align=left,align=left,draw=white!15!black,fill opacity=0.6,draw opacity=1,text opacity=1,font=\footnotesize}
]
\addplot [color=black,solid]
  table[row sep=crcr]{%
1	15\\
2	15\\
3	15\\
4	15\\
5	15\\
6	15\\
7	15\\
8	15\\
9	15\\
10	15\\
11	15\\
12	15\\
13	15\\
14	15\\
15	1\\
16	1\\
17	1\\
18	1\\
19	1\\
20	1\\
21	1\\
22	1\\
23	1\\
24	1\\
25	1\\
26	1\\
27	1\\
28	1\\
29	1\\
30	1\\
};
\addlegendentry{gt};

\addplot[area legend,solid,fill=black!10!red,opacity=1.000000e-01,draw=none,forget plot]
table[row sep=crcr] {%
x	y\\
1	1.62751375052095\\
2	14.5473596322241\\
3	14.5472582210994\\
4	14.863671615047\\
5	14.7538186993346\\
6	14.3162142576501\\
7	14.9651813752466\\
8	15.0327655873585\\
9	14.7434040712366\\
10	14.7148070008929\\
11	15.0231821212115\\
12	14.6081042881891\\
13	14.7664697887144\\
14	14.7269896083799\\
15	1.32272421945469\\
16	1.20476462056161\\
17	1.25903174937267\\
18	1.28116625893746\\
19	1.13782678187789\\
20	1.21550518971604\\
21	1.23829745756423\\
22	1.25286008208202\\
23	1.3042789604146\\
24	1.16018052134499\\
25	1.32684387765018\\
26	1.11534768743787\\
27	1.17229066236744\\
28	1.27135590859177\\
29	1.35261641799657\\
30	1.26817630191662\\
30	2.61272448195282\\
29	2.43265377349347\\
28	2.46312750089456\\
27	2.57111074301113\\
26	2.63064935848262\\
25	2.70405357309551\\
24	2.76270646568871\\
23	2.44665873376184\\
22	2.55545420187833\\
21	2.67072782363576\\
20	2.50794806589875\\
19	2.90561867284399\\
18	2.60657877454951\\
17	2.43176488717495\\
16	2.61832472278886\\
15	3.02101279711547\\
14	17.0559777023729\\
13	17.2304260025952\\
12	17.3380319375386\\
11	16.8265067241884\\
10	17.2213976877617\\
9	16.9717662856487\\
8	16.95790998819\\
7	17.039983610721\\
6	17.3775898762562\\
5	17.0343088929875\\
4	17.0055004008449\\
3	17.329979198861\\
2	16.5723952719042\\
1	2.69106701752164\\
}--cycle;

\addplot [color=pink!90!lightgray,solid,forget plot]
  table[row sep=crcr]{%
1	1.62751375052095\\
2	14.5473596322241\\
3	14.5472582210994\\
4	14.863671615047\\
5	14.7538186993346\\
6	14.3162142576501\\
7	14.9651813752466\\
8	15.0327655873585\\
9	14.7434040712366\\
10	14.7148070008929\\
11	15.0231821212115\\
12	14.6081042881891\\
13	14.7664697887144\\
14	14.7269896083799\\
15	1.32272421945469\\
16	1.20476462056161\\
17	1.25903174937267\\
18	1.28116625893746\\
19	1.13782678187789\\
20	1.21550518971604\\
21	1.23829745756423\\
22	1.25286008208202\\
23	1.3042789604146\\
24	1.16018052134499\\
25	1.32684387765018\\
26	1.11534768743787\\
27	1.17229066236744\\
28	1.27135590859177\\
29	1.35261641799657\\
30	1.26817630191662\\
};
\addplot [color=pink!90!lightgray,solid,forget plot]
  table[row sep=crcr]{%
1	2.69106701752164\\
2	16.5723952719042\\
3	17.329979198861\\
4	17.0055004008449\\
5	17.0343088929875\\
6	17.3775898762562\\
7	17.039983610721\\
8	16.95790998819\\
9	16.9717662856487\\
10	17.2213976877617\\
11	16.8265067241884\\
12	17.3380319375386\\
13	17.2304260025952\\
14	17.0559777023729\\
15	3.02101279711547\\
16	2.61832472278886\\
17	2.43176488717495\\
18	2.60657877454951\\
19	2.90561867284399\\
20	2.50794806589875\\
21	2.67072782363576\\
22	2.55545420187833\\
23	2.44665873376184\\
24	2.76270646568871\\
25	2.70405357309551\\
26	2.63064935848262\\
27	2.57111074301113\\
28	2.46312750089456\\
29	2.43265377349347\\
30	2.61272448195282\\
};
\addplot [color=black!10!red,solid]
  table[row sep=crcr]{%
1	2.1592903840213\\
2	15.5598774520642\\
3	15.9386187099802\\
4	15.934586007946\\
5	15.894063796161\\
6	15.8469020669532\\
7	16.0025824929838\\
8	15.9953377877743\\
9	15.8575851784426\\
10	15.9681023443273\\
11	15.9248444226999\\
12	15.9730681128639\\
13	15.9984478956548\\
14	15.8914836553764\\
15	2.17186850828508\\
16	1.91154467167523\\
17	1.84539831827381\\
18	1.94387251674348\\
19	2.02172272736094\\
20	1.8617266278074\\
21	1.9545126406\\
22	1.90415714198018\\
23	1.87546884708822\\
24	1.96144349351685\\
25	2.01544872537285\\
26	1.87299852296024\\
27	1.87170070268929\\
28	1.86724170474316\\
29	1.89263509574502\\
30	1.94045039193472\\
};
\addlegendentry{PHD};

\addplot[area legend,solid,fill=black!10!green,opacity=1.000000e-01,draw=none,forget plot]
table[row sep=crcr] {%
x	y\\
1	1.45738854568646\\
2	14.970740382555\\
3	14.6954038878426\\
4	14.8537076054323\\
5	14.7385162975473\\
6	14.2945084223329\\
7	14.9524793622345\\
8	15.0178623874335\\
9	14.7217067383061\\
10	14.6983534470441\\
11	15.0114025984057\\
12	14.5905386717627\\
13	14.7554433282678\\
14	14.711798276518\\
15	1.24383047852356\\
16	0.994720779868413\\
17	1.04430954354801\\
18	1.09712484671142\\
19	0.885669810528753\\
20	1.05575214199386\\
21	0.991775813720272\\
22	1.05004290960764\\
23	1.09802445189387\\
24	0.884644230903466\\
25	1.05594988021856\\
26	0.913888377158647\\
27	0.824884489609009\\
28	1.04009145032242\\
29	1.17158968602009\\
30	1.02888827060029\\
30	2.49423050198676\\
29	2.25234537450952\\
28	2.34959775471116\\
27	2.6035451264763\\
26	2.45819303522405\\
25	2.59633121991664\\
24	2.71936298731076\\
23	2.28562472483198\\
22	2.38029138065511\\
21	2.57240353120221\\
20	2.29112050377596\\
19	2.80673205079873\\
18	2.45209277892596\\
17	2.29565822396883\\
16	2.38143756236414\\
15	2.80159208756603\\
14	17.0561711282073\\
13	17.2211639484557\\
12	17.340555118246\\
11	16.8190408040083\\
10	17.2238342267774\\
9	16.9767656579626\\
8	16.9571219102333\\
7	17.0408100442926\\
6	17.381821711707\\
5	17.0336907251027\\
4	17.0126879529475\\
3	17.3389956440797\\
2	18.3943914900495\\
1	8.32393673665512\\
}--cycle;

\addplot [color=mycolor1,solid,forget plot]
  table[row sep=crcr]{%
1	1.45738854568646\\
2	14.970740382555\\
3	14.6954038878426\\
4	14.8537076054323\\
5	14.7385162975473\\
6	14.2945084223329\\
7	14.9524793622345\\
8	15.0178623874335\\
9	14.7217067383061\\
10	14.6983534470441\\
11	15.0114025984057\\
12	14.5905386717627\\
13	14.7554433282678\\
14	14.711798276518\\
15	1.24383047852356\\
16	0.994720779868413\\
17	1.04430954354801\\
18	1.09712484671142\\
19	0.885669810528753\\
20	1.05575214199386\\
21	0.991775813720272\\
22	1.05004290960764\\
23	1.09802445189387\\
24	0.884644230903466\\
25	1.05594988021856\\
26	0.913888377158647\\
27	0.824884489609009\\
28	1.04009145032242\\
29	1.17158968602009\\
30	1.02888827060029\\
};
\addplot [color=mycolor1,solid,forget plot]
  table[row sep=crcr]{%
1	8.32393673665512\\
2	18.3943914900495\\
3	17.3389956440797\\
4	17.0126879529475\\
5	17.0336907251027\\
6	17.381821711707\\
7	17.0408100442926\\
8	16.9571219102333\\
9	16.9767656579626\\
10	17.2238342267774\\
11	16.8190408040083\\
12	17.340555118246\\
13	17.2211639484557\\
14	17.0561711282073\\
15	2.80159208756603\\
16	2.38143756236414\\
17	2.29565822396883\\
18	2.45209277892596\\
19	2.80673205079873\\
20	2.29112050377596\\
21	2.57240353120221\\
22	2.38029138065511\\
23	2.28562472483198\\
24	2.71936298731076\\
25	2.59633121991664\\
26	2.45819303522405\\
27	2.6035451264763\\
28	2.34959775471116\\
29	2.25234537450952\\
30	2.49423050198676\\
};
\addplot [color=black!10!green,solid]
  table[row sep=crcr]{%
1	4.89066264117079\\
2	16.6825659363023\\
3	16.0171997659612\\
4	15.9331977791899\\
5	15.886103511325\\
6	15.83816506702\\
7	15.9966447032636\\
8	15.9874921488334\\
9	15.8492361981344\\
10	15.9610938369108\\
11	15.915221701207\\
12	15.9655468950043\\
13	15.9883036383617\\
14	15.8839847023627\\
15	2.0227112830448\\
16	1.68807917111627\\
17	1.66998388375842\\
18	1.77460881281869\\
19	1.84620093066374\\
20	1.67343632288491\\
21	1.78208967246124\\
22	1.71516714513137\\
23	1.69182458836293\\
24	1.80200360910711\\
25	1.8261405500676\\
26	1.68604070619135\\
27	1.71421480804265\\
28	1.69484460251679\\
29	1.7119675302648\\
30	1.76155938629353\\
};
\addlegendentry{Panjer};

\addplot[area legend,solid,fill=black!10!blue,opacity=1.000000e-01,draw=none,forget plot]
table[row sep=crcr] {%
x	y\\
1	-0.00779533260380738\\
2	14.9182537507708\\
3	15.1345088100587\\
4	15.1304720648582\\
5	15.089176699112\\
6	14.8725711855208\\
7	15.0735853628305\\
8	15.1622932750252\\
9	15.065005477457\\
10	14.8300779517115\\
11	15.0732594091247\\
12	14.8348955741352\\
13	14.9515650981331\\
14	15.023996302012\\
15	9.60532083879631\\
16	3.13133244763486\\
17	1.69293855006851\\
18	1.28338439903598\\
19	1.04728476903141\\
20	1.0224516293203\\
21	0.935389395141228\\
22	1.01630354845997\\
23	0.993761791194492\\
24	0.730443930602039\\
25	0.811747337946062\\
26	0.822686199725349\\
27	0.796293498439429\\
28	0.921830883554061\\
29	1.01713418161316\\
30	0.849469750846209\\
30	1.89145570675539\\
29	1.63694128288794\\
28	1.74480213794665\\
27	1.91308632356648\\
26	1.91630139791764\\
25	2.05022124182392\\
24	2.08587509126669\\
23	1.73239315271784\\
22	1.79231352546286\\
21	2.01926747131401\\
20	2.01039292295229\\
19	2.36685212226053\\
18	2.51988620960742\\
17	3.3097703989685\\
16	6.22107673403303\\
15	13.0267776446624\\
14	16.4341845273124\\
13	16.6540757399377\\
12	16.7246013166602\\
11	16.3531129231708\\
10	16.685758243888\\
9	16.3910951253133\\
8	16.440271505786\\
7	16.5292654840553\\
6	16.6686025339606\\
5	16.5079882173277\\
4	16.7211256571837\\
3	17.1955758555805\\
2	19.2396304468375\\
1	6.37478760207862\\
}--cycle;

\addplot [color=mycolor2,solid,forget plot]
  table[row sep=crcr]{%
1	-0.00779533260380738\\
2	14.9182537507708\\
3	15.1345088100587\\
4	15.1304720648582\\
5	15.089176699112\\
6	14.8725711855208\\
7	15.0735853628305\\
8	15.1622932750252\\
9	15.065005477457\\
10	14.8300779517115\\
11	15.0732594091247\\
12	14.8348955741352\\
13	14.9515650981331\\
14	15.023996302012\\
15	9.60532083879631\\
16	3.13133244763486\\
17	1.69293855006851\\
18	1.28338439903598\\
19	1.04728476903141\\
20	1.0224516293203\\
21	0.935389395141228\\
22	1.01630354845997\\
23	0.993761791194492\\
24	0.730443930602039\\
25	0.811747337946062\\
26	0.822686199725349\\
27	0.796293498439429\\
28	0.921830883554061\\
29	1.01713418161316\\
30	0.849469750846209\\
};
\addplot [color=mycolor2,solid,forget plot]
  table[row sep=crcr]{%
1	6.37478760207862\\
2	19.2396304468375\\
3	17.1955758555805\\
4	16.7211256571837\\
5	16.5079882173277\\
6	16.6686025339606\\
7	16.5292654840553\\
8	16.440271505786\\
9	16.3910951253133\\
10	16.685758243888\\
11	16.3531129231708\\
12	16.7246013166602\\
13	16.6540757399377\\
14	16.4341845273124\\
15	13.0267776446624\\
16	6.22107673403303\\
17	3.3097703989685\\
18	2.51988620960742\\
19	2.36685212226053\\
20	2.01039292295229\\
21	2.01926747131401\\
22	1.79231352546286\\
23	1.73239315271784\\
24	2.08587509126669\\
25	2.05022124182392\\
26	1.91630139791764\\
27	1.91308632356648\\
28	1.74480213794665\\
29	1.63694128288794\\
30	1.89145570675539\\
};
\addplot [color=black!10!blue,solid]
  table[row sep=crcr]{%
1	3.18349613473741\\
2	17.0789420988041\\
3	16.1650423328196\\
4	15.925798861021\\
5	15.7985824582199\\
6	15.7705868597407\\
7	15.8014254234429\\
8	15.8012823904056\\
9	15.7280503013851\\
10	15.7579180977998\\
11	15.7131861661478\\
12	15.7797484453977\\
13	15.8028204190354\\
14	15.7290904146622\\
15	11.3160492417293\\
16	4.67620459083395\\
17	2.5013544745185\\
18	1.9016353043217\\
19	1.70706844564597\\
20	1.5164222761363\\
21	1.47732843322762\\
22	1.40430853696141\\
23	1.36307747195616\\
24	1.40815951093437\\
25	1.43098428988499\\
26	1.36949379882149\\
27	1.35468991100295\\
28	1.33331651075036\\
29	1.32703773225055\\
30	1.3704627288008\\
};
\addlegendentry{CPHD};

\end{axis}
\end{tikzpicture}%

%% file: card_mean1_likeli2.tikz
%
\definecolor{mycolor1}{rgb}{0.75000,0.97500,0.75000}%
\definecolor{mycolor2}{rgb}{0.75000,0.75000,0.97500}%
\begin{tikzpicture}

\begin{axis}[%
width=0.8\linewidth,
height=1.2in,
at={(1.011111in,0.641667in)},
scale only axis,
xmin=1,
xmax=30,
xlabel={time},
ymin=-0.00779533260380028,
ymax=19.9,
ylabel={\# targets},
axis x line*=bottom,
axis y line*=left,
legend style={at={(0.97,0.97)},anchor=north east,legend cell align=left,align=left,draw=white!15!black,fill opacity=0.6,draw opacity=1,text opacity=1,font=\footnotesize}
]
\addplot [color=black,solid]
  table[row sep=crcr]{%
1	15\\
2	15\\
3	15\\
4	15\\
5	15\\
6	15\\
7	15\\
8	15\\
9	15\\
10	15\\
11	15\\
12	15\\
13	15\\
14	15\\
15	1\\
16	1\\
17	1\\
18	1\\
19	1\\
20	1\\
21	1\\
22	1\\
23	1\\
24	1\\
25	1\\
26	1\\
27	1\\
28	1\\
29	1\\
30	1\\
};
\addlegendentry{gt};

\addplot[area legend,solid,fill=black!10!red,opacity=1.000000e-01,draw=none,forget plot]
table[row sep=crcr] {%
x	y\\
1	1.62751375052095\\
2	14.5465327062797\\
3	14.5548160984654\\
4	14.8958316432427\\
5	14.7531691988678\\
6	14.3178589552913\\
7	14.9722052497261\\
8	15.0405152179303\\
9	14.7065697618377\\
10	14.7845625837421\\
11	15.0595431766322\\
12	14.6523962858315\\
13	14.8604150931831\\
14	14.7332089346325\\
15	1.37835543071528\\
16	1.22542761883394\\
17	1.25054869197978\\
18	1.29678870605035\\
19	1.22606270929443\\
20	1.25537130311369\\
21	1.34105947039887\\
22	1.31745254807669\\
23	1.30586983870261\\
24	1.23668291087041\\
25	1.35921153772765\\
26	1.16928851635797\\
27	1.19950546236184\\
28	1.33407024193991\\
29	1.34709437169764\\
30	1.30409884252607\\
30	2.56514611359595\\
29	2.43729042859674\\
28	2.37087011279646\\
27	2.53643956092269\\
26	2.54001260063799\\
25	2.63718592890051\\
24	2.63727360828986\\
23	2.41290591116621\\
22	2.45087257265054\\
21	2.53938614996516\\
20	2.44945877890357\\
19	2.7518764583247\\
18	2.57954571043811\\
17	2.43332535749107\\
16	2.57335732878422\\
15	2.89862085346795\\
14	17.0294869321792\\
13	17.0504809843881\\
12	17.2175545498754\\
11	16.7326297180807\\
10	17.0907200015301\\
9	16.9736149592562\\
8	16.9300114383815\\
7	16.9822760046795\\
6	17.3543650997181\\
5	17.0233492084587\\
4	16.8999744151184\\
3	17.3020629007367\\
2	16.5694085037541\\
1	2.69106701752164\\
}--cycle;

\addplot [color=pink!90!lightgray,solid,forget plot]
  table[row sep=crcr]{%
1	1.62751375052095\\
2	14.5465327062797\\
3	14.5548160984654\\
4	14.8958316432427\\
5	14.7531691988678\\
6	14.3178589552913\\
7	14.9722052497261\\
8	15.0405152179303\\
9	14.7065697618377\\
10	14.7845625837421\\
11	15.0595431766322\\
12	14.6523962858315\\
13	14.8604150931831\\
14	14.7332089346325\\
15	1.37835543071528\\
16	1.22542761883394\\
17	1.25054869197978\\
18	1.29678870605035\\
19	1.22606270929443\\
20	1.25537130311369\\
21	1.34105947039887\\
22	1.31745254807669\\
23	1.30586983870261\\
24	1.23668291087041\\
25	1.35921153772765\\
26	1.16928851635797\\
27	1.19950546236184\\
28	1.33407024193991\\
29	1.34709437169764\\
30	1.30409884252607\\
};
\addplot [color=pink!90!lightgray,solid,forget plot]
  table[row sep=crcr]{%
1	2.69106701752164\\
2	16.5694085037541\\
3	17.3020629007367\\
4	16.8999744151184\\
5	17.0233492084587\\
6	17.3543650997181\\
7	16.9822760046795\\
8	16.9300114383815\\
9	16.9736149592562\\
10	17.0907200015301\\
11	16.7326297180807\\
12	17.2175545498754\\
13	17.0504809843881\\
14	17.0294869321792\\
15	2.89862085346795\\
16	2.57335732878422\\
17	2.43332535749107\\
18	2.57954571043811\\
19	2.7518764583247\\
20	2.44945877890357\\
21	2.53938614996516\\
22	2.45087257265054\\
23	2.41290591116621\\
24	2.63727360828986\\
25	2.63718592890051\\
26	2.54001260063799\\
27	2.53643956092269\\
28	2.37087011279646\\
29	2.43729042859674\\
30	2.56514611359595\\
};
\addplot [color=black!10!red,dashed]
  table[row sep=crcr]{%
1	2.1592903840213\\
2	15.5579706050169\\
3	15.928439499601\\
4	15.8979030291806\\
5	15.8882592036633\\
6	15.8361120275047\\
7	15.9772406272028\\
8	15.9852633281559\\
9	15.8400923605469\\
10	15.9376412926361\\
11	15.8960864473565\\
12	15.9349754178534\\
13	15.9554480387856\\
14	15.8813479334058\\
15	2.13848814209162\\
16	1.89939247380908\\
17	1.84193702473542\\
18	1.93816720824423\\
19	1.98896958380956\\
20	1.85241504100863\\
21	1.94022281018202\\
22	1.88416256036361\\
23	1.85938787493441\\
24	1.93697825958014\\
25	1.99819873331408\\
26	1.85465055849798\\
27	1.86797251164227\\
28	1.85247017736819\\
29	1.89219240014719\\
30	1.93462247806101\\
};
\addlegendentry{PHD};

\addplot[area legend,solid,fill=black!10!green,opacity=1.000000e-01,draw=none,forget plot]
table[row sep=crcr] {%
x	y\\
1	0.0120383933270016\\
2	0.344768511817382\\
3	11.1072621487528\\
4	13.7127741924107\\
5	13.6933587221848\\
6	13.4085122875145\\
7	13.7686863036304\\
8	14.093707937268\\
9	13.6860062917996\\
10	13.7052048270902\\
11	14.2075599310206\\
12	13.8350556032146\\
13	13.888452540332\\
14	13.9185610679227\\
15	0.931708310942652\\
16	0.908674323978748\\
17	0.737906412250545\\
18	0.939459638551977\\
19	0.935122190011521\\
20	0.944776438100476\\
21	0.942567881101373\\
22	0.949180643715074\\
23	0.946539224188454\\
24	0.743512737542939\\
25	0.853880479572336\\
26	0.937105367837514\\
27	0.931344003626043\\
28	0.942319789107299\\
29	0.945228245477849\\
30	0.936437751326982\\
30	1.03490657431206\\
29	1.00915174713091\\
28	1.00074974558948\\
27	1.03095978647428\\
26	1.03061800314325\\
25	1.13017803569925\\
24	1.25381768085598\\
23	1.016884188797\\
22	1.00725817544991\\
21	1.00409689392817\\
20	1.00619369080957\\
19	1.04756490342635\\
18	1.03706603237164\\
17	1.25800740302208\\
16	1.079410333586\\
15	1.53863750881069\\
14	15.828484509178\\
13	15.8664675608658\\
12	16.0039148914671\\
11	15.5788882395238\\
10	16.0210996676004\\
9	15.8271200156985\\
8	15.73197440925\\
7	15.8932179218215\\
6	16.1363022427564\\
5	16.00942755805\\
4	15.8770526369662\\
3	13.626077375317\\
2	0.657515337485459\\
1	0.0259769387405984\\
}--cycle;

\addplot [color=mycolor1,solid,forget plot]
  table[row sep=crcr]{%
1	0.0120383933270016\\
2	0.344768511817382\\
3	11.1072621487528\\
4	13.7127741924107\\
5	13.6933587221848\\
6	13.4085122875145\\
7	13.7686863036304\\
8	14.093707937268\\
9	13.6860062917996\\
10	13.7052048270902\\
11	14.2075599310206\\
12	13.8350556032146\\
13	13.888452540332\\
14	13.9185610679227\\
15	0.931708310942652\\
16	0.908674323978748\\
17	0.737906412250545\\
18	0.939459638551977\\
19	0.935122190011521\\
20	0.944776438100476\\
21	0.942567881101373\\
22	0.949180643715074\\
23	0.946539224188454\\
24	0.743512737542939\\
25	0.853880479572336\\
26	0.937105367837514\\
27	0.931344003626043\\
28	0.942319789107299\\
29	0.945228245477849\\
30	0.936437751326982\\
};
\addplot [color=mycolor1,solid,forget plot]
  table[row sep=crcr]{%
1	0.0259769387405984\\
2	0.657515337485459\\
3	13.626077375317\\
4	15.8770526369662\\
5	16.00942755805\\
6	16.1363022427564\\
7	15.8932179218215\\
8	15.73197440925\\
9	15.8271200156985\\
10	16.0210996676004\\
11	15.5788882395238\\
12	16.0039148914671\\
13	15.8664675608658\\
14	15.828484509178\\
15	1.53863750881069\\
16	1.079410333586\\
17	1.25800740302208\\
18	1.03706603237164\\
19	1.04756490342635\\
20	1.00619369080957\\
21	1.00409689392817\\
22	1.00725817544991\\
23	1.016884188797\\
24	1.25381768085598\\
25	1.13017803569925\\
26	1.03061800314325\\
27	1.03095978647428\\
28	1.00074974558948\\
29	1.00915174713091\\
30	1.03490657431206\\
};
\addplot [color=black!10!green,dashed]
  table[row sep=crcr]{%
1	0.0190076660338\\
2	0.50114192465142\\
3	12.3666697620349\\
4	14.7949134146885\\
5	14.8513931401174\\
6	14.7724072651355\\
7	14.830952112726\\
8	14.912841173259\\
9	14.756563153749\\
10	14.8631522473453\\
11	14.8932240852722\\
12	14.9194852473409\\
13	14.8774600505989\\
14	14.8735227885504\\
15	1.23517290987667\\
16	0.994042328782375\\
17	0.997956907636311\\
18	0.988262835461807\\
19	0.991343546718936\\
20	0.975485064455023\\
21	0.973332387514771\\
22	0.978219409582491\\
23	0.981711706492727\\
24	0.998665209199457\\
25	0.992029257635792\\
26	0.983861685490383\\
27	0.981151895050162\\
28	0.971534767348388\\
29	0.97718999630438\\
30	0.985672162819521\\
};
\addlegendentry{Panjer};

\addplot[area legend,solid,fill=black!10!blue,opacity=1.000000e-01,draw=none,forget plot]
table[row sep=crcr] {%
x	y\\
1	-0.00779533260380028\\
2	14.9149963408403\\
3	15.1369380602345\\
4	15.1401703731827\\
5	15.0896258006635\\
6	14.8791076120136\\
7	15.0790344546489\\
8	15.1688057652727\\
9	15.0710745227365\\
10	14.8908047308974\\
11	15.0823416595864\\
12	14.8915310454557\\
13	15.0661031217741\\
14	15.0317006873071\\
15	9.60601748412057\\
16	3.12477533215073\\
17	1.68286645527721\\
18	1.28212518236005\\
19	1.12146385728121\\
20	1.04864851764843\\
21	1.00979130265287\\
22	1.02241359200672\\
23	1.04552229176608\\
24	0.776936829568434\\
25	0.826298540773267\\
26	0.887264352994048\\
27	0.81262530584837\\
28	1.00563744784577\\
29	1.01951489428684\\
30	0.925343537699906\\
30	1.81501728721375\\
29	1.65060019001972\\
28	1.66052279022534\\
27	1.91930225508865\\
26	1.87227639630053\\
25	2.05450483362198\\
24	2.02152789880573\\
23	1.67131886319484\\
22	1.77714697405385\\
21	1.92503988045127\\
20	1.99220638555318\\
19	2.27846793379335\\
18	2.54603059539313\\
17	3.31303278444445\\
16	6.18789278273457\\
15	13.0149933418547\\
14	16.4185897018755\\
13	16.48122994301\\
12	16.6109073227319\\
11	16.315875257232\\
10	16.5874342795325\\
9	16.3688778454059\\
8	16.4213260827761\\
7	16.5028730392358\\
6	16.6516598922\\
5	16.5051012475656\\
4	16.6600258856962\\
3	17.1731639913679\\
2	19.2371369305149\\
1	6.37478760207861\\
}--cycle;

\addplot [color=mycolor2,solid,forget plot]
  table[row sep=crcr]{%
1	-0.00779533260380028\\
2	14.9149963408403\\
3	15.1369380602345\\
4	15.1401703731827\\
5	15.0896258006635\\
6	14.8791076120136\\
7	15.0790344546489\\
8	15.1688057652727\\
9	15.0710745227365\\
10	14.8908047308974\\
11	15.0823416595864\\
12	14.8915310454557\\
13	15.0661031217741\\
14	15.0317006873071\\
15	9.60601748412057\\
16	3.12477533215073\\
17	1.68286645527721\\
18	1.28212518236005\\
19	1.12146385728121\\
20	1.04864851764843\\
21	1.00979130265287\\
22	1.02241359200672\\
23	1.04552229176608\\
24	0.776936829568434\\
25	0.826298540773267\\
26	0.887264352994048\\
27	0.81262530584837\\
28	1.00563744784577\\
29	1.01951489428684\\
30	0.925343537699906\\
};
\addplot [color=mycolor2,solid,forget plot]
  table[row sep=crcr]{%
1	6.37478760207861\\
2	19.2371369305149\\
3	17.1731639913679\\
4	16.6600258856962\\
5	16.5051012475656\\
6	16.6516598922\\
7	16.5028730392358\\
8	16.4213260827761\\
9	16.3688778454059\\
10	16.5874342795325\\
11	16.315875257232\\
12	16.6109073227319\\
13	16.48122994301\\
14	16.4185897018755\\
15	13.0149933418547\\
16	6.18789278273457\\
17	3.31303278444445\\
18	2.54603059539313\\
19	2.27846793379335\\
20	1.99220638555318\\
21	1.92503988045127\\
22	1.77714697405385\\
23	1.67131886319484\\
24	2.02152789880573\\
25	2.05450483362198\\
26	1.87227639630053\\
27	1.91930225508865\\
28	1.66052279022534\\
29	1.65060019001972\\
30	1.81501728721375\\
};
\addplot [color=black!10!blue,dashed]
  table[row sep=crcr]{%
1	3.1834961347374\\
2	17.0760666356776\\
3	16.1550510258012\\
4	15.9000981294394\\
5	15.7973635241145\\
6	15.7653837521068\\
7	15.7909537469423\\
8	15.7950659240244\\
9	15.7199761840712\\
10	15.739119505215\\
11	15.6991084584092\\
12	15.7512191840938\\
13	15.7736665323921\\
14	15.7251451945913\\
15	11.3105054129877\\
16	4.65633405744265\\
17	2.49794961986083\\
18	1.91407788887659\\
19	1.69996589553728\\
20	1.5204274516008\\
21	1.46741559155207\\
22	1.39978028303029\\
23	1.35842057748046\\
24	1.39923236418708\\
25	1.44040168719762\\
26	1.37977037464729\\
27	1.36596378046851\\
28	1.33308011903555\\
29	1.33505754215328\\
30	1.37018041245683\\
};
\addlegendentry{CPHD};

\end{axis}
\end{tikzpicture}%

%% file: rmse_mean2.tikz
%
\definecolor{mycolor1}{rgb}{0.75000,0.97500,0.75000}%
\definecolor{mycolor2}{rgb}{0.75000,0.75000,0.97500}%
\begin{tikzpicture}

\begin{axis}[%
width=0.8\linewidth,
height=1.6in,
at={(1.011111in,0.641667in)},
scale only axis,
xmin=1,
xmax=30,
xlabel={time},
xmajorgrids,
ymin=0,
ymax=4,
ymajorgrids,
ylabel={RMSE},
axis x line*=bottom,
axis y line*=left,
legend style={at={(0.03,0.97)},anchor=north west,legend cell align=left,align=left,draw=white!15!black,fill opacity=0.6,draw opacity=1,text opacity=1,font=\footnotesize}
]

\addplot[area legend,solid,fill=black!10!red,opacity=1.000000e-01,draw=none,forget plot]
table[row sep=crcr] {%
x	y\\
1	0\\
2	0.0259613984275853\\
3	0.0287222494029831\\
4	0.0517283411212978\\
5	0.0534350650271038\\
6	0.0586966326161587\\
7	0.0938407252364642\\
8	0.101045446728505\\
9	0.124267174770742\\
10	0.148606200902613\\
11	0.164532730937703\\
12	0.169862046276022\\
13	0.15446827064687\\
14	0.142145769908614\\
15	0.142664194923853\\
16	-0.119808953677932\\
17	-0.107529995524163\\
18	-0.116874283844136\\
19	-0.113823613811556\\
20	-0.126738835009372\\
21	-0.151189015158923\\
22	-0.160750205774087\\
23	-0.16613818638779\\
24	-0.166239803637505\\
25	-0.191391650574393\\
26	-0.224139036027468\\
27	-0.228980620897713\\
28	-0.221469370397791\\
29	-0.226612590960604\\
30	-0.230831629244058\\
30	1.57555134705261\\
29	1.50096084857982\\
28	1.44564009970352\\
27	1.40007798133031\\
26	1.35275064526326\\
25	1.28951089412873\\
24	1.21695117039593\\
23	1.16662320859966\\
22	1.12220970190029\\
21	1.06314024268974\\
20	0.996300023875796\\
19	0.936865107468436\\
18	0.901975556132487\\
17	0.868784718358837\\
16	0.863718240326325\\
15	1.09905363672613\\
14	0.950215860641127\\
13	0.803156091070514\\
12	0.650758967605923\\
11	0.522446333224783\\
10	0.42382333329361\\
9	0.35299214493776\\
8	0.284255792014722\\
7	0.213117505647763\\
6	0.130404359993609\\
5	0.105339839339395\\
4	0.0857553316914863\\
3	0.0555657367812994\\
2	0.0355562126720676\\
1	0\\
}--cycle;

\addplot [color=pink!90!lightgray,solid,forget plot]
  table[row sep=crcr]{%
1	0\\
2	0.0259613984275853\\
3	0.0287222494029831\\
4	0.0517283411212978\\
5	0.0534350650271038\\
6	0.0586966326161587\\
7	0.0938407252364642\\
8	0.101045446728505\\
9	0.124267174770742\\
10	0.148606200902613\\
11	0.164532730937703\\
12	0.169862046276022\\
13	0.15446827064687\\
14	0.142145769908614\\
15	0.142664194923853\\
16	-0.119808953677932\\
17	-0.107529995524163\\
18	-0.116874283844136\\
19	-0.113823613811556\\
20	-0.126738835009372\\
21	-0.151189015158923\\
22	-0.160750205774087\\
23	-0.16613818638779\\
24	-0.166239803637505\\
25	-0.191391650574393\\
26	-0.224139036027468\\
27	-0.228980620897713\\
28	-0.221469370397791\\
29	-0.226612590960604\\
30	-0.230831629244058\\
};
\addplot [color=pink!90!lightgray,solid,forget plot]
  table[row sep=crcr]{%
1	0\\
2	0.0355562126720676\\
3	0.0555657367812994\\
4	0.0857553316914863\\
5	0.105339839339395\\
6	0.130404359993609\\
7	0.213117505647763\\
8	0.284255792014722\\
9	0.35299214493776\\
10	0.42382333329361\\
11	0.522446333224783\\
12	0.650758967605923\\
13	0.803156091070514\\
14	0.950215860641127\\
15	1.09905363672613\\
16	0.863718240326325\\
17	0.868784718358837\\
18	0.901975556132487\\
19	0.936865107468436\\
20	0.996300023875796\\
21	1.06314024268974\\
22	1.12220970190029\\
23	1.16662320859966\\
24	1.21695117039593\\
25	1.28951089412873\\
26	1.35275064526326\\
27	1.40007798133031\\
28	1.44564009970352\\
29	1.50096084857982\\
30	1.57555134705261\\
};
\addplot [color=black!10!red,solid]
  table[row sep=crcr]{%
1	0\\
2	0.0307588055498265\\
3	0.0421439930921413\\
4	0.068741836406392\\
5	0.0793874521832496\\
6	0.0945504963048839\\
7	0.153479115442114\\
8	0.192650619371614\\
9	0.238629659854251\\
10	0.286214767098112\\
11	0.343489532081243\\
12	0.410310506940972\\
13	0.478812180858692\\
14	0.546180815274871\\
15	0.620858915824993\\
16	0.371954643324197\\
17	0.380627361417337\\
18	0.392550636144175\\
19	0.41152074682844\\
20	0.434780594433212\\
21	0.455975613765407\\
22	0.480729748063101\\
23	0.500242511105934\\
24	0.525355683379211\\
25	0.549059621777168\\
26	0.564305804617893\\
27	0.585548680216296\\
28	0.612085364652864\\
29	0.637174128809608\\
30	0.672359858904276\\
};
\addlegendentry{PHD L1};

\addplot[area legend,solid,fill=black!10!red,opacity=1.000000e-01,draw=none,forget plot]
table[row sep=crcr] {%
x	y\\
1	0\\
2	0.0340947903279366\\
3	0.0773917417784637\\
4	0.148908733415202\\
5	0.20433460996423\\
6	0.259006594443945\\
7	0.369373058548131\\
8	0.490901168224085\\
9	0.613444378735908\\
10	0.715435441067351\\
11	0.799545931194634\\
12	0.902923272572187\\
13	1.03959599314045\\
14	1.20413454149311\\
15	1.38454303081234\\
16	1.34438128279174\\
17	1.24593591231308\\
18	1.2306154725327\\
19	1.26559793174582\\
20	1.33836676235493\\
21	1.46549211279366\\
22	1.61181990422704\\
23	1.7429799835288\\
24	1.87519954123612\\
25	1.99638626192684\\
26	2.11343715745522\\
27	2.22974202525913\\
28	2.32736785298046\\
29	2.4214381622905\\
30	2.51399310672971\\
30	3.5274637269194\\
29	3.39078686869439\\
28	3.26111472607228\\
27	3.11159681904407\\
26	2.94205142985144\\
25	2.78238000808591\\
24	2.62003592399017\\
23	2.45242191809103\\
22	2.28384858451736\\
21	2.108570561526\\
20	1.95190161825117\\
19	1.84414109965432\\
18	1.7783708711665\\
17	1.73892073687416\\
16	1.70242303401465\\
15	1.63103081714337\\
14	1.41369269216183\\
13	1.21832734090286\\
12	1.05651887563194\\
11	0.933389475126378\\
10	0.827478372619045\\
9	0.705279782314564\\
8	0.564324425339997\\
7	0.426374751956926\\
6	0.302089016078851\\
5	0.236452668203094\\
4	0.171257537634111\\
3	0.091374422458074\\
2	0.0421209857285367\\
1	0\\
}--cycle;

\addplot [color=pink!90!lightgray,solid,forget plot]
  table[row sep=crcr]{%
1	0\\
2	0.0340947903279366\\
3	0.0773917417784637\\
4	0.148908733415202\\
5	0.20433460996423\\
6	0.259006594443945\\
7	0.369373058548131\\
8	0.490901168224085\\
9	0.613444378735908\\
10	0.715435441067351\\
11	0.799545931194634\\
12	0.902923272572187\\
13	1.03959599314045\\
14	1.20413454149311\\
15	1.38454303081234\\
16	1.34438128279174\\
17	1.24593591231308\\
18	1.2306154725327\\
19	1.26559793174582\\
20	1.33836676235493\\
21	1.46549211279366\\
22	1.61181990422704\\
23	1.7429799835288\\
24	1.87519954123612\\
25	1.99638626192684\\
26	2.11343715745522\\
27	2.22974202525913\\
28	2.32736785298046\\
29	2.4214381622905\\
30	2.51399310672971\\
};
\addplot [color=pink!90!lightgray,solid,forget plot]
  table[row sep=crcr]{%
1	0\\
2	0.0421209857285367\\
3	0.091374422458074\\
4	0.171257537634111\\
5	0.236452668203094\\
6	0.302089016078851\\
7	0.426374751956926\\
8	0.564324425339997\\
9	0.705279782314564\\
10	0.827478372619045\\
11	0.933389475126378\\
12	1.05651887563194\\
13	1.21832734090286\\
14	1.41369269216183\\
15	1.63103081714337\\
16	1.70242303401465\\
17	1.73892073687416\\
18	1.7783708711665\\
19	1.84414109965432\\
20	1.95190161825117\\
21	2.108570561526\\
22	2.28384858451736\\
23	2.45242191809103\\
24	2.62003592399017\\
25	2.78238000808591\\
26	2.94205142985144\\
27	3.11159681904407\\
28	3.26111472607228\\
29	3.39078686869439\\
30	3.5274637269194\\
};
\addplot [color=black!10!red,dashed]
  table[row sep=crcr]{%
1	0\\
2	0.0381078880282366\\
3	0.0843830821182689\\
4	0.160083135524657\\
5	0.220393639083662\\
6	0.280547805261398\\
7	0.397873905252529\\
8	0.527612796782041\\
9	0.659362080525236\\
10	0.771456906843198\\
11	0.866467703160506\\
12	0.979721074102064\\
13	1.12896166702165\\
14	1.30891361682747\\
15	1.50778692397785\\
16	1.52340215840319\\
17	1.49242832459362\\
18	1.5044931718496\\
19	1.55486951570007\\
20	1.64513419030305\\
21	1.78703133715983\\
22	1.9478342443722\\
23	2.09770095080991\\
24	2.24761773261315\\
25	2.38938313500637\\
26	2.52774429365333\\
27	2.6706694221516\\
28	2.79424128952637\\
29	2.90611251549245\\
30	3.02072841682455\\
};
\addlegendentry{PHD L2};

\addplot[area legend,solid,fill=black!10!green,opacity=1.000000e-01,draw=none,forget plot]
table[row sep=crcr] {%
x	y\\
1	0\\
2	0.026077490287864\\
3	0.0312055437461912\\
4	0.0558717933026283\\
5	0.0570585236654828\\
6	0.0630094022017192\\
7	0.101660749069129\\
8	0.108415520269773\\
9	0.133048427563001\\
10	0.148878264000704\\
11	0.174047008298179\\
12	0.198585477951186\\
13	0.17983303782613\\
14	0.192296733721834\\
15	0.158215209264864\\
16	0.00800085543178536\\
17	-0.0125721733530861\\
18	-0.0254754816757582\\
19	-0.0296813980921812\\
20	-0.0140638612282906\\
21	-0.0218326611068313\\
22	-0.0475670873585419\\
23	-0.0486324683699867\\
24	-0.0397008951688517\\
25	-0.0529562682172007\\
26	-0.0564078686441947\\
27	-0.0730390384429959\\
28	-0.0766465376608013\\
29	-0.0910263044600009\\
30	-0.0844568168792165\\
30	1.27823962328613\\
29	1.23805208584098\\
28	1.19180744593525\\
27	1.13513189833398\\
26	1.08647359335691\\
25	1.04761800277482\\
24	1.00568589109873\\
23	0.973190807332493\\
22	0.938568125246386\\
21	0.887129173942079\\
20	0.841021094731183\\
19	0.814130779515496\\
18	0.783123214419243\\
17	0.754852050893122\\
16	0.701341076252408\\
15	1.05569985310888\\
14	0.844551715149985\\
13	0.731192277988958\\
12	0.592003761864646\\
11	0.505993159486884\\
10	0.426131414257543\\
9	0.350131393581887\\
8	0.286576598015551\\
7	0.210688053564811\\
6	0.129817011904327\\
5	0.1046874347706\\
4	0.0826346596935356\\
3	0.0556842662616704\\
2	0.0363916855559074\\
1	0\\
}--cycle;

\addplot [color=mycolor1,solid,forget plot]
  table[row sep=crcr]{%
1	0\\
2	0.026077490287864\\
3	0.0312055437461912\\
4	0.0558717933026283\\
5	0.0570585236654828\\
6	0.0630094022017192\\
7	0.101660749069129\\
8	0.108415520269773\\
9	0.133048427563001\\
10	0.148878264000704\\
11	0.174047008298179\\
12	0.198585477951186\\
13	0.17983303782613\\
14	0.192296733721834\\
15	0.158215209264864\\
16	0.00800085543178536\\
17	-0.0125721733530861\\
18	-0.0254754816757582\\
19	-0.0296813980921812\\
20	-0.0140638612282906\\
21	-0.0218326611068313\\
22	-0.0475670873585419\\
23	-0.0486324683699867\\
24	-0.0397008951688517\\
25	-0.0529562682172007\\
26	-0.0564078686441947\\
27	-0.0730390384429959\\
28	-0.0766465376608013\\
29	-0.0910263044600009\\
30	-0.0844568168792165\\
};
\addplot [color=mycolor1,solid,forget plot]
  table[row sep=crcr]{%
1	0\\
2	0.0363916855559074\\
3	0.0556842662616704\\
4	0.0826346596935356\\
5	0.1046874347706\\
6	0.129817011904327\\
7	0.210688053564811\\
8	0.286576598015551\\
9	0.350131393581887\\
10	0.426131414257543\\
11	0.505993159486884\\
12	0.592003761864646\\
13	0.731192277988958\\
14	0.844551715149985\\
15	1.05569985310888\\
16	0.701341076252408\\
17	0.754852050893122\\
18	0.783123214419243\\
19	0.814130779515496\\
20	0.841021094731183\\
21	0.887129173942079\\
22	0.938568125246386\\
23	0.973190807332493\\
24	1.00568589109873\\
25	1.04761800277482\\
26	1.08647359335691\\
27	1.13513189833398\\
28	1.19180744593525\\
29	1.23805208584098\\
30	1.27823962328613\\
};
\addplot [color=black!10!green,solid]
  table[row sep=crcr]{%
1	0\\
2	0.0312345879218857\\
3	0.0434449050039308\\
4	0.0692532264980819\\
5	0.0808729792180413\\
6	0.0964132070530232\\
7	0.15617440131697\\
8	0.197496059142662\\
9	0.241589910572444\\
10	0.287504839129124\\
11	0.340020083892531\\
12	0.395294619907916\\
13	0.455512657907544\\
14	0.518424224435909\\
15	0.606957531186873\\
16	0.354670965842097\\
17	0.371139938770018\\
18	0.378823866371743\\
19	0.392224690711658\\
20	0.413478616751446\\
21	0.432648256417624\\
22	0.445500518943922\\
23	0.462279169481253\\
24	0.48299249796494\\
25	0.497330867278809\\
26	0.515032862356358\\
27	0.531046429945493\\
28	0.557580454137223\\
29	0.57351289069049\\
30	0.596891403203454\\
};
\addlegendentry{Panjer L1};

\addplot[area legend,solid,fill=black!10!green,opacity=1.000000e-01,draw=none,forget plot]
table[row sep=crcr] {%
x	y\\
1	0\\
2	0.0342865942592827\\
3	0.0760783570881424\\
4	0.138256145644889\\
5	0.187553288824267\\
6	0.235360062460659\\
7	0.341276385354558\\
8	0.457380133473282\\
9	0.575857373229245\\
10	0.674915182897524\\
11	0.759189331362056\\
12	0.857798786354171\\
13	0.991883979669499\\
14	1.15667654973041\\
15	1.33775718145999\\
16	1.31054275652372\\
17	1.09774896018609\\
18	1.06883480604434\\
19	1.09747138970398\\
20	1.16575784048788\\
21	1.27174310861288\\
22	1.40562486213646\\
23	1.52362566524447\\
24	1.64523700947749\\
25	1.75148938158095\\
26	1.85568860847108\\
27	1.96144226081955\\
28	2.05542565582869\\
29	2.13517016687516\\
30	2.21024992256634\\
30	3.12596974984592\\
29	3.00973338719854\\
28	2.8966168101497\\
27	2.77165663740676\\
26	2.62629122985941\\
25	2.49148624421525\\
24	2.34737382160629\\
23	2.19599511749729\\
22	2.04348579600346\\
21	1.87868682077465\\
20	1.72788071759007\\
19	1.63925159191863\\
18	1.59085070065662\\
17	1.59317598436041\\
16	1.6696260664262\\
15	1.60405268802594\\
14	1.39494958830476\\
13	1.20244869068038\\
12	1.03321979408863\\
11	0.899827315131383\\
10	0.793460494173965\\
9	0.673762579131523\\
8	0.534723864425237\\
7	0.398088373432657\\
6	0.277468997239595\\
5	0.216612645725524\\
4	0.158898766857328\\
3	0.0894668093230475\\
2	0.0413297674947783\\
1	0\\
}--cycle;

\addplot [color=mycolor1,solid,forget plot]
  table[row sep=crcr]{%
1	0\\
2	0.0342865942592827\\
3	0.0760783570881424\\
4	0.138256145644889\\
5	0.187553288824267\\
6	0.235360062460659\\
7	0.341276385354558\\
8	0.457380133473282\\
9	0.575857373229245\\
10	0.674915182897524\\
11	0.759189331362056\\
12	0.857798786354171\\
13	0.991883979669499\\
14	1.15667654973041\\
15	1.33775718145999\\
16	1.31054275652372\\
17	1.09774896018609\\
18	1.06883480604434\\
19	1.09747138970398\\
20	1.16575784048788\\
21	1.27174310861288\\
22	1.40562486213646\\
23	1.52362566524447\\
24	1.64523700947749\\
25	1.75148938158095\\
26	1.85568860847108\\
27	1.96144226081955\\
28	2.05542565582869\\
29	2.13517016687516\\
30	2.21024992256634\\
};
\addplot [color=mycolor1,solid,forget plot]
  table[row sep=crcr]{%
1	0\\
2	0.0413297674947783\\
3	0.0894668093230475\\
4	0.158898766857328\\
5	0.216612645725524\\
6	0.277468997239595\\
7	0.398088373432657\\
8	0.534723864425237\\
9	0.673762579131523\\
10	0.793460494173965\\
11	0.899827315131383\\
12	1.03321979408863\\
13	1.20244869068038\\
14	1.39494958830476\\
15	1.60405268802594\\
16	1.6696260664262\\
17	1.59317598436041\\
18	1.59085070065662\\
19	1.63925159191863\\
20	1.72788071759007\\
21	1.87868682077465\\
22	2.04348579600346\\
23	2.19599511749729\\
24	2.34737382160629\\
25	2.49148624421525\\
26	2.62629122985941\\
27	2.77165663740676\\
28	2.8966168101497\\
29	3.00973338719854\\
30	3.12596974984592\\
};
\addplot [color=black!10!green,dashed]
  table[row sep=crcr]{%
1	0\\
2	0.0378081808770305\\
3	0.0827725832055949\\
4	0.148577456251109\\
5	0.202082967274895\\
6	0.256414529850127\\
7	0.369682379393607\\
8	0.496051998949259\\
9	0.624809976180384\\
10	0.734187838535744\\
11	0.82950832324672\\
12	0.9455092902214\\
13	1.09716633517494\\
14	1.27581306901758\\
15	1.47090493474297\\
16	1.49008441147496\\
17	1.34546247227325\\
18	1.32984275335048\\
19	1.36836149081131\\
20	1.44681927903898\\
21	1.57521496469377\\
22	1.72455532906996\\
23	1.85981039137088\\
24	1.99630541554189\\
25	2.1214878128981\\
26	2.24098991916524\\
27	2.36654944911315\\
28	2.47602123298919\\
29	2.57245177703685\\
30	2.66810983620613\\
};
\addlegendentry{Panjer L2};

\addplot[area legend,solid,fill=black!10!blue,opacity=1.000000e-01,draw=none,forget plot]
table[row sep=crcr] {%
x	y\\
1	0\\
2	0.0264140388991936\\
3	0.0308851941715641\\
4	0.0507398799025889\\
5	0.0526358717547444\\
6	0.054219161848714\\
7	0.104802585949575\\
8	0.119077579918211\\
9	0.146846838133543\\
10	0.165637369510768\\
11	0.175026813583533\\
12	0.179808106301529\\
13	0.152409209826843\\
14	0.142681441807321\\
15	0.109006563914593\\
16	-0.0211806917031352\\
17	-0.0200766075437498\\
18	-0.0245062201184041\\
19	-0.0265754309669128\\
20	-0.0160892177542896\\
21	-0.0249475391712543\\
22	-0.0178579548811486\\
23	0.00686216481703206\\
24	0.0168616948409481\\
25	0.00337821692974405\\
26	-0.0112248733781366\\
27	-0.0161628088747708\\
28	-0.03622676065347\\
29	-0.0372295781470796\\
30	-0.027210415205503\\
30	1.25806099065421\\
29	1.21511201967447\\
28	1.1717465155606\\
27	1.11941028113521\\
26	1.07582406814836\\
25	1.0241058691741\\
24	0.970609788980623\\
23	0.906307692602273\\
22	0.886607299989398\\
21	0.841184349947008\\
20	0.790330774550828\\
19	0.779010352766905\\
18	0.752650610692138\\
17	0.720867095210481\\
16	0.654850582476505\\
15	1.03761008748694\\
14	0.868331042999594\\
13	0.755637654476177\\
12	0.604826914859961\\
11	0.494858073296641\\
10	0.404168967984749\\
9	0.338803953121973\\
8	0.271916323987617\\
7	0.211610186151305\\
6	0.137849473568479\\
5	0.106716957017806\\
4	0.0863947656723569\\
3	0.0581022599387899\\
2	0.0359568198369431\\
1	0\\
}--cycle;

\addplot [color=mycolor2,solid,forget plot]
  table[row sep=crcr]{%
1	0\\
2	0.0264140388991936\\
3	0.0308851941715641\\
4	0.0507398799025889\\
5	0.0526358717547444\\
6	0.054219161848714\\
7	0.104802585949575\\
8	0.119077579918211\\
9	0.146846838133543\\
10	0.165637369510768\\
11	0.175026813583533\\
12	0.179808106301529\\
13	0.152409209826843\\
14	0.142681441807321\\
15	0.109006563914593\\
16	-0.0211806917031352\\
17	-0.0200766075437498\\
18	-0.0245062201184041\\
19	-0.0265754309669128\\
20	-0.0160892177542896\\
21	-0.0249475391712543\\
22	-0.0178579548811486\\
23	0.00686216481703206\\
24	0.0168616948409481\\
25	0.00337821692974405\\
26	-0.0112248733781366\\
27	-0.0161628088747708\\
28	-0.03622676065347\\
29	-0.0372295781470796\\
30	-0.027210415205503\\
};
\addplot [color=mycolor2,solid,forget plot]
  table[row sep=crcr]{%
1	0\\
2	0.0359568198369431\\
3	0.0581022599387899\\
4	0.0863947656723569\\
5	0.106716957017806\\
6	0.137849473568479\\
7	0.211610186151305\\
8	0.271916323987617\\
9	0.338803953121973\\
10	0.404168967984749\\
11	0.494858073296641\\
12	0.604826914859961\\
13	0.755637654476177\\
14	0.868331042999594\\
15	1.03761008748694\\
16	0.654850582476505\\
17	0.720867095210481\\
18	0.752650610692138\\
19	0.779010352766905\\
20	0.790330774550828\\
21	0.841184349947008\\
22	0.886607299989398\\
23	0.906307692602273\\
24	0.970609788980623\\
25	1.0241058691741\\
26	1.07582406814836\\
27	1.11941028113521\\
28	1.1717465155606\\
29	1.21511201967447\\
30	1.25806099065421\\
};
\addplot [color=black!10!blue,solid]
  table[row sep=crcr]{%
1	0\\
2	0.0311854293680683\\
3	0.044493727055177\\
4	0.0685673227874729\\
5	0.079676414386275\\
6	0.0960343177085967\\
7	0.15820638605044\\
8	0.195496951952914\\
9	0.242825395627758\\
10	0.284903168747759\\
11	0.334942443440087\\
12	0.392317510580745\\
13	0.45402343215151\\
14	0.505506242403458\\
15	0.573308325700766\\
16	0.316834945386685\\
17	0.350395243833366\\
18	0.364072195286867\\
19	0.376217460899996\\
20	0.387120778398269\\
21	0.408118405387877\\
22	0.434374672554125\\
23	0.456584928709653\\
24	0.493735741910786\\
25	0.513742043051923\\
26	0.532299597385111\\
27	0.551623736130217\\
28	0.567759877453566\\
29	0.588941220763695\\
30	0.615425287724356\\
};
\addlegendentry{CPHD L1};

\addplot[area legend,solid,fill=black!10!blue,opacity=1.000000e-01,draw=none,forget plot]
table[row sep=crcr] {%
x	y\\
1	0\\
2	0.0345050444957009\\
3	0.0774976696294275\\
4	0.14871132051776\\
5	0.204614321046549\\
6	0.259260382280008\\
7	0.371101646360353\\
8	0.493489731855954\\
9	0.619719083402088\\
10	0.725926106242652\\
11	0.81235179153164\\
12	0.916112616906345\\
13	1.05152507786074\\
14	1.2170572734162\\
15	1.40193312068922\\
16	1.36608296848074\\
17	1.17855513628768\\
18	1.14193425309194\\
19	1.17200055755784\\
20	1.23526195878532\\
21	1.34629592288675\\
22	1.47204591966096\\
23	1.59168468510339\\
24	1.70855551549902\\
25	1.81088146484603\\
26	1.91091121560327\\
27	2.01954464993305\\
28	2.11292387341732\\
29	2.18814875882477\\
30	2.26122841400151\\
30	3.47233962414666\\
29	3.33614908062583\\
28	3.20449066718675\\
27	3.05836040132684\\
26	2.89387629409054\\
25	2.7363199317007\\
24	2.56552551672971\\
23	2.39201883169347\\
22	2.22240739093952\\
21	2.04127718159117\\
20	1.87778627649119\\
19	1.77040422000827\\
18	1.72030668536294\\
17	1.6985256817644\\
16	1.72162558115753\\
15	1.62440941120056\\
14	1.40862435629566\\
13	1.21575999210479\\
12	1.05327572562736\\
11	0.928559325318761\\
10	0.822395643653305\\
9	0.702939029437979\\
8	0.563854894889613\\
7	0.426115494483879\\
6	0.303589396643114\\
5	0.236926504578053\\
4	0.171780939218782\\
3	0.0915678332568607\\
2	0.0414240827836043\\
1	0\\
}--cycle;

\addplot [color=mycolor2,solid,forget plot]
  table[row sep=crcr]{%
1	0\\
2	0.0345050444957009\\
3	0.0774976696294275\\
4	0.14871132051776\\
5	0.204614321046549\\
6	0.259260382280008\\
7	0.371101646360353\\
8	0.493489731855954\\
9	0.619719083402088\\
10	0.725926106242652\\
11	0.81235179153164\\
12	0.916112616906345\\
13	1.05152507786074\\
14	1.2170572734162\\
15	1.40193312068922\\
16	1.36608296848074\\
17	1.17855513628768\\
18	1.14193425309194\\
19	1.17200055755784\\
20	1.23526195878532\\
21	1.34629592288675\\
22	1.47204591966096\\
23	1.59168468510339\\
24	1.70855551549902\\
25	1.81088146484603\\
26	1.91091121560327\\
27	2.01954464993305\\
28	2.11292387341732\\
29	2.18814875882477\\
30	2.26122841400151\\
};
\addplot [color=mycolor2,solid,forget plot]
  table[row sep=crcr]{%
1	0\\
2	0.0414240827836043\\
3	0.0915678332568607\\
4	0.171780939218782\\
5	0.236926504578053\\
6	0.303589396643114\\
7	0.426115494483879\\
8	0.563854894889613\\
9	0.702939029437979\\
10	0.822395643653305\\
11	0.928559325318761\\
12	1.05327572562736\\
13	1.21575999210479\\
14	1.40862435629566\\
15	1.62440941120056\\
16	1.72162558115753\\
17	1.6985256817644\\
18	1.72030668536294\\
19	1.77040422000827\\
20	1.87778627649119\\
21	2.04127718159117\\
22	2.22240739093952\\
23	2.39201883169347\\
24	2.56552551672971\\
25	2.7363199317007\\
26	2.89387629409054\\
27	3.05836040132684\\
28	3.20449066718675\\
29	3.33614908062583\\
30	3.47233962414666\\
};
\addplot [color=black!10!blue,dashed]
  table[row sep=crcr]{%
1	0\\
2	0.0379645636396526\\
3	0.0845327514431441\\
4	0.160246129868271\\
5	0.220770412812301\\
6	0.281424889461561\\
7	0.398608570422116\\
8	0.528672313372783\\
9	0.661329056420034\\
10	0.774160874947978\\
11	0.8704555584252\\
12	0.98469417126685\\
13	1.13364253498277\\
14	1.31284081485593\\
15	1.51317126594489\\
16	1.54385427481914\\
17	1.43854040902604\\
18	1.43112046922744\\
19	1.47120238878306\\
20	1.55652411763826\\
21	1.69378655223896\\
22	1.84722665530024\\
23	1.99185175839843\\
24	2.13704051611436\\
25	2.27360069827336\\
26	2.40239375484691\\
27	2.53895252562994\\
28	2.65870727030203\\
29	2.7621489197253\\
30	2.86678401907408\\
};
\addlegendentry{CPHD L2};

\end{axis}
\end{tikzpicture}%

%% file: card_mean2_likeli1.tikz
%
\definecolor{mycolor1}{rgb}{0.75000,0.97500,0.75000}%
\definecolor{mycolor2}{rgb}{0.75000,0.75000,0.97500}%
\begin{tikzpicture}

\begin{axis}[%
width=0.8\linewidth,
height=1.2in,
at={(1.011111in,0.641667in)},
scale only axis,
xmin=1,
xmax=30,
xlabel={time},
ymin=-1,
ymax=35,
ylabel={\# targets},
axis x line*=bottom,
axis y line*=left,
legend style={at={(0.97,0.03)},anchor=south east,legend cell align=left,align=left,draw=white!15!black,fill opacity=0.6,draw opacity=1,text opacity=1,font=\footnotesize}
]
\addplot [color=black,solid]
  table[row sep=crcr]{%
1	15\\
2	15\\
3	15\\
4	15\\
5	15\\
6	15\\
7	15\\
8	15\\
9	15\\
10	15\\
11	15\\
12	15\\
13	15\\
14	15\\
15	30\\
16	30\\
17	30\\
18	30\\
19	30\\
20	30\\
21	30\\
22	30\\
23	30\\
24	30\\
25	30\\
26	30\\
27	30\\
28	30\\
29	30\\
30	30\\
};
\addlegendentry{gt};

\addplot[area legend,solid,fill=black!10!red,opacity=1.000000e-01,draw=none,forget plot]
table[row sep=crcr] {%
x	y\\
1	1.5895264110252\\
2	14.1366483378613\\
3	14.2383065314578\\
4	14.8342161437969\\
5	14.8062405958309\\
6	14.5263225987559\\
7	14.7754428561671\\
8	14.8734032272406\\
9	14.8268217823016\\
10	14.6414775363967\\
11	14.7888513379419\\
12	14.902479398486\\
13	14.9790892985309\\
14	14.7958001464373\\
15	15.9875292791285\\
16	29.118512655791\\
17	29.4860696649301\\
18	29.692865308058\\
19	29.6245942435289\\
20	29.2156042048266\\
21	29.2068318459677\\
22	29.3669910751196\\
23	29.4724308881625\\
24	29.2321528258014\\
25	29.139966829008\\
26	29.767632465257\\
27	29.4222457560203\\
28	29.376520865876\\
29	29.8182158327123\\
30	29.7109091420745\\
30	32.3238330444491\\
29	32.2631828041292\\
28	32.2290213266293\\
27	32.4119788783014\\
26	32.0223642374232\\
25	32.3857922467192\\
24	32.4770346740865\\
23	32.270804656586\\
22	32.5288225810857\\
21	32.7798711215263\\
20	32.5224378363164\\
19	32.5554415739262\\
18	32.254752213576\\
17	31.8971554662288\\
16	32.1161342149808\\
15	18.3813149548534\\
14	17.2516980265044\\
13	17.007411697123\\
12	16.9848022897383\\
11	17.1230433579652\\
10	17.1514211830162\\
9	17.2469456307056\\
8	17.196900170823\\
7	16.995806997168\\
6	16.9925526684773\\
5	17.1599199036833\\
4	16.7836179071499\\
3	17.4277672239672\\
2	16.7823625572951\\
1	2.70965451824017\\
}--cycle;

\addplot [color=pink!90!lightgray,solid,forget plot]
  table[row sep=crcr]{%
1	1.5895264110252\\
2	14.1366483378613\\
3	14.2383065314578\\
4	14.8342161437969\\
5	14.8062405958309\\
6	14.5263225987559\\
7	14.7754428561671\\
8	14.8734032272406\\
9	14.8268217823016\\
10	14.6414775363967\\
11	14.7888513379419\\
12	14.902479398486\\
13	14.9790892985309\\
14	14.7958001464373\\
15	15.9875292791285\\
16	29.118512655791\\
17	29.4860696649301\\
18	29.692865308058\\
19	29.6245942435289\\
20	29.2156042048266\\
21	29.2068318459677\\
22	29.3669910751196\\
23	29.4724308881625\\
24	29.2321528258014\\
25	29.139966829008\\
26	29.767632465257\\
27	29.4222457560203\\
28	29.376520865876\\
29	29.8182158327123\\
30	29.7109091420745\\
};
\addplot [color=pink!90!lightgray,solid,forget plot]
  table[row sep=crcr]{%
1	2.70965451824017\\
2	16.7823625572951\\
3	17.4277672239672\\
4	16.7836179071499\\
5	17.1599199036833\\
6	16.9925526684773\\
7	16.995806997168\\
8	17.196900170823\\
9	17.2469456307056\\
10	17.1514211830162\\
11	17.1230433579652\\
12	16.9848022897383\\
13	17.007411697123\\
14	17.2516980265044\\
15	18.3813149548534\\
16	32.1161342149808\\
17	31.8971554662288\\
18	32.254752213576\\
19	32.5554415739262\\
20	32.5224378363164\\
21	32.7798711215263\\
22	32.5288225810857\\
23	32.270804656586\\
24	32.4770346740865\\
25	32.3857922467192\\
26	32.0223642374232\\
27	32.4119788783014\\
28	32.2290213266293\\
29	32.2631828041292\\
30	32.3238330444491\\
};
\addplot [color=black!10!red,solid]
  table[row sep=crcr]{%
1	2.14959046463268\\
2	15.4595054475782\\
3	15.8330368777125\\
4	15.8089170254734\\
5	15.9830802497571\\
6	15.7594376336166\\
7	15.8856249266675\\
8	16.0351516990318\\
9	16.0368837065036\\
10	15.8964493597064\\
11	15.9559473479535\\
12	15.9436408441121\\
13	15.993250497827\\
14	16.0237490864709\\
15	17.1844221169909\\
16	30.6173234353859\\
17	30.6916125655794\\
18	30.973808760817\\
19	31.0900179087275\\
20	30.8690210205715\\
21	30.993351483747\\
22	30.9479068281027\\
23	30.8716177723742\\
24	30.854593749944\\
25	30.7628795378636\\
26	30.8949983513401\\
27	30.9171123171608\\
28	30.8027710962527\\
29	31.0406993184207\\
30	31.0173710932618\\
};
\addlegendentry{PHD};

\addplot[area legend,solid,fill=black!10!green,opacity=1.000000e-01,draw=none,forget plot]
table[row sep=crcr] {%
x	y\\
1	1.3607361371618\\
2	14.7946870085108\\
3	14.5265801993649\\
4	14.8187115182093\\
5	14.7923012235904\\
6	14.5092456889384\\
7	14.7699782079538\\
8	14.8610782563177\\
9	14.8123378076358\\
10	14.6241530684997\\
11	14.7761958640464\\
12	14.8896582991427\\
13	14.9634205983912\\
14	14.7754305913058\\
15	16.0128275992531\\
16	29.1948896566114\\
17	29.5824901458275\\
18	29.715884579601\\
19	29.6620932640333\\
20	29.2556302323482\\
21	29.2486103375675\\
22	29.4043629144724\\
23	29.5207854970423\\
24	29.2703227225098\\
25	29.180179135425\\
26	29.7946026576013\\
27	29.4534578626146\\
28	29.4107410372433\\
29	29.8426835557481\\
30	29.7402660670005\\
30	32.3224763247365\\
29	32.2588529771692\\
28	32.2405071217085\\
27	32.4005161140525\\
26	32.0263015851033\\
25	32.3883247270268\\
24	32.4690296431269\\
23	32.2515881117554\\
22	32.5216523726199\\
21	32.7721453137872\\
20	32.5164646787628\\
19	32.5450751062722\\
18	32.2687035902111\\
17	31.898010921264\\
16	32.2514441329402\\
15	18.3908170329615\\
14	17.2635658454151\\
13	17.0071399392342\\
12	16.9823823976532\\
11	17.1188524740271\\
10	17.1534530672219\\
9	17.2500164277031\\
8	17.1864085305369\\
7	16.9805586938764\\
6	16.9931838627377\\
5	17.161104009736\\
4	16.7951666663925\\
3	17.44081830144\\
2	18.4887267327471\\
1	8.43897892705005\\
}--cycle;

\addplot [color=mycolor1,solid,forget plot]
  table[row sep=crcr]{%
1	1.3607361371618\\
2	14.7946870085108\\
3	14.5265801993649\\
4	14.8187115182093\\
5	14.7923012235904\\
6	14.5092456889384\\
7	14.7699782079538\\
8	14.8610782563177\\
9	14.8123378076358\\
10	14.6241530684997\\
11	14.7761958640464\\
12	14.8896582991427\\
13	14.9634205983912\\
14	14.7754305913058\\
15	16.0128275992531\\
16	29.1948896566114\\
17	29.5824901458275\\
18	29.715884579601\\
19	29.6620932640333\\
20	29.2556302323482\\
21	29.2486103375675\\
22	29.4043629144724\\
23	29.5207854970423\\
24	29.2703227225098\\
25	29.180179135425\\
26	29.7946026576013\\
27	29.4534578626146\\
28	29.4107410372433\\
29	29.8426835557481\\
30	29.7402660670005\\
};
\addplot [color=mycolor1,solid,forget plot]
  table[row sep=crcr]{%
1	8.43897892705005\\
2	18.4887267327471\\
3	17.44081830144\\
4	16.7951666663925\\
5	17.161104009736\\
6	16.9931838627377\\
7	16.9805586938764\\
8	17.1864085305369\\
9	17.2500164277031\\
10	17.1534530672219\\
11	17.1188524740271\\
12	16.9823823976532\\
13	17.0071399392342\\
14	17.2635658454151\\
15	18.3908170329615\\
16	32.2514441329402\\
17	31.898010921264\\
18	32.2687035902111\\
19	32.5450751062722\\
20	32.5164646787628\\
21	32.7721453137872\\
22	32.5216523726199\\
23	32.2515881117554\\
24	32.4690296431269\\
25	32.3883247270268\\
26	32.0263015851033\\
27	32.4005161140525\\
28	32.2405071217085\\
29	32.2588529771692\\
30	32.3224763247365\\
};
\addplot [color=black!10!green,solid]
  table[row sep=crcr]{%
1	4.89985753210592\\
2	16.6417068706289\\
3	15.9836992504024\\
4	15.8069390923009\\
5	15.9767026166632\\
6	15.751214775838\\
7	15.8752684509151\\
8	16.0237433934273\\
9	16.0311771176694\\
10	15.8888030678608\\
11	15.9475241690368\\
12	15.9360203483979\\
13	15.9852802688127\\
14	16.0194982183604\\
15	17.2018223161073\\
16	30.7231668947758\\
17	30.7402505335458\\
18	30.9922940849061\\
19	31.1035841851527\\
20	30.8860474555555\\
21	31.0103778256773\\
22	30.9630076435462\\
23	30.8861868043988\\
24	30.8696761828184\\
25	30.7842519312259\\
26	30.9104521213523\\
27	30.9269869883335\\
28	30.8256240794759\\
29	31.0507682664586\\
30	31.0313711958685\\
};
\addlegendentry{Panjer};

\addplot[area legend,solid,fill=black!10!blue,opacity=1.000000e-01,draw=none,forget plot]
table[row sep=crcr] {%
x	y\\
1	-0.0532085681844996\\
2	14.7762195589625\\
3	14.9318452201393\\
4	15.047661379624\\
5	14.9893558855135\\
6	14.9421147101408\\
7	14.9862641507106\\
8	14.9892528367978\\
9	15.0532065269085\\
10	15.0384711365427\\
11	15.0365369529917\\
12	15.0989910810924\\
13	15.0211516310926\\
14	15.087435197116\\
15	15.6225491466678\\
16	29.4630659625011\\
17	29.5479394199866\\
18	29.8187288656313\\
19	29.7883222727443\\
20	29.8795639220538\\
21	29.5847870636682\\
22	29.7575685447261\\
23	29.8796375026994\\
24	29.7633960522396\\
25	29.6086151316568\\
26	29.9507457457157\\
27	29.6457128866063\\
28	29.7900722409315\\
29	29.8477181789095\\
30	29.8827298045896\\
30	31.6794277785473\\
29	31.6604073626162\\
28	31.5072112326829\\
27	31.7662469048021\\
26	31.4237794618084\\
25	31.6623466735332\\
24	31.5974340432111\\
23	31.5437680585739\\
22	31.7648348147075\\
21	31.9857045752535\\
20	31.5996214918427\\
19	31.8600340696814\\
18	31.6340908020657\\
17	31.7080931394423\\
16	32.7085085792793\\
15	17.2590733923114\\
14	16.52109911764\\
13	16.5221158517089\\
12	16.380110562772\\
11	16.51958506063\\
10	16.5180336691914\\
9	16.5835058042574\\
8	16.5771310439135\\
7	16.4224255770029\\
6	16.4330045892729\\
5	16.6670573018249\\
4	16.6517118046547\\
3	17.3657184875901\\
2	19.3489044897177\\
1	6.43238142306383\\
}--cycle;

\addplot [color=mycolor2,solid,forget plot]
  table[row sep=crcr]{%
1	-0.0532085681844996\\
2	14.7762195589625\\
3	14.9318452201393\\
4	15.047661379624\\
5	14.9893558855135\\
6	14.9421147101408\\
7	14.9862641507106\\
8	14.9892528367978\\
9	15.0532065269085\\
10	15.0384711365427\\
11	15.0365369529917\\
12	15.0989910810924\\
13	15.0211516310926\\
14	15.087435197116\\
15	15.6225491466678\\
16	29.4630659625011\\
17	29.5479394199866\\
18	29.8187288656313\\
19	29.7883222727443\\
20	29.8795639220538\\
21	29.5847870636682\\
22	29.7575685447261\\
23	29.8796375026994\\
24	29.7633960522396\\
25	29.6086151316568\\
26	29.9507457457157\\
27	29.6457128866063\\
28	29.7900722409315\\
29	29.8477181789095\\
30	29.8827298045896\\
};
\addplot [color=mycolor2,solid,forget plot]
  table[row sep=crcr]{%
1	6.43238142306383\\
2	19.3489044897177\\
3	17.3657184875901\\
4	16.6517118046547\\
5	16.6670573018249\\
6	16.4330045892729\\
7	16.4224255770029\\
8	16.5771310439135\\
9	16.5835058042574\\
10	16.5180336691914\\
11	16.51958506063\\
12	16.380110562772\\
13	16.5221158517089\\
14	16.52109911764\\
15	17.2590733923114\\
16	32.7085085792793\\
17	31.7080931394423\\
18	31.6340908020657\\
19	31.8600340696814\\
20	31.5996214918427\\
21	31.9857045752535\\
22	31.7648348147075\\
23	31.5437680585739\\
24	31.5974340432111\\
25	31.6623466735332\\
26	31.4237794618084\\
27	31.7662469048021\\
28	31.5072112326829\\
29	31.6604073626162\\
30	31.6794277785473\\
};
\addplot [color=black!10!blue,solid]
  table[row sep=crcr]{%
1	3.18958642743966\\
2	17.0625620243401\\
3	16.1487818538647\\
4	15.8496865921393\\
5	15.8282065936692\\
6	15.6875596497068\\
7	15.7043448638568\\
8	15.7831919403557\\
9	15.818356165583\\
10	15.7782524028671\\
11	15.7780610068108\\
12	15.7395508219322\\
13	15.7716337414007\\
14	15.804267157378\\
15	16.4408112694896\\
16	31.0857872708902\\
17	30.6280162797145\\
18	30.7264098338485\\
19	30.8241781712128\\
20	30.7395927069482\\
21	30.7852458194609\\
22	30.7612016797168\\
23	30.7117027806366\\
24	30.6804150477253\\
25	30.635480902595\\
26	30.687262603762\\
27	30.7059798957042\\
28	30.6486417368072\\
29	30.7540627707629\\
30	30.7810787915685\\
};
\addlegendentry{CPHD};

\end{axis}
\end{tikzpicture}%

%% file: card_mean2_likeli2.tikz
%
\definecolor{mycolor1}{rgb}{0.75000,0.97500,0.75000}%
\definecolor{mycolor2}{rgb}{0.75000,0.75000,0.97500}%
\begin{tikzpicture}

\begin{axis}[%
width=0.8\linewidth,
height=1.2in,
at={(1.011111in,0.641667in)},
scale only axis,
xmin=1,
xmax=30,
xlabel={time},
ymin=-1,
ymax=35,
ylabel={\# targets},
axis x line*=bottom,
axis y line*=left,
legend style={at={(0.97,0.03)},anchor=south east,legend cell align=left,align=left,draw=white!15!black,fill opacity=0.6,draw opacity=1,text opacity=1,font=\footnotesize}
]
\addplot [color=black,solid]
  table[row sep=crcr]{%
1	15\\
2	15\\
3	15\\
4	15\\
5	15\\
6	15\\
7	15\\
8	15\\
9	15\\
10	15\\
11	15\\
12	15\\
13	15\\
14	15\\
15	30\\
16	30\\
17	30\\
18	30\\
19	30\\
20	30\\
21	30\\
22	30\\
23	30\\
24	30\\
25	30\\
26	30\\
27	30\\
28	30\\
29	30\\
30	30\\
};
\addlegendentry{gt};

\addplot[area legend,solid,fill=black!10!red,opacity=1.000000e-01,draw=none,forget plot]
table[row sep=crcr] {%
x	y\\
1	1.5895264110252\\
2	14.1531574184214\\
3	14.2428222208178\\
4	14.8129418489013\\
5	14.8166133007521\\
6	14.5227132248387\\
7	14.7682107217163\\
8	14.8656858415678\\
9	14.8417190510719\\
10	14.6510480550972\\
11	14.8077700293205\\
12	14.8759537270403\\
13	14.9977724509533\\
14	14.8205509703365\\
15	16.0533979779979\\
16	25.5080818586639\\
17	26.4448622740084\\
18	28.2061619250339\\
19	28.8175615212135\\
20	28.783818329212\\
21	29.0385003564394\\
22	29.2319416385886\\
23	29.4006350564982\\
24	29.1850316170327\\
25	29.0855208915172\\
26	29.6800574177978\\
27	29.4350251421715\\
28	29.364886971433\\
29	29.8410149935438\\
30	29.7211786426927\\
30	32.2951530035691\\
29	32.1951559391823\\
28	32.193661088952\\
27	32.3379169116129\\
26	31.9884028613395\\
25	32.3498474479374\\
24	32.375503073644\\
23	32.1135275265299\\
22	32.3007404851282\\
21	32.3768014495688\\
20	32.1043935258836\\
19	31.8252033940715\\
18	30.7350827832593\\
17	29.4030852845701\\
16	28.6965865803633\\
15	18.2019571915114\\
14	17.1785061212432\\
13	16.9746569937395\\
12	16.9761081622482\\
11	17.0273902294589\\
10	17.1164359520411\\
9	17.1890446964574\\
8	17.1787082946278\\
7	16.9278269509052\\
6	16.9861907892786\\
5	17.1095794872815\\
4	16.7833704259826\\
3	17.3973953942582\\
2	16.7524475319312\\
1	2.70965451824016\\
}--cycle;

\addplot [color=pink!90!lightgray,solid,forget plot]
  table[row sep=crcr]{%
1	1.5895264110252\\
2	14.1531574184214\\
3	14.2428222208178\\
4	14.8129418489013\\
5	14.8166133007521\\
6	14.5227132248387\\
7	14.7682107217163\\
8	14.8656858415678\\
9	14.8417190510719\\
10	14.6510480550972\\
11	14.8077700293205\\
12	14.8759537270403\\
13	14.9977724509533\\
14	14.8205509703365\\
15	16.0533979779979\\
16	25.5080818586639\\
17	26.4448622740084\\
18	28.2061619250339\\
19	28.8175615212135\\
20	28.783818329212\\
21	29.0385003564394\\
22	29.2319416385886\\
23	29.4006350564982\\
24	29.1850316170327\\
25	29.0855208915172\\
26	29.6800574177978\\
27	29.4350251421715\\
28	29.364886971433\\
29	29.8410149935438\\
30	29.7211786426927\\
};
\addplot [color=pink!90!lightgray,solid,forget plot]
  table[row sep=crcr]{%
1	2.70965451824016\\
2	16.7524475319312\\
3	17.3973953942582\\
4	16.7833704259826\\
5	17.1095794872815\\
6	16.9861907892786\\
7	16.9278269509052\\
8	17.1787082946278\\
9	17.1890446964574\\
10	17.1164359520411\\
11	17.0273902294589\\
12	16.9761081622482\\
13	16.9746569937395\\
14	17.1785061212432\\
15	18.2019571915114\\
16	28.6965865803633\\
17	29.4030852845701\\
18	30.7350827832593\\
19	31.8252033940715\\
20	32.1043935258836\\
21	32.3768014495688\\
22	32.3007404851282\\
23	32.1135275265299\\
24	32.375503073644\\
25	32.3498474479374\\
26	31.9884028613395\\
27	32.3379169116129\\
28	32.193661088952\\
29	32.1951559391823\\
30	32.2951530035691\\
};
\addplot [color=black!10!red,dashed]
  table[row sep=crcr]{%
1	2.14959046463268\\
2	15.4528024751763\\
3	15.820108807538\\
4	15.7981561374419\\
5	15.9630963940168\\
6	15.7544520070587\\
7	15.8480188363108\\
8	16.0221970680978\\
9	16.0153818737647\\
10	15.8837420035692\\
11	15.9175801293897\\
12	15.9260309446443\\
13	15.9862147223464\\
14	15.9995285457899\\
15	17.1276775847546\\
16	27.1023342195136\\
17	27.9239737792893\\
18	29.4706223541466\\
19	30.3213824576425\\
20	30.4441059275478\\
21	30.7076509030041\\
22	30.7663410618584\\
23	30.757081291514\\
24	30.7802673453384\\
25	30.7176841697273\\
26	30.8342301395686\\
27	30.8864710268922\\
28	30.7792740301925\\
29	31.018085466363\\
30	31.0081658231309\\
};
\addlegendentry{PHD};

\addplot[area legend,solid,fill=black!10!green,opacity=1.000000e-01,draw=none,forget plot]
table[row sep=crcr] {%
x	y\\
1	0.0114626867555353\\
2	0.270898707146552\\
3	10.2681652303357\\
4	13.4278382536746\\
5	13.7058082751022\\
6	13.5951777510564\\
7	13.6608600245137\\
8	13.9942603874106\\
9	13.9820792044405\\
10	13.7573970879871\\
11	13.7765253746005\\
12	13.932209505762\\
13	13.8943619750992\\
14	13.9954025958415\\
15	13.8331190875811\\
16	16.7207308629513\\
17	22.519662664845\\
18	26.1082785894765\\
19	27.2707103530114\\
20	27.671981744021\\
21	27.6740076785555\\
22	28.0806705970709\\
23	28.1503054590204\\
24	28.0215203170141\\
25	27.6896045865868\\
26	28.2329935556756\\
27	28.3224229103977\\
28	28.1932249294362\\
29	28.5310058356942\\
30	28.7332018239027\\
30	31.0470620717508\\
29	31.1445747415215\\
28	31.0509218008829\\
27	31.1320473192722\\
26	30.9914435231888\\
25	31.2961227168556\\
24	31.0761887037884\\
23	31.0280332935646\\
22	31.0299469086043\\
21	31.1655240139662\\
20	30.6650665867618\\
19	30.5102179762723\\
18	29.2563815381473\\
17	26.1040046968227\\
16	19.228819179042\\
15	16.073345315805\\
14	15.8713185241996\\
13	15.9990750731475\\
12	15.8302590487457\\
11	15.9870324448759\\
10	16.0351904913186\\
9	15.9938175217941\\
8	15.9712625333263\\
7	16.0153914300226\\
6	15.9627356955675\\
5	15.9930834127184\\
4	15.9260043224172\\
3	13.9932137155575\\
2	0.644513164504806\\
1	0.0265040636463273\\
}--cycle;

\addplot [color=mycolor1,solid,forget plot]
  table[row sep=crcr]{%
1	0.0114626867555353\\
2	0.270898707146552\\
3	10.2681652303357\\
4	13.4278382536746\\
5	13.7058082751022\\
6	13.5951777510564\\
7	13.6608600245137\\
8	13.9942603874106\\
9	13.9820792044405\\
10	13.7573970879871\\
11	13.7765253746005\\
12	13.932209505762\\
13	13.8943619750992\\
14	13.9954025958415\\
15	13.8331190875811\\
16	16.7207308629513\\
17	22.519662664845\\
18	26.1082785894765\\
19	27.2707103530114\\
20	27.671981744021\\
21	27.6740076785555\\
22	28.0806705970709\\
23	28.1503054590204\\
24	28.0215203170141\\
25	27.6896045865868\\
26	28.2329935556756\\
27	28.3224229103977\\
28	28.1932249294362\\
29	28.5310058356942\\
30	28.7332018239027\\
};
\addplot [color=mycolor1,solid,forget plot]
  table[row sep=crcr]{%
1	0.0265040636463273\\
2	0.644513164504806\\
3	13.9932137155575\\
4	15.9260043224172\\
5	15.9930834127184\\
6	15.9627356955675\\
7	16.0153914300226\\
8	15.9712625333263\\
9	15.9938175217941\\
10	16.0351904913186\\
11	15.9870324448759\\
12	15.8302590487457\\
13	15.9990750731475\\
14	15.8713185241996\\
15	16.073345315805\\
16	19.228819179042\\
17	26.1040046968227\\
18	29.2563815381473\\
19	30.5102179762723\\
20	30.6650665867618\\
21	31.1655240139662\\
22	31.0299469086043\\
23	31.0280332935646\\
24	31.0761887037884\\
25	31.2961227168556\\
26	30.9914435231888\\
27	31.1320473192722\\
28	31.0509218008829\\
29	31.1445747415215\\
30	31.0470620717508\\
};
\addplot [color=black!10!green,dashed]
  table[row sep=crcr]{%
1	0.0189833752009313\\
2	0.457705935825679\\
3	12.1306894729466\\
4	14.6769212880459\\
5	14.8494458439103\\
6	14.7789567233119\\
7	14.8381257272682\\
8	14.9827614603685\\
9	14.9879483631173\\
10	14.8962937896528\\
11	14.8817789097382\\
12	14.8812342772539\\
13	14.9467185241233\\
14	14.9333605600205\\
15	14.9532322016931\\
16	17.9747750209966\\
17	24.3118336808338\\
18	27.6823300638119\\
19	28.8904641646419\\
20	29.1685241653914\\
21	29.4197658462609\\
22	29.5553087528376\\
23	29.5891693762925\\
24	29.5488545104013\\
25	29.4928636517212\\
26	29.6122185394322\\
27	29.7272351148349\\
28	29.6220733651596\\
29	29.8377902886079\\
30	29.8901319478268\\
};
\addlegendentry{Panjer};

\addplot[area legend,solid,fill=black!10!blue,opacity=1.000000e-01,draw=none,forget plot]
table[row sep=crcr] {%
x	y\\
1	-0.053208568184496\\
2	14.7981676681208\\
3	14.9377533534909\\
4	15.0348715815096\\
5	14.9993733684858\\
6	14.938567132284\\
7	14.9877157361952\\
8	14.980416059306\\
9	15.0599970327649\\
10	15.0455922611911\\
11	15.0489577957393\\
12	15.0763214712991\\
13	15.0360250487439\\
14	15.1089740152084\\
15	15.6824332017306\\
16	25.5539846542897\\
17	27.1917123673268\\
18	28.2694658049993\\
19	29.0139161467063\\
20	29.4989886018089\\
21	29.4655397961425\\
22	29.6638700705777\\
23	29.816623119987\\
24	29.6759215352858\\
25	29.4798353776146\\
26	29.8564775868662\\
27	29.6660535499363\\
28	29.7810351002737\\
29	29.8643704090037\\
30	29.8900141358936\\
30	31.6763118104557\\
29	31.6257538169762\\
28	31.4845317932839\\
27	31.6952327667658\\
26	31.4169443045145\\
25	31.6986711987713\\
24	31.5746253580504\\
23	31.4425179020909\\
22	31.5983999443853\\
21	31.6858001629449\\
20	31.2902543955454\\
19	31.2942608447824\\
18	30.6350543386893\\
17	29.5902876736811\\
16	28.5030040165655\\
15	17.0977152650305\\
14	16.4776513752209\\
13	16.5050400080992\\
12	16.3979598898953\\
11	16.4665773857246\\
10	16.4934488029794\\
9	16.5522815683276\\
8	16.570060210027\\
7	16.384284677603\\
6	16.4359685514325\\
5	16.6359443073814\\
4	16.6589643096991\\
3	17.3330597770635\\
2	19.3064938769013\\
1	6.43238142306382\\
}--cycle;

\addplot [color=mycolor2,solid,forget plot]
  table[row sep=crcr]{%
1	-0.053208568184496\\
2	14.7981676681208\\
3	14.9377533534909\\
4	15.0348715815096\\
5	14.9993733684858\\
6	14.938567132284\\
7	14.9877157361952\\
8	14.980416059306\\
9	15.0599970327649\\
10	15.0455922611911\\
11	15.0489577957393\\
12	15.0763214712991\\
13	15.0360250487439\\
14	15.1089740152084\\
15	15.6824332017306\\
16	25.5539846542897\\
17	27.1917123673268\\
18	28.2694658049993\\
19	29.0139161467063\\
20	29.4989886018089\\
21	29.4655397961425\\
22	29.6638700705777\\
23	29.816623119987\\
24	29.6759215352858\\
25	29.4798353776146\\
26	29.8564775868662\\
27	29.6660535499363\\
28	29.7810351002737\\
29	29.8643704090037\\
30	29.8900141358936\\
};
\addplot [color=mycolor2,solid,forget plot]
  table[row sep=crcr]{%
1	6.43238142306382\\
2	19.3064938769013\\
3	17.3330597770635\\
4	16.6589643096991\\
5	16.6359443073814\\
6	16.4359685514325\\
7	16.384284677603\\
8	16.570060210027\\
9	16.5522815683276\\
10	16.4934488029794\\
11	16.4665773857246\\
12	16.3979598898953\\
13	16.5050400080992\\
14	16.4776513752209\\
15	17.0977152650305\\
16	28.5030040165655\\
17	29.5902876736811\\
18	30.6350543386893\\
19	31.2942608447824\\
20	31.2902543955454\\
21	31.6858001629449\\
22	31.5983999443853\\
23	31.4425179020909\\
24	31.5746253580504\\
25	31.6986711987713\\
26	31.4169443045145\\
27	31.6952327667658\\
28	31.4845317932839\\
29	31.6257538169762\\
30	31.6763118104557\\
};
\addplot [color=black!10!blue,dashed]
  table[row sep=crcr]{%
1	3.18958642743966\\
2	17.052330772511\\
3	16.1354065652772\\
4	15.8469179456044\\
5	15.8176588379336\\
6	15.6872678418583\\
7	15.6860002068991\\
8	15.7752381346665\\
9	15.8061393005462\\
10	15.7695205320853\\
11	15.757767590732\\
12	15.7371406805972\\
13	15.7705325284215\\
14	15.7933126952147\\
15	16.3900742333806\\
16	27.0284943354276\\
17	28.3910000205039\\
18	29.4522600718443\\
19	30.1540884957444\\
20	30.3946214986771\\
21	30.5756699795437\\
22	30.6311350074815\\
23	30.6295705110389\\
24	30.6252734466681\\
25	30.5892532881929\\
26	30.6367109456904\\
27	30.6806431583511\\
28	30.6327834467788\\
29	30.74506211299\\
30	30.7831629731746\\
};
\addlegendentry{CPHD};

\end{axis}
\end{tikzpicture}%